\documentclass{amsart}

\usepackage[normalem]{ulem}
\usepackage{mathtools}
\usepackage{amsmath}
\usepackage{amssymb}
\usepackage{amsthm}
\usepackage{subcaption}
\usepackage{hyperref}
\usepackage[capitalize,nameinlink,noabbrev,nosort]{cleveref}
\usepackage{graphicx}
\graphicspath{{images/}}
\usepackage{tikz}
\usetikzlibrary{positioning}
\usetikzlibrary{patterns}

\newtheorem{thm}{Theorem}[section]
\newtheorem{prop}[thm]{Proposition}
\newtheorem{lem}[thm]{Lemma}
\newtheorem{cor}[thm]{Corollary}
\newtheorem{rem}[thm]{Remark}
\theoremstyle{definition}
\newtheorem{defn}[thm]{Definition}

\newcommand{\ds}{\displaystyle}
\newcommand{\area}{\mathbf{a}}
\newcommand{\bounce}{\mathbf{b}}
\newcommand{\ab}{\mathbf{ab}}
\newcommand{\abp}{\text{ab}}
\newcommand{\qbinom}[2]{\genfrac[]{0pt}{0}{#1}{#2}_q}

\DeclareMathOperator{\Par}{Par}
\newcommand{\dycks}{\mathcal{D}}
\newcommand{\dycksb}{\mathcal{D}^*}
\newcommand{\cdycks}{\mathcal{C}}
\newcommand{\comp}{\textbf{Comp}}
\renewcommand{\ds}{\displaystyle}
\DeclareMathOperator{\N}{N}
\DeclareMathOperator{\E}{E}
\DeclareMathOperator{\AF}{\mathcal{AF}}
\DeclareMathOperator{\BF}{\mathcal{BF}}
\DeclareMathOperator{\bnc}{bnc}
\DeclareMathOperator{\row}{row}
\newcommand{\CAP}[1]{{\small #1}}

\begin{document}

\title
[An area-bounce exchanging bijection on a subset of Dyck paths]
{An area-bounce exchanging bijection on a large subset of Dyck paths}
\author{Arvind Ayyer}
\address{Arvind Ayyer, Department of Mathematics,
  Indian Institute of Science, Bangalore  560012, India.}
\email{arvind@iisc.ac.in}

\author{Naren Sundaravaradan}
\address{Naren Sundaravaradan, HFN Inc., Bengaluru, India 560076}
\email{nano.naren@gmx.com}

\date{\today}

\begin{abstract}
  It is a longstanding open problem to find a bijection exchanging area and bounce statistics on Dyck paths. We settle this problem for an exponentially large subset of Dyck paths via an explicit bijection. Moreover, we prove that this bijection is natural by showing that it maps what we call bounce-minimal paths to area-minimal paths. 
As a consequence of the proof ideas, we show combinatorially that a path with area $a$ and bounce $b$ exists if and only if a path with area $b$ and bounce $a$ exists.
We finally show that the number of distinct values of the sum of the area and bounce statistics is the number of nonzero coefficients in Johnson's $q$-Bell polynomial.
\end{abstract}

\maketitle

\section{Introduction}

The subject of $(q, t)$-Catalan combinatorics arose in the study of diagonal harmonics by Garsia--Haiman~\cite{garsia-haiman-1996} and has given rise to a rich body of work involving objects counted by Catalan numbers~\cite{garsia-haglund-2002,haglund-2003,haglund-loehr-2005,loehr-2005,haglund-2008,pappe-paul-schilling-2022}. In its simplest formulation, there are two statistics on Dyck paths known as \emph{area} and \emph{bounce} which have a symmetric joint distribution; see \cref{sec:def} for the basic definitions. There are other known statistics which have the same distribution, such as \emph{dinv}~\cite[Corollary~5.6.1]{haglund-2008}, \emph{depth}~\cite[Eq. (2.6)]{pappe-paul-schilling-2022}
and \emph{ddinv}~\cite[Eq. (2.7)]{pappe-paul-schilling-2022}, but we will focus only on area and bounce in this article.

There are very few bijective results in this area. The most well-known is the so-called \emph{zeta map}~\cite[Theorem 3.15]{haglund-2008}, which is a bijection on Dyck paths which takes the pair (area, dinv) to (bounce, area). 
{Lee--Li--Loehr~\cite{lee-li-loehr-2014,lee-li-loehr-2018} have constructed a bijection for certain subsets of Dyck paths which switch the dinv and area statistics.
Han--Lee--Li--Loehr have further considerably extended the work of Lee--Li--Loehr in a series of articles~\cite{han-lee-li-loehr-2020,han-lee-li-loehr-2022}. 
}
In recent work of Pappe--Paul--Schilling~\cite[Theorem 3.19]{pappe-paul-schilling-2022}, natural bijections are constructed showing that area and depth, as well as dinv and ddinv, have symmetric joint distributions.

In this work, we make progress on a longstanding open problem of Haglund's \cite[Open Problem 3.11]{haglund-2008} by exhibiting a {direct} bijection between area and bounce statistics on a nontrivial subset of Dyck paths.
{Since our bijection is close in spirit of the works of Han, Lee, Li and Loehr (HLLL for short), we point out the key differences.}

{
First of all, the bijection of HLLL works on what they call \emph{Dyck partitions}, which are the partitions formed by the cells above the Dyck paths. Therefore, the size of the Dyck path does not matter to them. 
In one sense, this is good because this bijection is uniform for all sizes. They are also able to show a complete bijection for Dyck paths up to size $7$.
Lee--Li--Loehr first show a bijection between
two small subsets of Dycks paths.
Secondly, they conjecture~\cite[Conjecture 6.9]{lee-li-loehr-2018} that all Dyck paths can be decomposed into disjoint sets which are in bijection. 
The other papers of HLLL make progress on this conjecture. In particular, they prove properties that these sets must satisfy. As a result, they are able to extend their bijection for Dyck paths up to size $11$.
}

{
On the other hand, in our work, we construct an explicit bijection which works for every fixed size. 
We have verified that our bijections are different, even after using the zeta map~\cite[Theorem 3.15]{haglund-2008} to compare our results, but 
it is difficult to make a direct comparison between our bijections.
In our bijection, the Dyck path $p_{n,\lambda}$ is mapped to $p_{n,\lambda'}$ (see \eqref{def-pnalpha}), which seems a natural property to have. However, HLLL only
  conjecture that their bijection will extend the bijection
  $\zeta^{-1}(p_{n,\lambda}) \leftrightarrow \zeta^{-1}(p_{n,\lambda'})$.
We do show a lower bound that is exponential in the size, but their techniques do not seem to allow them to give bounds for any fixed size.
}

The plan of this article is as follows.
We first give the basic definitions in \cref{sec:def}.

We present the bijection in \cref{sec:bij}. We first define operators that modify the Dyck path by adding or removing boxes in \cref{sec:ops}.
We then construct the bijection and prove that exponentially many Dyck paths are in bijection in \cref{sec:const}.
We then illustrate how this bijection could be extended in \cref{sec:ext bij}.

We then show that this bijection is natural for another subset of Dyck paths,
namely those with fixed sum of area and bounce statistics. For this subset of Dyck paths, we show that the bijection maps paths with minimal bounce to those with minimal area in \cref{sec:minimal}. 
We first define the operators which increase and decrease the area while keeping this sum fixed in \cref{sec:UD}.
We then study properties of bounce paths of these minimal area and minimal bounce paths, and prove the bijection in \cref{sec:area-bounce-minimal}.
We then show the existence of paths with intermediate values of area and bounce in this subset in some special cases, and explicitly characterize paths with largest and next largest sum of area and bounce, in \cref{sec:intermediate}.

Lastly, we count the number of distinct values of the sum of area and bounce statistics in \cref{sec:sums}. Specifically, we show that this number is intimately related to one of the many $q$-generalizations of the Bell numbers.

We give illustrative examples throughout the article. For every Dyck path, we
write the area and bounce at the bottom left of the figure as values of $a$ and
$b$. 

{The code that implements the bijection along with computed examples are available on the open \texttt{GitHub} repository: 
\begin{quote}
\url{https://github.com/nanonaren/qtcatalan-bijection}. 
\end{quote}
The interested reader can download the \texttt{Haskell} program to compute the bijection from \cref{sec:bij} as shown and verify \cref{thm:Phi_bijection}. The bijection for all Dyck paths is $\AF_n$ is precomputed in a separate file for each $n$, $1 \leq n \leq 15$ in the \texttt{examples} folder. For each Dyck path in these examples, the reader can also verify \cref{thm:area-bounce-minimal}.
}

\section{Definitions}
\label{sec:def}

\begin{defn}
A \emph{Dyck path} of semilength $n$ is a path in $\mathbb{Z}^2$ starting at the origin, taking unit north and east steps, denoted $\N$ and $\E$ respectively, such that the path always stays on or above the diagonal line $x = y$, and ends at $(n, n)$.
\end{defn}

We will write Dyck paths as words in the alphabet $\{\N, \E\}$.
Let $\dycks(n)$ be the set of Dyck paths of semilength $n$.

The bounce path of a Dyck path $\pi$ is constructed as follows: we start
at $(0, 0)$ and take $\N$ steps until we meet a point of $\pi$.
At that
point we take $\E$ steps until we hit the diagonal. Now, we again take $\N$
steps until we meet an $\E$ step of $\pi$ and repeat the process until we reach
$(n, n)$. Let the points we hit on the diagonal be $(0,0) = (b_0,b_0), (b_1,b_1), \dots,
(b_\ell,b_\ell) = (n,n)$. 
Then $b_0,\dots, b_\ell$ are called the \emph{bounce points} of $\pi$ 
and the \emph{bounce} of $\pi$ is $\bounce(\pi) =
\sum_{i=1}^\ell (n-b_i)$.
The \emph{bounce path} of a Dyck path $\pi$ having bounce points $b_0, \dots, b_\ell$ is then the path
\begin{equation}
\label{def-pnalpha}
  p_{n,\alpha} = \N^{\alpha_1} \E^{\alpha_1} \dots \allowbreak \N^{\alpha_{\ell}}\E^{\alpha_{\ell}}.
\end{equation}
Here $\alpha = (\alpha_1, \dots, \alpha_\ell)$ is a \emph{composition} with $\alpha_i = b_i - b_{i-1}$, where we recall that a composition is a tuple of positive integers. 
The sum of the parts of a composition $\alpha$ is called its \emph{size},
denoted $|\alpha|$, and the number of parts is called its \emph{length}, denoted
$\ell(\alpha)$. 

Notice that bounce paths $p_{n,\alpha}$ are Dyck paths, and the  bounce path of $p_{n,\alpha}$ is also $p_{n,\alpha}$.
Let $\cdycks(n) = \{ p_{n,\alpha} \mid \ell(\alpha) \leq n, |\alpha| = n \}$ be the set of such Dyck paths.

\begin{figure}[h!]
  \centering
    \includegraphics[width=0.22\textwidth]{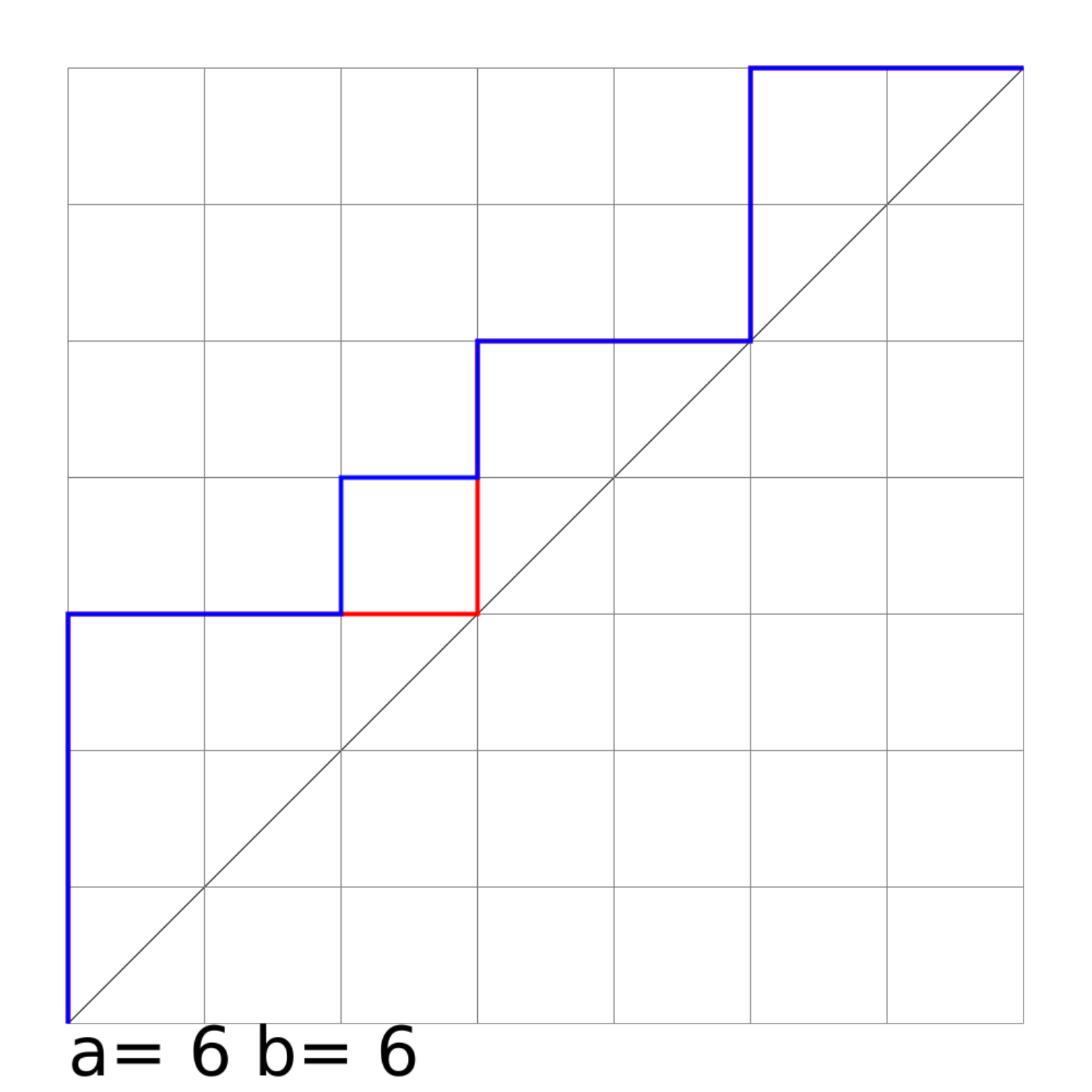}
    \caption{The Dyck path, $\pi = \N^3 \E^2 \N \E \N \E^2 \N^2 \E^2$ has area sequence $a_\pi = (0,1,2,1,1,0,1)$, area $\area(\pi)=6$, bounce path (in red) $p_{7, (3,2,2)}$, bounce points $(0,3,5,7)$ and bounce $\bounce(\pi)=6$.
      This path has $1$ floating cell.}
    \label{fig:properties}
\end{figure}

\begin{prop}[{\cite{haglund-2003}}]
The number of Dyck paths in $\dycks(n)$ with bounce path $p_{n,\alpha}$  is given by
\begin{equation}
  \prod_{j=2}^{\ell(\alpha)}\binom{\alpha_{j-1} + \alpha_j - 1}{\alpha_j}.
\end{equation}
\end{prop}

Note that every path lies above its bounce path.
For a path $\pi$, we call the cells that lie between $\pi$ and its bounce path \emph{floating cells}.
  The \emph{area sequence} $a_\pi = (a_1, \dots, a_n)$ of a Dyck path $\pi$ counts the
  number of entire unit cells between $\pi$ and the diagonal in each row from $1$ to
  $n$ reading bottom to top. The \emph{area} of $\pi$ is then given by 
  $\area(\pi) = \sum_{i=1}^n a_i$.
\cref{fig:properties} gives an example of a Dyck path and its statistics. It is easy to see that the area and bounce of $p_{n,\alpha}$ are given by
\begin{equation}
  \label{eq:ab_wrt_composition}
  \area(p_{n,\alpha}) = \sum_{i=1}^{\ell(\alpha)} \binom{\alpha_i}{2}
    \quad \text{and} \quad 
  \bounce(p_{n,\alpha}) = \sum_{i=1}^{\ell(\alpha)} (i-1)\alpha_i.
\end{equation}
Recall that a \emph{partition} is a composition whose parts are weakly decreasing.
  Let $\Par_n$ be the set of \emph{partitions} of size $n$.
If $\lambda$ is a partition, then $\bounce(p_{n,\lambda})$ is an important statistic in combinatorics, commonly denoted $n(\lambda)$ in the literature and sometimes called the \emph{weighted size}. Recall that the conjugate of a partition $\lambda$, denoted $\lambda'$ is the partition obtained by transposing the Young diagram of $\lambda$. Then one can show that $\area(p_{n,\lambda}) =  n(\lambda')$; see \cite[Equations (1.5) and (1.6)]{macdonald-1995} for example.

Let $F_n(q,t)$ be the \emph{$(q,t)$-Catalan polynomial} given by
\begin{equation}
\label{def-Fn}
  F_n(q,t) = \sum_{\pi \in \dycks(n)} q^{\area(\pi)}t^{\bounce(\pi)}.
\end{equation}
The open problem~\cite[Open Problem 3.11]{haglund-2008} is to show bijectively that $F_n(q, t) = F_n(t, q)$. We will prove such a statement for a subset of $\dycks(n)$ in \cref{sec:bij}.

Let $P_n(a,b) \subseteq \dycks(n)$ be the subset of paths having area $a$ and bounce sequence $b = (b_0, b_1, \dots)$. Let $\sim$ be an equivalence relation on $\dycks(n)$ such that $\pi
\sim \tau$ if and only if $\area(\pi) = \area(\tau)$ and {$\bounce(\pi) = \bounce(\tau)$}. Let
$[\pi]$ represent the paths equivalent to $\pi$ under this relation.

\section{The bijection}
\label{sec:bij}

\subsection{The operators $A_i$, $C_i$, $S_i$ and $B_{i,k}$}
\label{sec:ops}

\begin{defn}
  Let $\dycksb(n) = \dycks(n) \cup \{\bot\}$ be the set of Dyck paths of semilength $n$ adjoined with a symbol $\bot$.
\end{defn}

We define several operators on $\dycksb(n)$. For the operators $A_i$ and $C_i$, $i$ will run from $1$ through $n$.
For the operators $S_i$ and $B_{i,k}$, $i$ will be indexed by the indices of the bounce points of the path.
All these operations are applied left to right and all of them applied on $\bot$ result in $\bot$.

The operator $A_i : \dycksb(n) \rightarrow \dycksb(n)$ adds a cell to the left
of the path at row $i$ {; i.e., if $(a_1,\dots,a_n)$ is the area
  sequence of a path $\pi$, then $(a_1, \dots, a_i+1,a_{i+1},\dots,a_n)$ is the
  area sequence of $A_i(\pi)$.}. Similarly, $A_i^{-1} : \dycksb(n) \rightarrow
\dycksb(n)$ removes the leftmost cell from the path at row $i$. If either of
these operations results in a path which is not a Dyck path, it returns $\bot$.
The figure below shows some examples. Note that $\pi \cdot A_5 = \bot$.

  \begin{center}
    \begin{tabular*}{0.75\textwidth}{c c c}
      \includegraphics[width=0.2\textwidth]{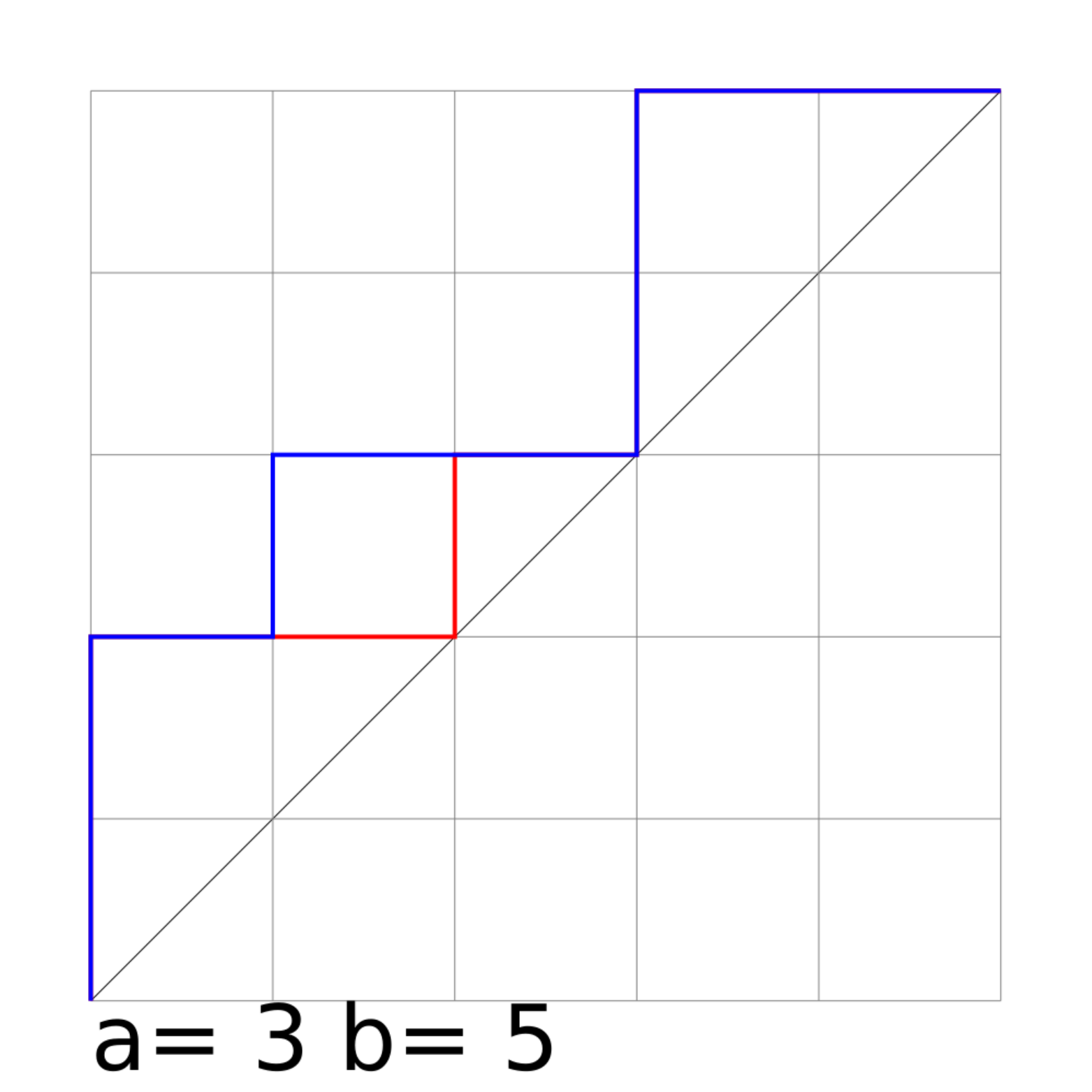} &
      \includegraphics[width=0.2\textwidth]{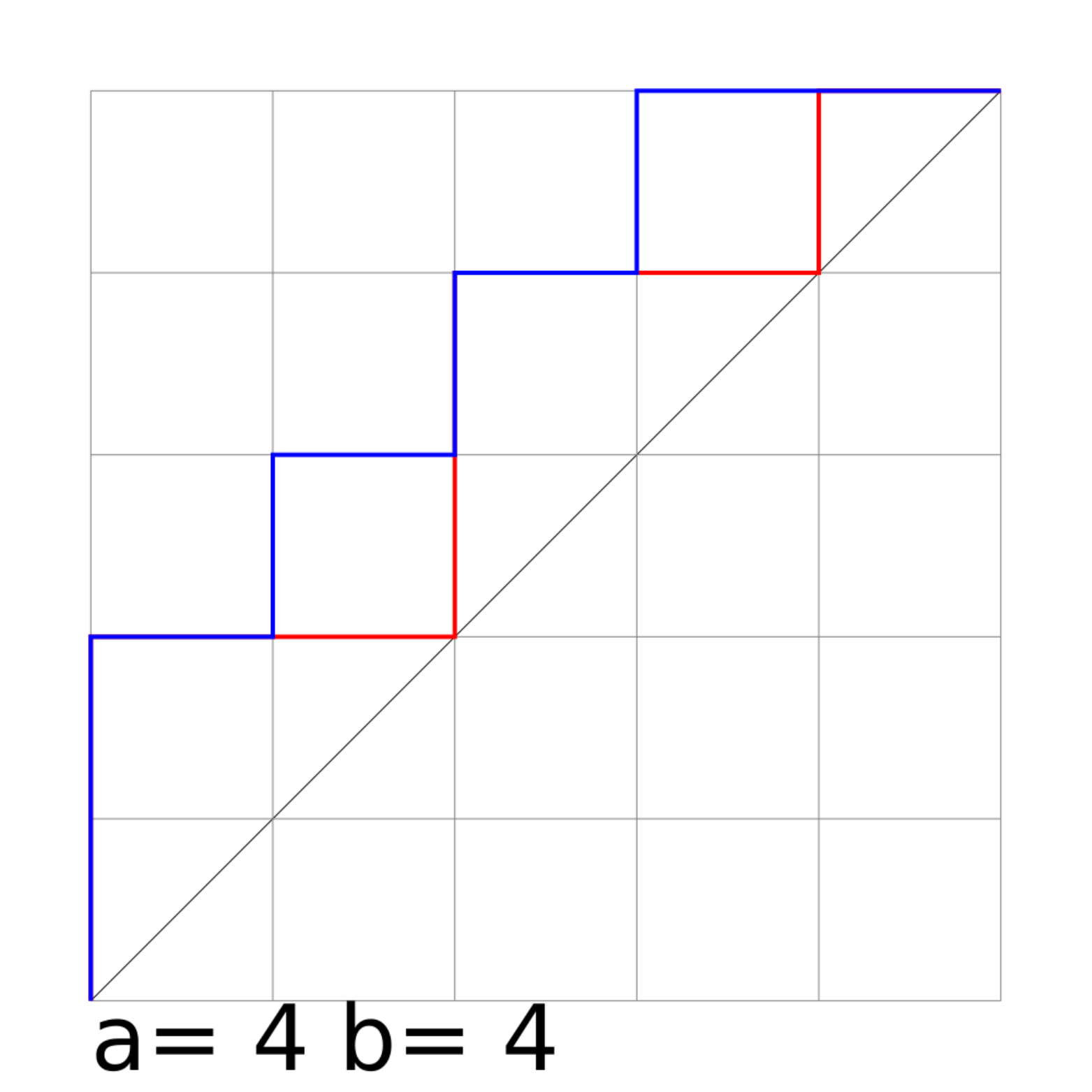} &
      \includegraphics[width=0.2\textwidth]{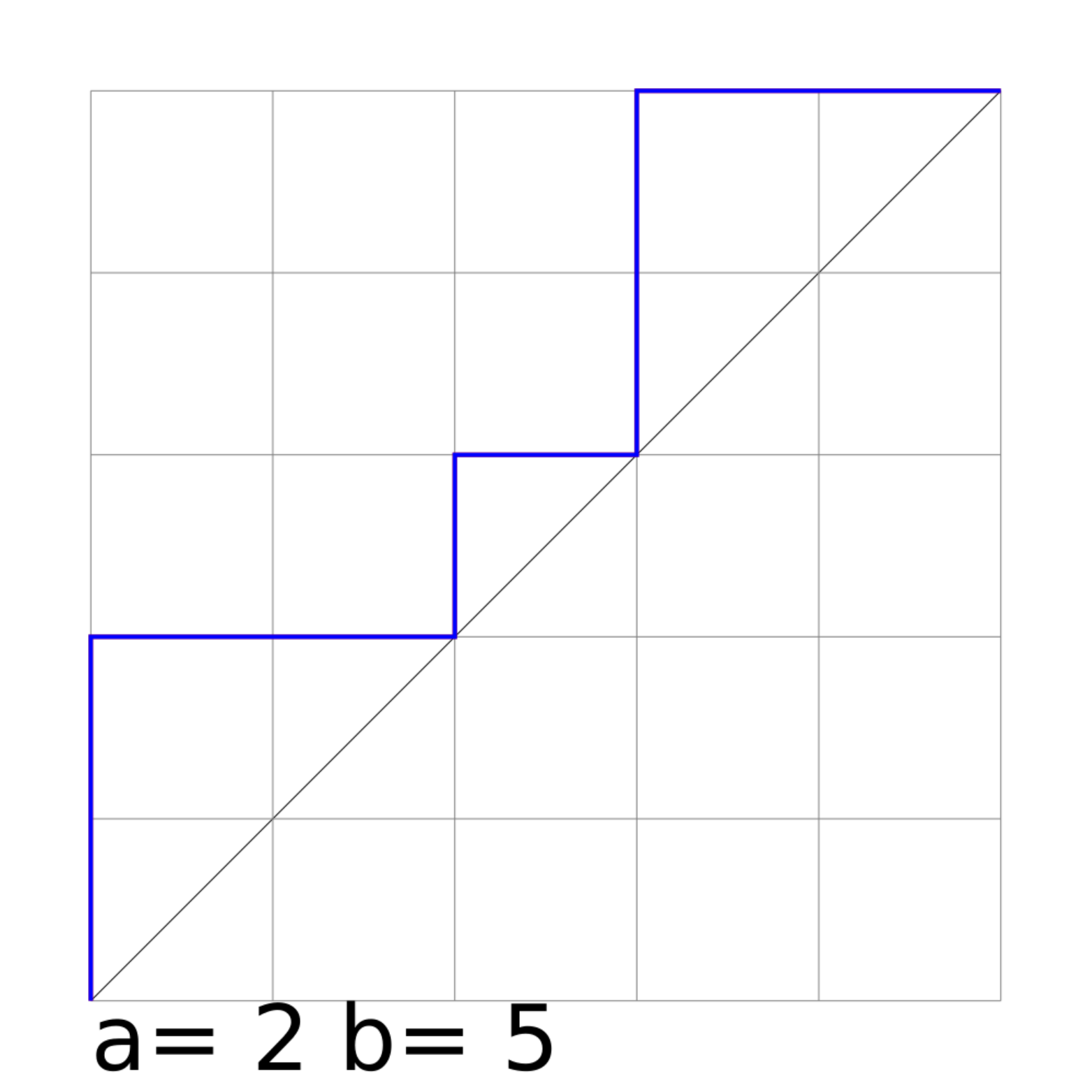} \\
      \CAP{$\pi$} & \CAP{$\pi \cdot A_4$} & \CAP{$\pi \cdot A_3^{-1}$}
    \end{tabular*}
  \end{center}

The operator $C_i : \dycksb(n) \rightarrow \dycksb(n)$ adds a cell to the top of
the path at column $i$. Similarly, $C_i^{-1} : \dycksb(n) \rightarrow
\dycksb(n)$ removes the topmost cell from the path at column $i$. If either of
these operations results in a path which is not a Dyck path, it returns $\bot$.
The figure below shows some examples. Note that $\pi \cdot C_3 = \bot$.

  \begin{center}
    \begin{tabular*}{0.75\textwidth}{c c c}
      \includegraphics[width=0.2\textwidth]{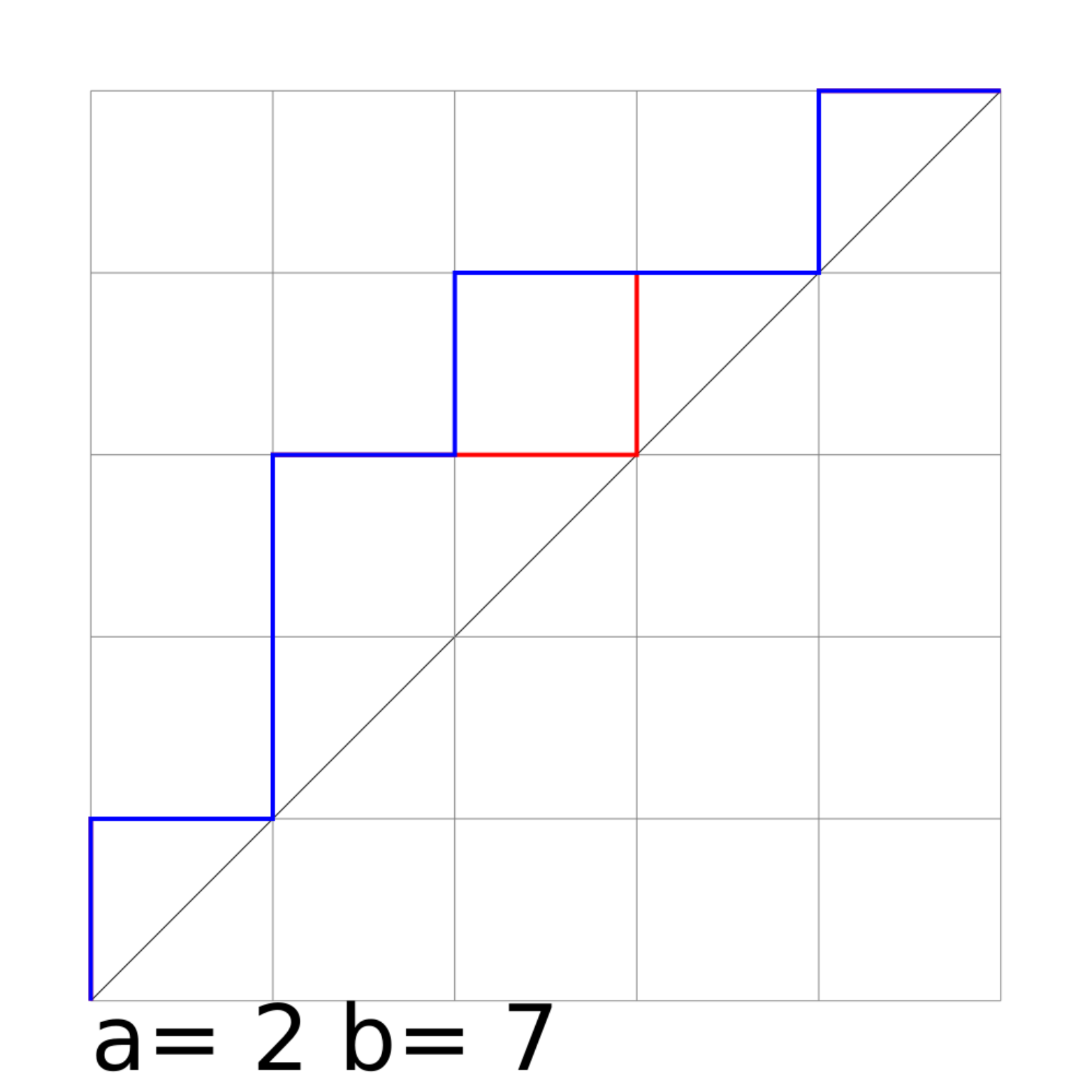} &
      \includegraphics[width=0.2\textwidth]{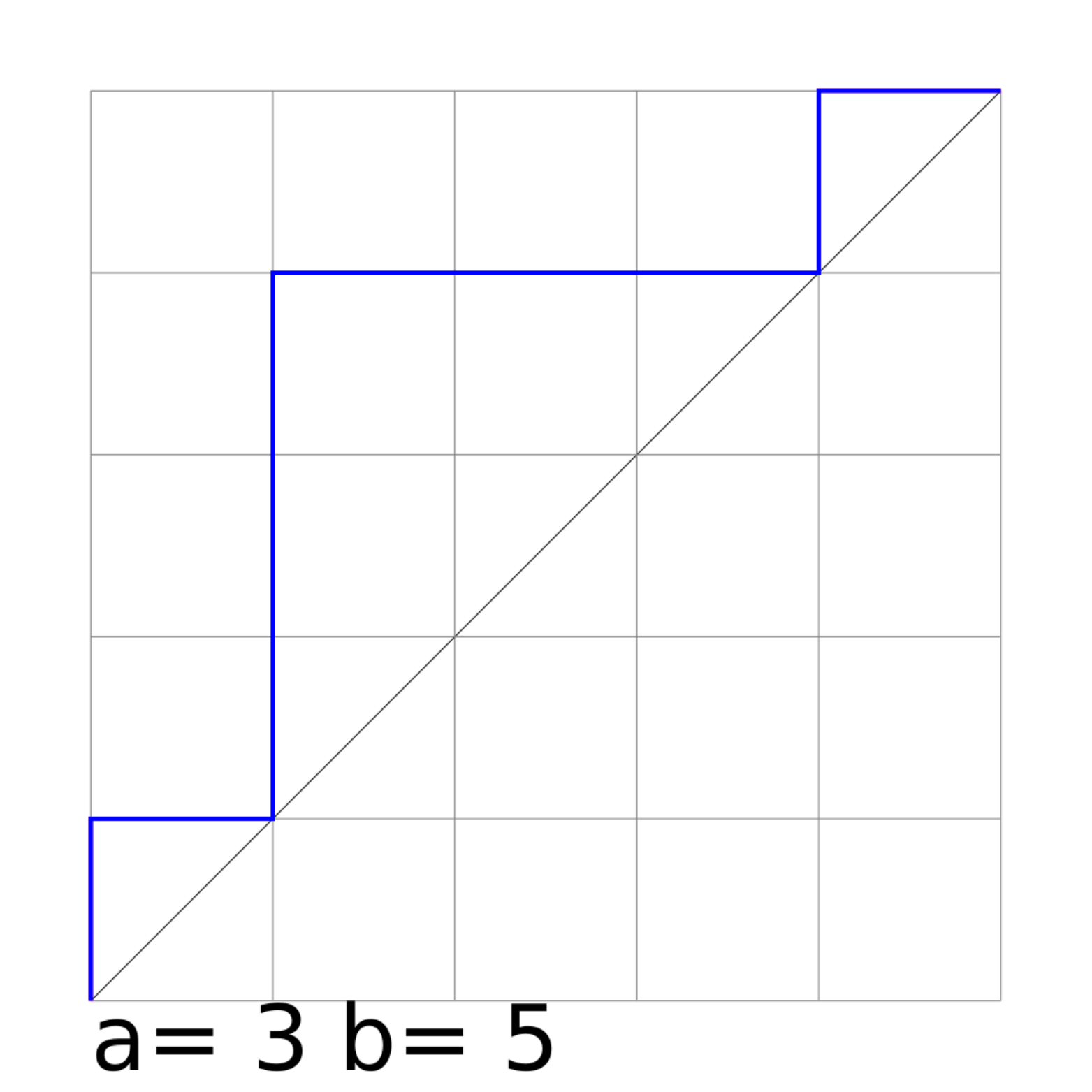} &
      \includegraphics[width=0.2\textwidth]{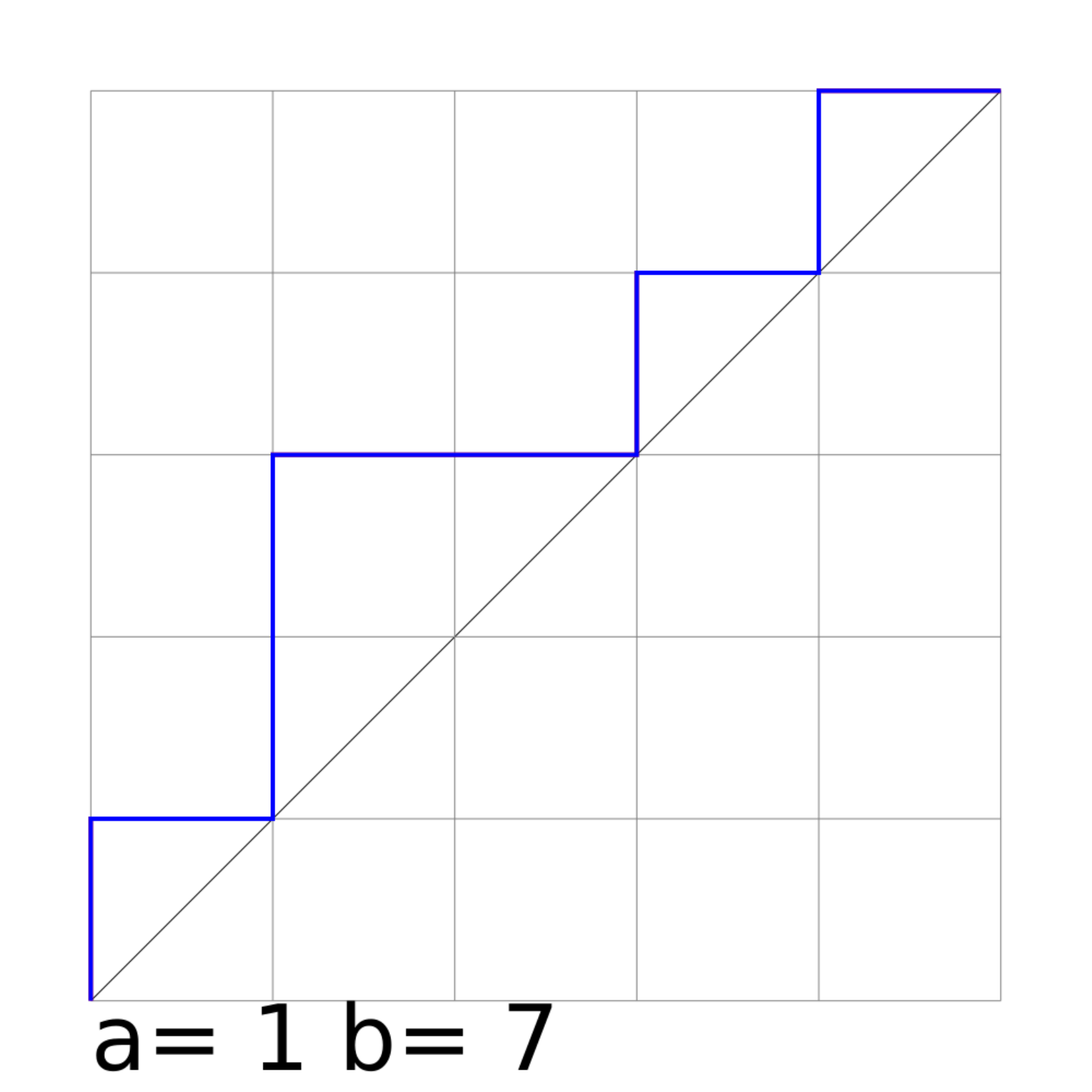} \\
      \CAP{$\pi$} & \CAP{$\pi \cdot C_2$} & \CAP{$\pi \cdot C_3^{-1}$}
    \end{tabular*}
  \end{center}

  For a Dyck path $\pi \in \dycks(n)$, let the height of column $i$, $h_i \equiv
  h_i(\pi)$,
  be the number of cells between $\pi$ and the $x$-axis in column $i$ for $1 \leq i \leq n$. 
  We now define the map $S_i$ for $i \geq 1$ as follows. {The map $S_i$ will
  try to reversibly shift the $i$'th bounce point of a path $\pi$ down by $1$ while preserving
  the area of the path. 
  To be precise, let $s^{(i)} = b_{i+1} - h_{b_i}$.
  The operation will achieve this by removing $s^{(i)}$ cells
  in a row and adding them to a column.} 
  Then
  \begin{equation}\label{eq:S}
    \pi \cdot S_i =
    \begin{cases}
      \bot &\text{if } h_{b_{i-1}} = b_i,\\
      \pi \cdot A_{b_i}^{-s^{(i)}} C_{b_i}^{s^{(i)}} &\text{ otherwise.}
    \end{cases}
  \end{equation}
  Note that $\pi \cdot S_i = \bot$ if $s^{(i)} > a_{b_i}$ or if $h_{b_{i-1}+j} >
  b_i$ for any $1 \le j \le s^{(i)}$. \cref{fig:S} illustrates the operation and
  the following figure shows some examples.
  
  \begin{center}
    \begin{tabular*}{0.75\textwidth}{c c c}
      \includegraphics[width=0.2\textwidth]{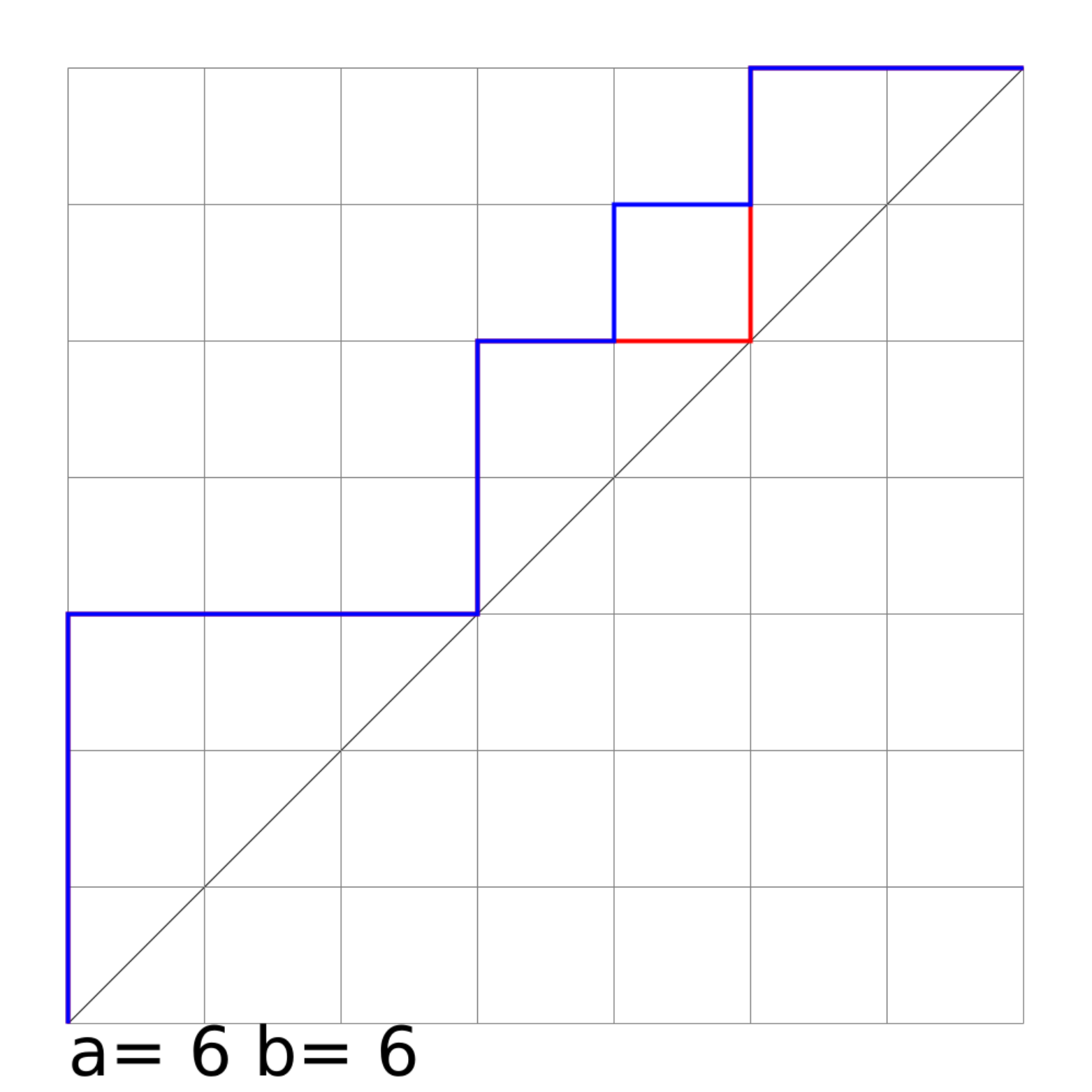} &
      \includegraphics[width=0.2\textwidth]{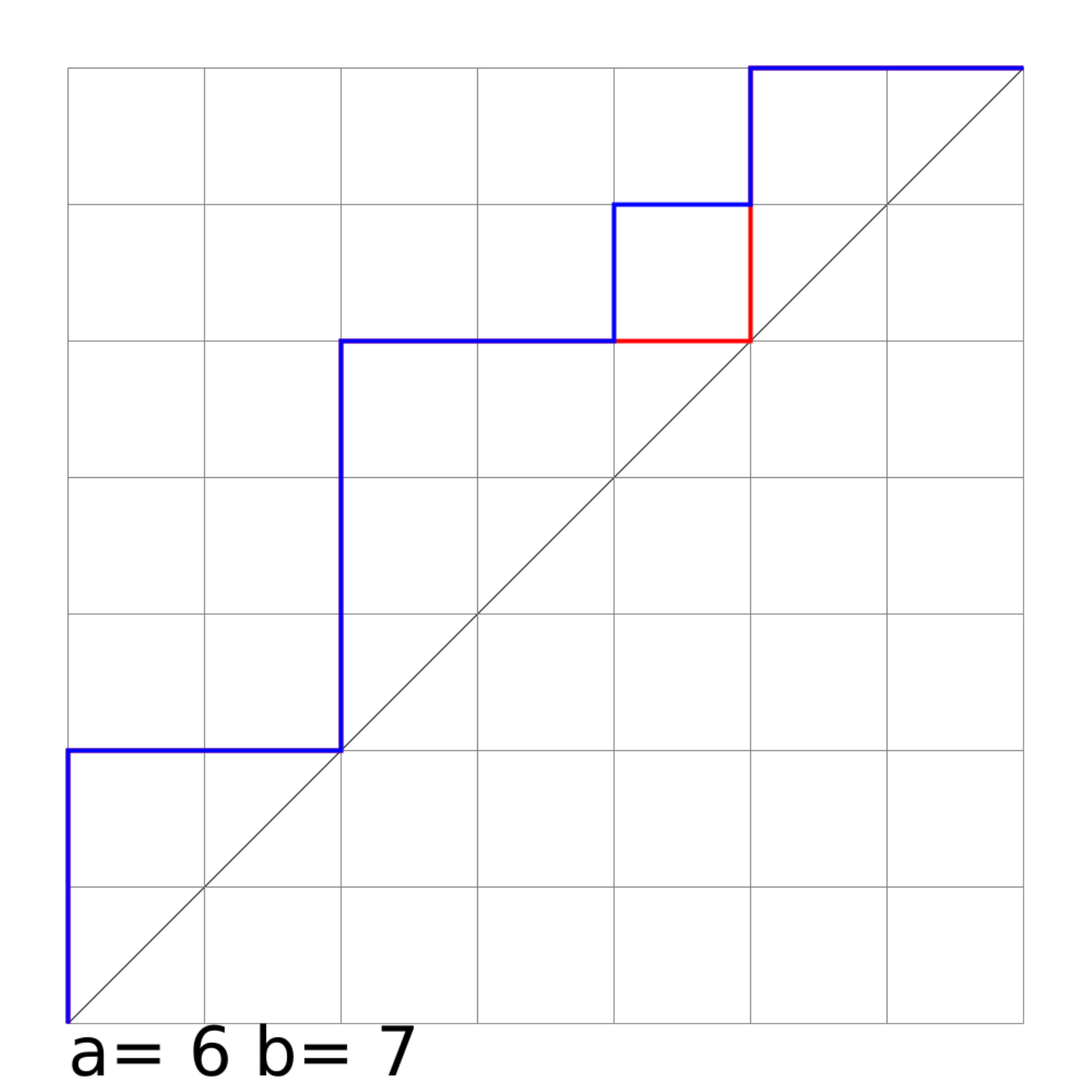} &
      \includegraphics[width=0.2\textwidth]{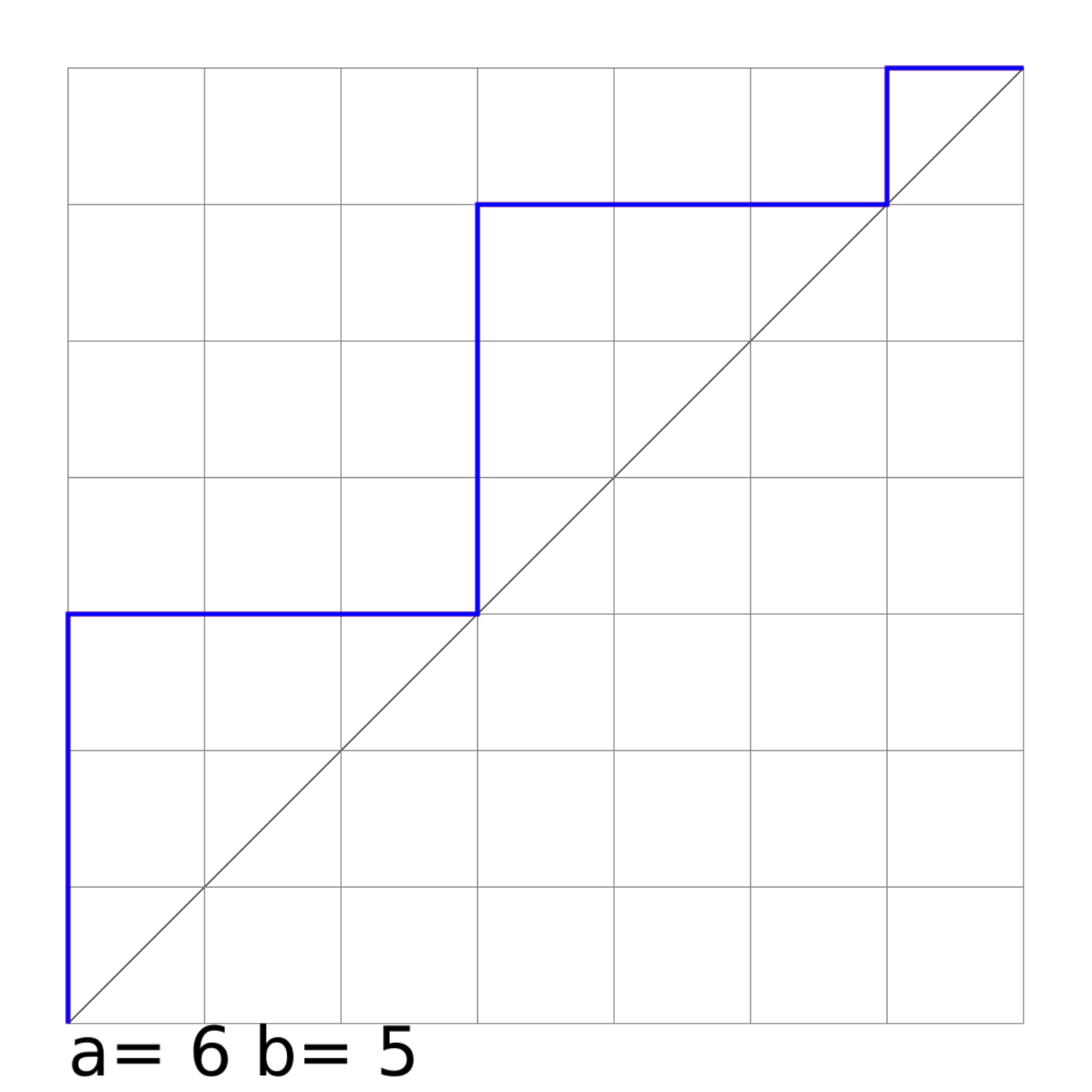} \\
      \CAP{$\pi$} & \CAP{$\pi \cdot S_1$} & \CAP{$\pi \cdot S_2^{-1}$}
    \end{tabular*}
  \end{center}

  It is not difficult to see that whenever $\tau = \pi \cdot S_i \ne \bot$, the map $S_i$ can be inverted. Specifically, $s^{(i)}$ can be recovered from $\tau$ as the number of east steps starting from $(b_{i-1},b_i-1)$ till the first north step.
  The following result is then immediate.

  \begin{figure}[h!]
    \centering
  \begin{tikzpicture}[scale=0.6]
      \draw[gray,very thin] (0.5,0.5) -- (8,8);

      \draw[dashed] (0.5,0.5) -- (0.5,4) -- (4,4) -- (4,8) -- (8,8);

      \draw[color=white,fill=red,opacity=0.3] (0.5,3.5) -- (0.5,4) -- (2,4) -- (2,3.5);

      \draw[color=white,fill=blue,opacity=0.3] (3.5,6.5) -- (3.5,8) -- (4,8) -- (4,6.5);

      \draw[dashed,color=blue] (0.5,3.5) -- (3.5,3.5) -- (3.5,8);

      \draw (0.5,3) -- (0.5,4) -- (2,4) -- (2.5,4);
      \draw (2.5,4) -- (2.5,6) -- (3,6) -- (3,6.5) -- (4,6.5) -- (4,8) -- (5,8);

      \node at (4.7,4) {$b_i$};
      \node at (4.6,3.4) {$b_i-1$};
      \node at (1.5,0.5) {$b_{i-1}$};
      \node at (8.7,8) {$b_{i+1}$};
      \draw[<->,thin] (4.2,6.5) -- (4.2,8) node[midway,xshift=9] {$s^{(i)}$};
      \draw[<->,thin] (0.5,3.3) -- (2,3.3) node[midway,yshift=-7]{$s^{(i)}$};
  \end{tikzpicture}
  \caption{Illustration of the action of $S_i$.}
  \label{fig:S}
\end{figure}
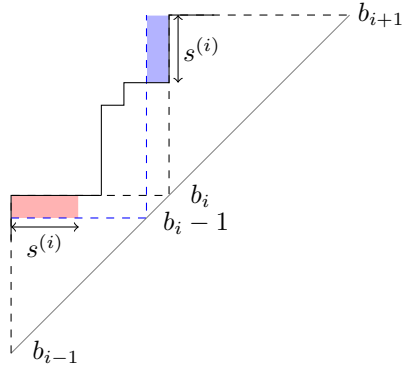

\begin{prop}
\label{prop:Si action}
If $\pi \cdot  S_i \neq \bot$, then $\area(\pi \cdot  S_i) = \area(\pi)$ and $\bounce(\pi \cdot  S_i) = \bounce(\pi) + 1$.
\end{prop}

\begin{proof}
Let $b_1,\dots,b_m$ be the bounce points of $\pi$.
  It is clear that the first $i-1$ bounce points of $\pi$ are also the first $i-1$
  bounce points of $\pi \cdot S_i$. To determine the $i$'th bounce point, start
  at $(b_{i-1},b_{i-1})$ and head north. The first east step occurs at
  $(b_{i-1},b_i-1)$ due to the removed cells in row $b_i$ (see \cref{fig:S}).
  So the $i$'th bounce point is $b_i-1$. Now, heading north from $(b_i-1,b_i-1)$,
  the first east step occurs at $(b_i-1,b_{i+1})$ since the height of column
  $b_i$ in $\pi \cdot S_i$ will be exactly $b_{i+1}$. Hence, the rest of the bounce
  points are $b_{i+1},\dots,b_m$ and $\bounce(\pi \cdot S_i) = \bounce(\pi)+1$.
  It is clear from \eqref{eq:S} that $\area(\pi \cdot S_i) = \area(\pi)$.
\end{proof}

\begin{figure}[h!]
  \captionsetup[subfigure]{labelformat=empty}
  \begin{tabular}{ccc}
  \begin{subfigure}{0.3\textwidth}
  \begin{tikzpicture}[scale=0.5]
    \draw[gray,very thin] (0,0) -- (7,7);

    \draw[dashed,color=blue] (0,0) -- (0,4) -- (4,4) -- (4,6) -- (6,6) -- (6,7);

    \draw (-1,1.5) -- (0,1.5) -- (0,4) -- (4,4) -- (4,6) -- (4.5,6) -- (4.5,7);

    \draw (0,2) circle (2pt) node[left=3pt] {$X$};
    \draw (4,6) circle (2pt) node[left=3pt] {$Y$};

    \node at (1,0) {$b_{i-1}$};
    \node at (4.8,4) {$b_i$};
    \node at (7,6) {$b_{i+1}$};

    \node at (0.6,3.5) {\tiny$\N^{\delta}$};
    \node at (0.6,2) {\tiny $\N^{k}$};
    \node at (2,4.3) {\tiny $\E^{\alpha_i}$};
    \node at (3.2,5.3) {\tiny $\N^{\alpha_{i+1}}$};
  \end{tikzpicture}
  \caption{$\pi$}
  \end{subfigure}
  &
  \begin{subfigure}{0.3\textwidth}
    \begin{tikzpicture}[scale=0.5]
      \draw[gray,very thin] (0,0) -- (7,7);

      \draw[dashed,color=blue] (0,0) -- (0,3.5) -- (3.5,3.5) -- (3.5,6) -- (6,6) -- (6,7);

      \draw (-1,1.5) -- (0,1.5) -- (0,3.5) -- (2,3.5) -- (2,4) -- (3.5,4) -- (3.5,6) -- (4.5,6) -- (4.5,7);

      \draw (0,2) circle (2pt) node[left=3pt] {$X$};
      \draw (4,6) circle (2pt) node[above=3pt] {$Y$};

      \node at (1,0) {$b_{i-1}$};
      \node at (5,3.5) {$b_i-1$};
      \node at (7,6) {$b_{i+1}$};

      \node at (0.9,3) {\tiny $\N^{\delta-1}$};
      \node at (0.6,2) {\tiny $\N^{k}$};
      \node at (1,3.8) {\tiny $\E^{\alpha_{i+1}}$};
      \node at (1,3.8) {\tiny $\E^{\alpha_{i+1}}$};
      \node at (2.6,5) {\tiny $\N^{\alpha_{i+1}}$};
      \node at (3.7, 5.7) {\tiny$\E$};
    \end{tikzpicture}
    \caption{$\pi \cdot S_i$}
  \end{subfigure}
  &
  \begin{subfigure}{0.3\textwidth}
    \begin{tikzpicture}[scale=0.5]
      \draw[gray,very thin] (0,0) -- (7,7);

      \draw[dashed,color=blue] (0,0) -- (0,2) -- (2,2) -- (2,6) -- (6,6) -- (6,7);

      \draw (-1,1.5) -- (0,1.5) -- (0,2) -- (2,2) -- (2,6) -- (3.5,6) -- (4.5,6) -- (4.5,7);

      \draw (0,2) circle (2pt) node[above=3pt] {$X$};
      \draw (4,6) circle (2pt) node[above=3pt] {$Y$};

      \node at (1,0) {$b_{i-1}$};
      \node at (3.3,2) {$b_i-\delta$};
      \node at (7,6) {$b_{i+1}$};

      \node at (-0.5,2) {\tiny $\N^k$};
      \node at (1.1,2.3) {\tiny $\E^{\alpha_{i+1}}$};
      \node at (1.5,4.5) {\tiny $\N^{\alpha_i}$};
      \node at (2.75, 5.7) {\tiny $\E^\delta$};
    \end{tikzpicture}
    \caption{$\pi' = \pi \cdot S_i^{\delta}$}
  \end{subfigure}
  \end{tabular}

  \caption{An illustration of \cref{prop:si_not_bottom} where $\alpha_i > \alpha_{i+1}$, $k \ge 0$ and $\delta = \alpha_i - \alpha_{i+1}$.}
  \label{fig:si_not_bottom}
\end{figure}
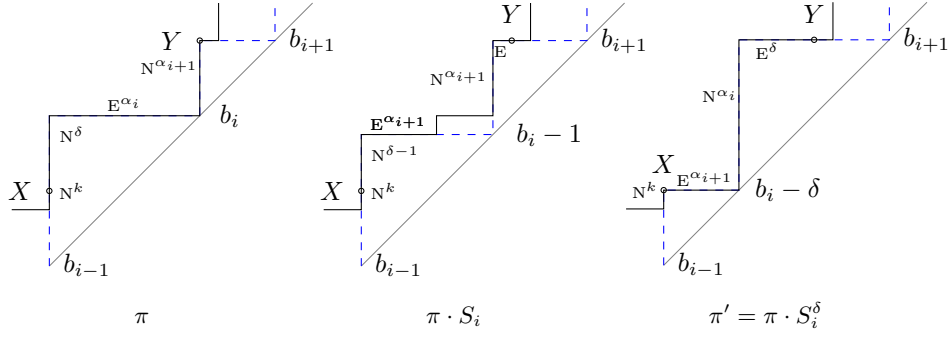

We now show two results that list the conditions under which $S_i$ can be applied
repeatedly and also the conditions under which this is reversible.

\begin{prop}
\label{prop:si_not_bottom}
Suppose $\pi \in \dycks(n)$ has bounce path $p_{n,\alpha}$ such that $\alpha_i > \alpha_{i+1}$ for some $1 \le i < \ell(\alpha)$. Let $\delta = \alpha_i - \alpha_{i+1}$.
Further, suppose
$\pi$ reaches the point ${X} = (b_{i-1}, b_{i-1} + \alpha_{i+1})$ and afterwards
takes steps $\N^\delta \E^{\alpha_i} \N^{\alpha_{i+1}}$ reaching ${Y} = (b_i, b_{i+1})$.
Then, $\pi \cdot S_i^p \ne \bot$ for all $1 \le p \leq \delta$. Moreover, $\pi' = \pi \cdot S_i^\delta$ modifies $\pi$ to $\E^{\alpha_{i+1}} \N^{\alpha_i}\E^\delta$ between the points $X$ and $Y$ and leaves the rest of the path unchanged, and has bounce path $p_{n, (\dots, \alpha_{i+1}, \alpha_i, \dots)}$.
\end{prop}

\begin{proof}
  \cref{fig:si_not_bottom} illustrates the proof ideas.
  First, compute $\pi \cdot S_i$ using \eqref{eq:S}. We have $s^{(i)} = \alpha_{i+1}$ and since $a_{b_i} = \alpha_i > \alpha_{i+1}$, remove $\alpha_{i+1}$ cells from row $b_i$ and add $\alpha_{i+1}$ cells to column $b_i$.
  If $\delta = 1$, $\pi$ between points $X$ and $Y$ is
  $\N \E^{\alpha_i} \N^{\alpha_i-1}$
and the path of $\pi \cdot S_i$ between those points is
$\E^{\alpha_i-1}\N^{\alpha_i} \E$. In this case ($\delta=1$), $\alpha_i$ and $\alpha_{i+1}$ swap.
On the other hand, if $\delta > 1$, the path of $\pi \cdot S_i$
between $X$ and $Y$ is
$\N^{\delta-1} \E^{\alpha_{i+1}} (\N \E^{\delta-1}) \N^{\alpha_{i+1}}\E$,
where the steps in parenthesis form the new floating cells
(see the middle diagram in \cref{fig:si_not_bottom}).
Notice that there are an equal number, namely $\alpha_{i+1}$, east (resp. north) steps before (resp. after) these floating cells.
Therefore, $S_i$ can be applied once again and the resulting path will once again have this property; see \cref{fig:si_not_bottom} for an illustration.
We can repeatedly apply $S_i$ $\delta$ many times in total,
until we obtain $\pi'$ as stated.
\end{proof}

Let $\pi \in \dycks(n)$ with bounce path $p_{n,\alpha}$ and fix $i <
\ell(\alpha)$.
Let {$\Delta_r = \alpha_i - \alpha_r$ for $r > i$}.
For $k$ a positive integer, let $m$ be the minimum value of $j$ such that
$d = k - \sum_{j=1}^m \max\{0,\Delta_{i+j}\}$ and $d \le 0$. Then define
\begin{equation}
  \label{eq:B}
  \pi \cdot B_{i,k} = \pi \cdot S_i^{\max\{0,\Delta_{i+1}\}} S_{i+1}^{\max\{0,\Delta_{i+2}\}} \cdots S_{i+m-1}^{\max\{0,\Delta_{i+m}\}{{+d}}}.
\end{equation}
By \cref{prop:Si action}, $\bounce(\pi \cdot B_{i,k}) = \bounce(\pi) + k$
if $\pi \cdot B_{i,k} \neq \bot$.
\cref{fig:b_reachable} shows an
example.
The following corollary now follows from \cref{prop:si_not_bottom}.

\begin{figure}[h!]
  \captionsetup[subfigure]{labelformat=empty}
  \centering
  \begin{subfigure}[t]{0.24\textwidth}
    \includegraphics[width=\textwidth]{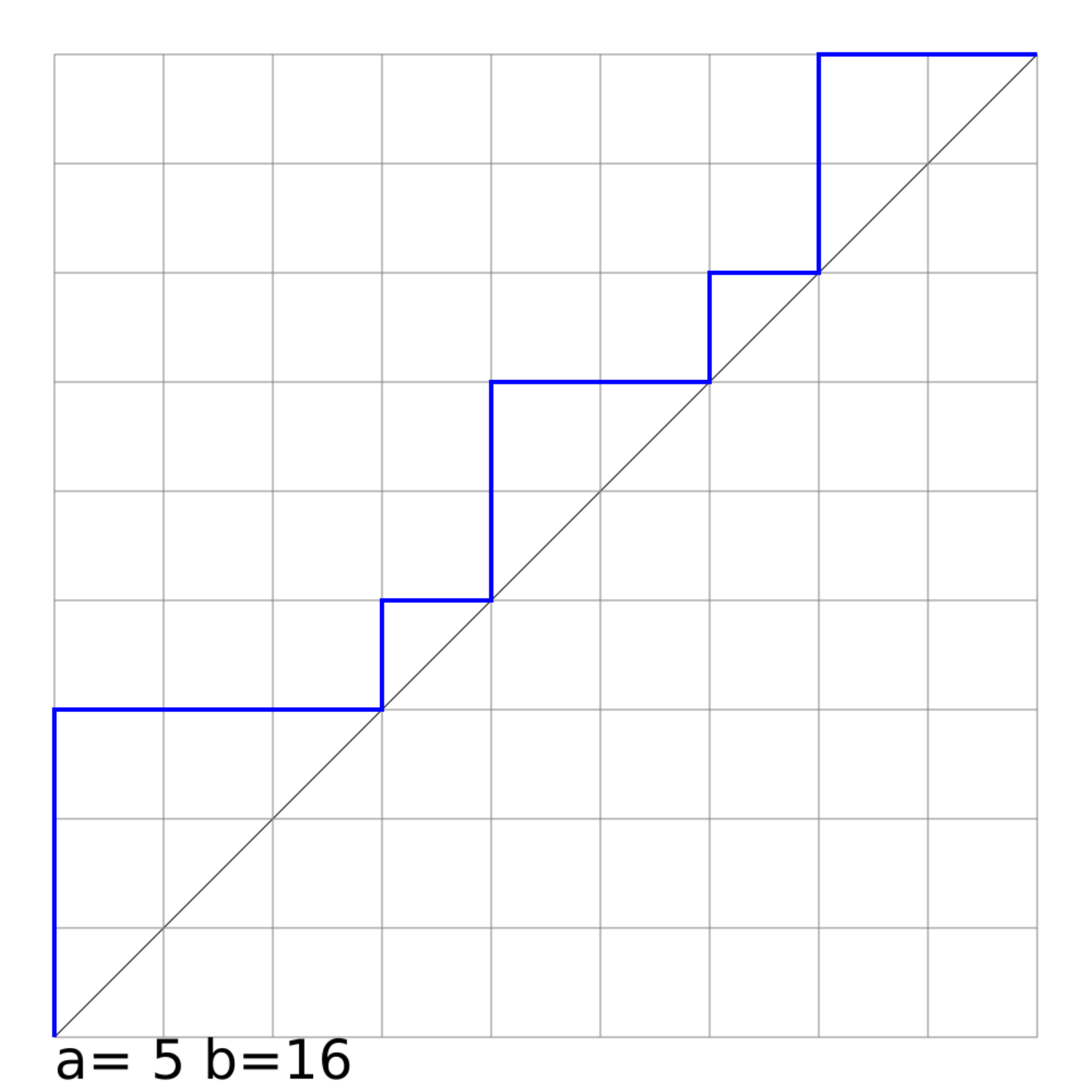}
    \caption{$\pi$}
  \end{subfigure}
  \hspace{10pt}
  \begin{subfigure}[t]{0.24\textwidth}
    \includegraphics[width=\textwidth]{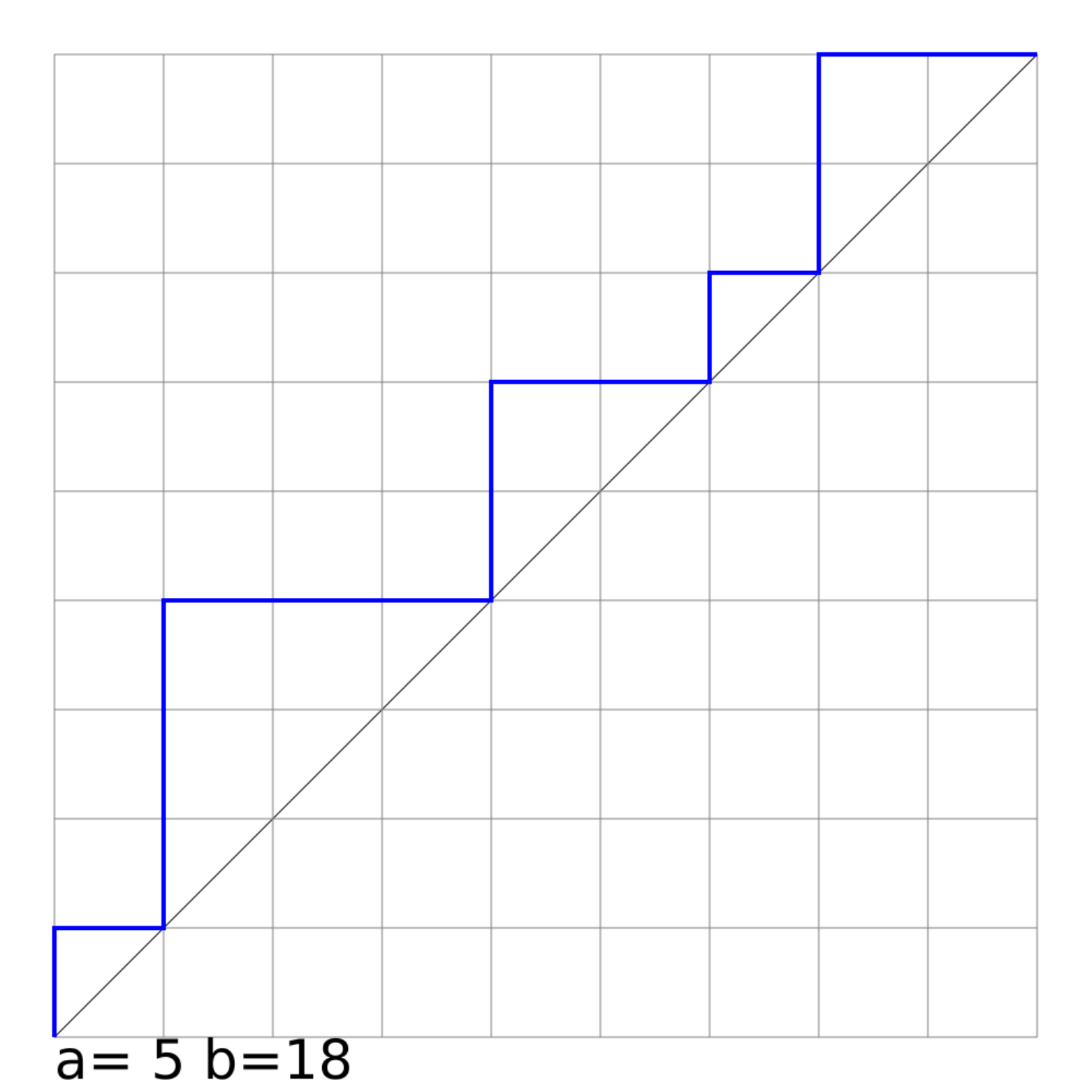}
    \caption{$\pi \cdot B_{1,2}=S_1^2$}
  \end{subfigure}
  \hspace{10pt}
  \begin{subfigure}[t]{0.24\textwidth}
    \includegraphics[width=\textwidth]{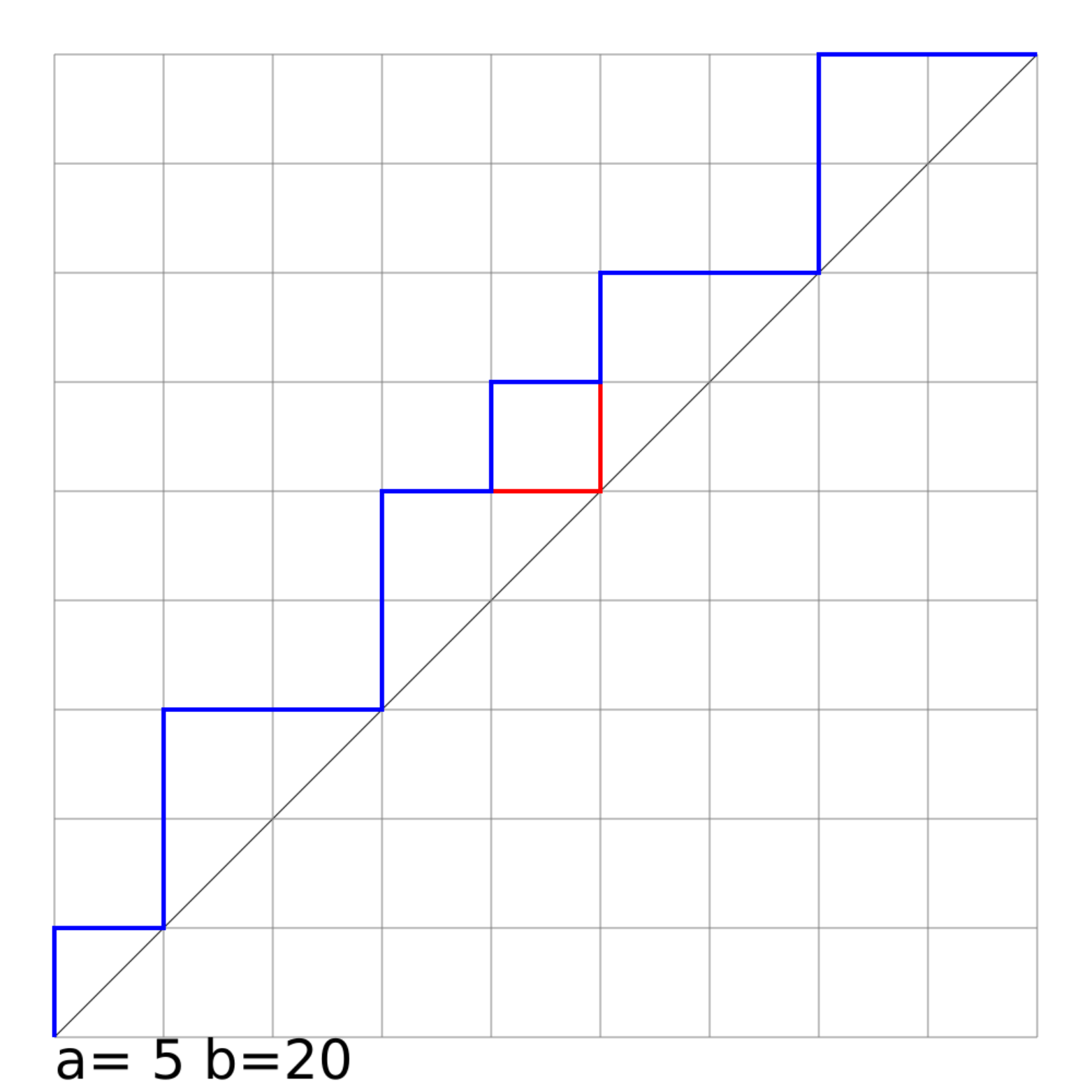}
    \caption{$\pi \cdot B_{1,4}= \pi \cdot S_1^2S_2^1S_3^1$ where $\Delta_2 = 2$, $\Delta_3=1$, and $\Delta_4=2$}
  \end{subfigure}
  \caption{Examples of the action of $B_{i,k}$.}
  \label{fig:b_reachable}
\end{figure}

\begin{cor}
  \label{cor:B_not_bottom}
  Let $\pi \in \dycks(n)$ with bounce path $p_{n,\alpha}$ so that 
  \[
    \pi = \dots (\N^{\alpha_i-\alpha_{i+1}}\E^{\alpha_i})(\N^{\alpha_{i+1}}\E^{\alpha_{i+1}}) \dots (\N^{\alpha_l}\E^{\alpha_l})
  \]
  for $l = \ell(\alpha)$ and some $1 \le i < \ell(\alpha)$.
  Suppose for all $j > i$, $\alpha_i > \alpha_j$.
  Then, $B_{i,k}(\pi) \ne \bot$
  for all $1 \le k \le {\sum_{j=i+1}^\ell} |\alpha_i - \alpha_j|$.
\end{cor}

Recall that the \emph{conjugate} of a partition $\lambda$, denoted $\lambda'$, is a partition such that $\lambda'_i = \max \{j \mid \lambda_j \geq i\}$.
Let $m_\lambda(i)$ be the number of parts in $\lambda$ of size $i$.
We will need to keep track of the number of distinct parts of a partition $\lambda$. So let $\bar{\lambda}$ be the partition with parts in $\lambda$ such that each part occurs once, and $\ell(\bar{\lambda})$ be the number of distinct parts in $\lambda$.
For instance, if $\lambda = (4,3,3,1,1,1)$, $\ell(\bar\lambda) = 3$ and
$m_\lambda(3)=2$.
It is easy to see that $\ell(\bar\lambda) = \ell(\overline{\lambda'})$. In the above example, $\lambda' = (6, 3, 3, 1)$ and $\ell(\overline{\lambda'}) = 3$. 
Note that $\bar{\lambda}_i$ is the $i$'th largest part of $\lambda$.

\subsection{Construction}
\label{sec:const}

We now construct our bijection. We need two sets to index, on the one hand, the locations of the bounce points, and on the other, the rows where boxes can be added to increase the area.

\begin{defn}
Let $\ell = \ell(\bar{\lambda})$ and let $\mathcal{I}_{\bnc}
(\lambda), \mathcal{I}_{\row}(\lambda) \subset \mathbb{N}^2$ be two indexing sets defined by
  \begin{align*}
    \mathcal{I}_{\bnc}(\lambda) &= \{ (i,r) \mid 1 \le i \leq \ell-1, 1 \le r \le m_\lambda(\bar{\lambda}_{\ell-i}) \},\\
    \mathcal{I}_{\row}(\lambda) &= \{ (i,r) \mid 1 \le i \leq \ell-1, 1 \le r \le \bar\lambda_{i+1}\}.
  \end{align*}
We call a function from one of these index sets to $\mathbb{N}$ a 
\emph{count map}.
\end{defn}

\begin{defn}
A \emph{bounce map} maps $(i,r) \in \mathcal{I}_{\bnc}(\lambda)$ to the {index} of a bounce point in $p_{n, \lambda}$,
\begin{equation}
  \bnc_\lambda(i,r) = 1-r + \sum_{j=1}^{\ell-i} m_\lambda(\bar{\lambda}_j).
\end{equation}
Similarly, a \emph{row map} maps $(i,r) \in \mathcal{I}_{\row}(\lambda)$ to a row
\begin{equation}
  \row_\lambda(i,r) = r + \sum_{j=1}^i m_\lambda(\bar{\lambda}_j)\bar{\lambda}_j.
\end{equation}
\end{defn}

\cref{fig:bncRow} shows an example of bounce and row maps. 

\begin{figure}[h!]
  \centering
  \includegraphics[width=0.5\textwidth]{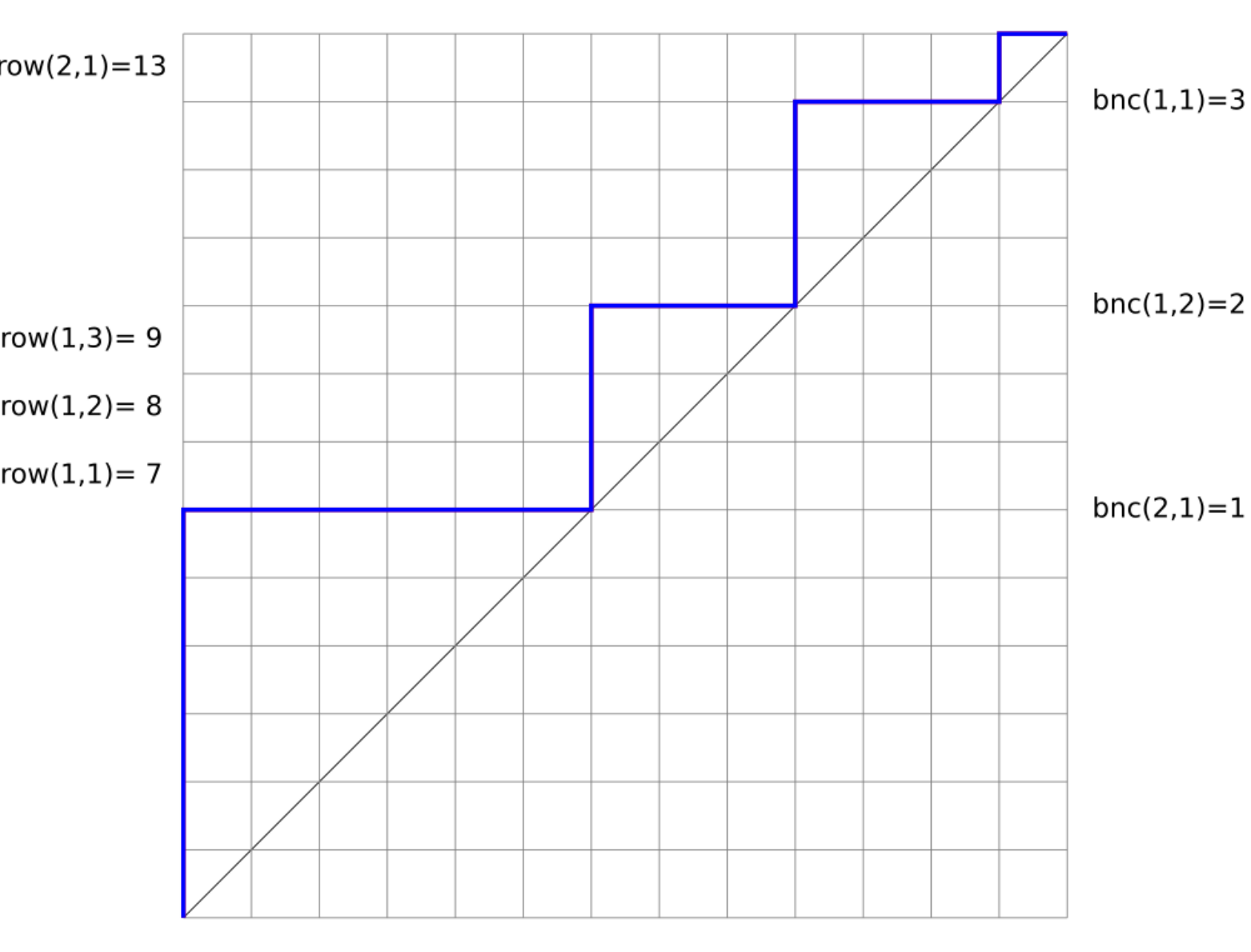}
  \caption{The bounce and row maps annotated on $p_{13, \lambda}$, where
  $\lambda = (6, 3, 3, 1)$. Thus, $\overline{\lambda}_1=6$ and $\overline{\lambda}_2=3$. 
  Note that $\mathcal{I}_{\bnc}(\lambda) = \{(1, 1), (1, 2), (2, 1)\}$
  and $\mathcal{I}_{\row}(\lambda) = \{(1, 1), (1, 2), (1, 3), (2, 1)\}$}.
  \label{fig:bncRow}
\end{figure}

Given a count map $f$ and a path $p_{n, \lambda}$ having partition $\lambda$ of size $n$, define the operators $A^f$ and $B_f$ by
\begin{align*}
  p_{n,\lambda} \cdot A^f & := p_{n, \lambda} \cdot \prod_{(i,r) \in \mathcal{I}_{\row}(\lambda)} A_{\row_\lambda(i,r)}^{f(i,r)}\\
  p_{n,\lambda} \cdot B_f & := p_{n, \lambda} \cdot \prod_{(i,r) \in \mathcal{I}_{\bnc}(\lambda)} B_{\bnc_\lambda(i,r), f(i,r)}
\end{align*}
where the operations are ordered from left to right by $i$ first and then by $r$ in ascending order.
We call a pair of Dyck paths $(\tau,\tau')$ for which
$\area(\tau)=\bounce(\tau')$ and $\bounce(\tau)=\area(\tau')$ an ab-\emph{pair}.
From the discussion after \eqref{eq:ab_wrt_composition}, the following result is
clear.

\begin{prop}
  \label{prop:partition}
  For $\lambda \in \Par_n$,  $(p_{n,\lambda}, p_{n,\lambda'})$ forms an \emph{$\abp$}-pair.
\end{prop}
The following result shows when we can construct a larger set of $\abp$-pairs.

\begin{thm}
  \label{thm:ab_pairs}
  Let $\lambda \in \Par_n$ and $\ell = \ell(\bar{\lambda})$. Suppose $f$ is a count map on $\mathcal{I}_{\bnc}(\lambda')$ which satisfies
  \begin{equation}
    \bar\lambda_{i} > f(i,1) \ge f(i,2) \ge \dots \ge f(i,m_{\lambda'}(\overline{\lambda'}_{\ell-i}))
  \end{equation}
for all $1 \le i \leq \ell-1$ and $p_{n,\lambda'} \cdot B_f \ne \bot$.
Then {$p_{n,\lambda} \cdot A^f \ne \bot$ and} $(p_{n,\lambda} \cdot A^f, p_{n,\lambda'} \cdot B_f)$ is an \emph{$\abp$}-pair.
\end{thm}

\begin{proof}
  First, note that $\mathcal{I}_{\bnc}(\lambda') \subseteq \mathcal{I}_{\row}(\lambda)$ since for all $1 \le i < \ell(\bar\lambda)$,
  \begin{align*}
    \max \{ r \mid (i, r) \in \mathcal{I}_{\bnc}(\lambda') \} =& m_{\lambda'}(\overline{\lambda'}_{\ell-i})
    = \bar\lambda_{i+1} - \bar\lambda_{i+2}\\
    \le & \bar\lambda_{i+1} = \max \{ r \mid (i, r) \in \mathcal{I}_{\row}(\lambda) \}.
  \end{align*}
  Second, the upper-bound condition on $f(i,1)$ ensures that $p_{n,\lambda}
  \cdot A_{\row_\lambda(i,1)}^{f(i,1)} \ne \bot$. And the weakly decreasing
  condition ensures that $p_{n,\lambda} \cdot A^f \ne \bot$. Let $\delta =
  \sum_{i,r} f(i,r)$. Then $p_{n,\lambda} \cdot A^f$ has area
  $\area(p_{n,\lambda}) + \delta$ and bounce $\bounce(p_{n,\lambda})$, while
  {$p_{n,\lambda'} \cdot B_f$} has area $\area(p_{n,\lambda'})=\bounce(p_{n,\lambda})$ and bounce
  $\bounce(p_{n,\lambda'}) + \delta = \area(p_{n,\lambda}) + \delta$.
\end{proof}

\begin{figure}[h!]
  \captionsetup[subfigure]{labelformat=empty}
  \centering
  \begin{subfigure}[t]{0.21\textwidth}
    \includegraphics[width=\textwidth]{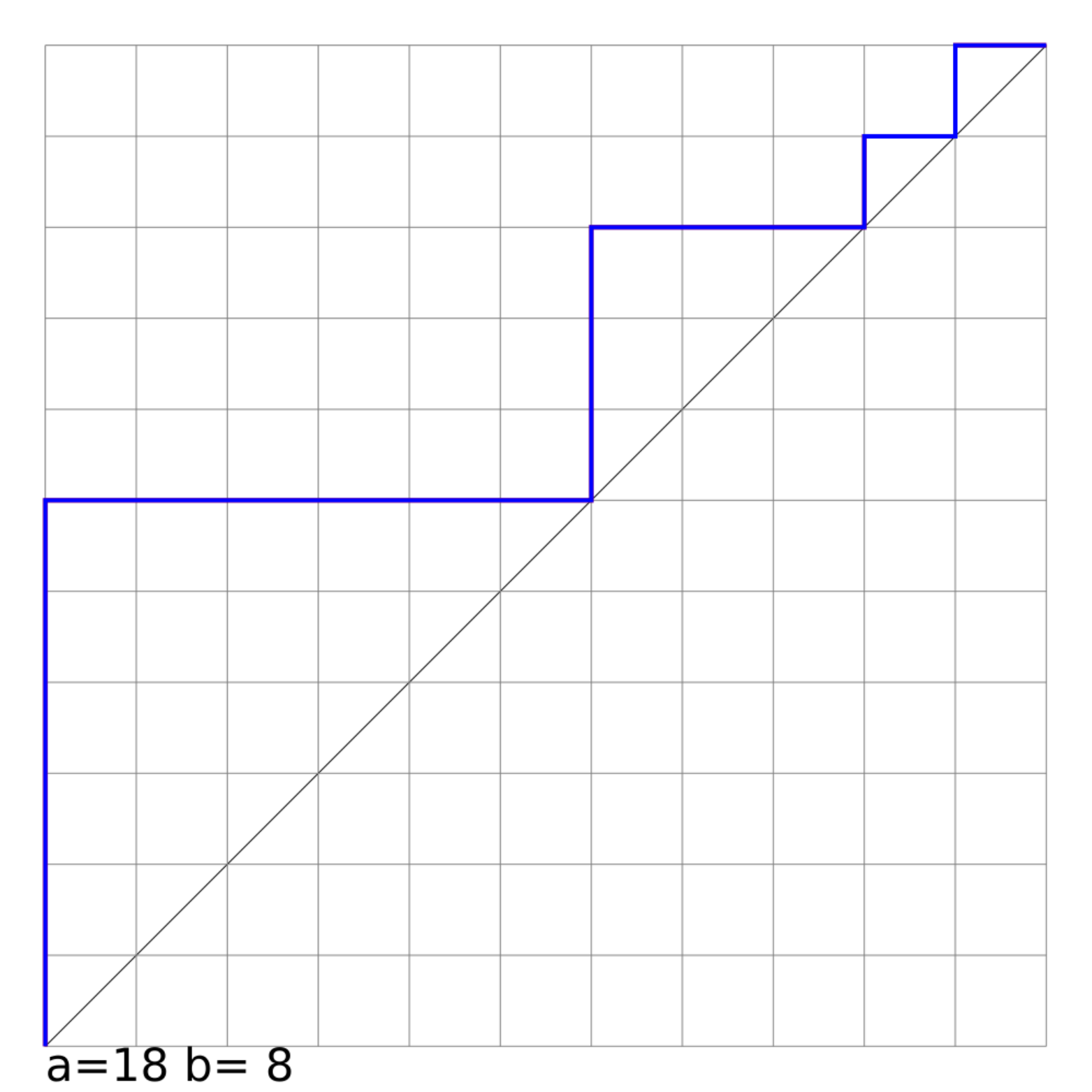}
    \caption{$p_{n,\lambda}$}
  \end{subfigure}
  \hspace{10pt}
  \begin{subfigure}[t]{0.21\textwidth}
    \includegraphics[width=\textwidth]{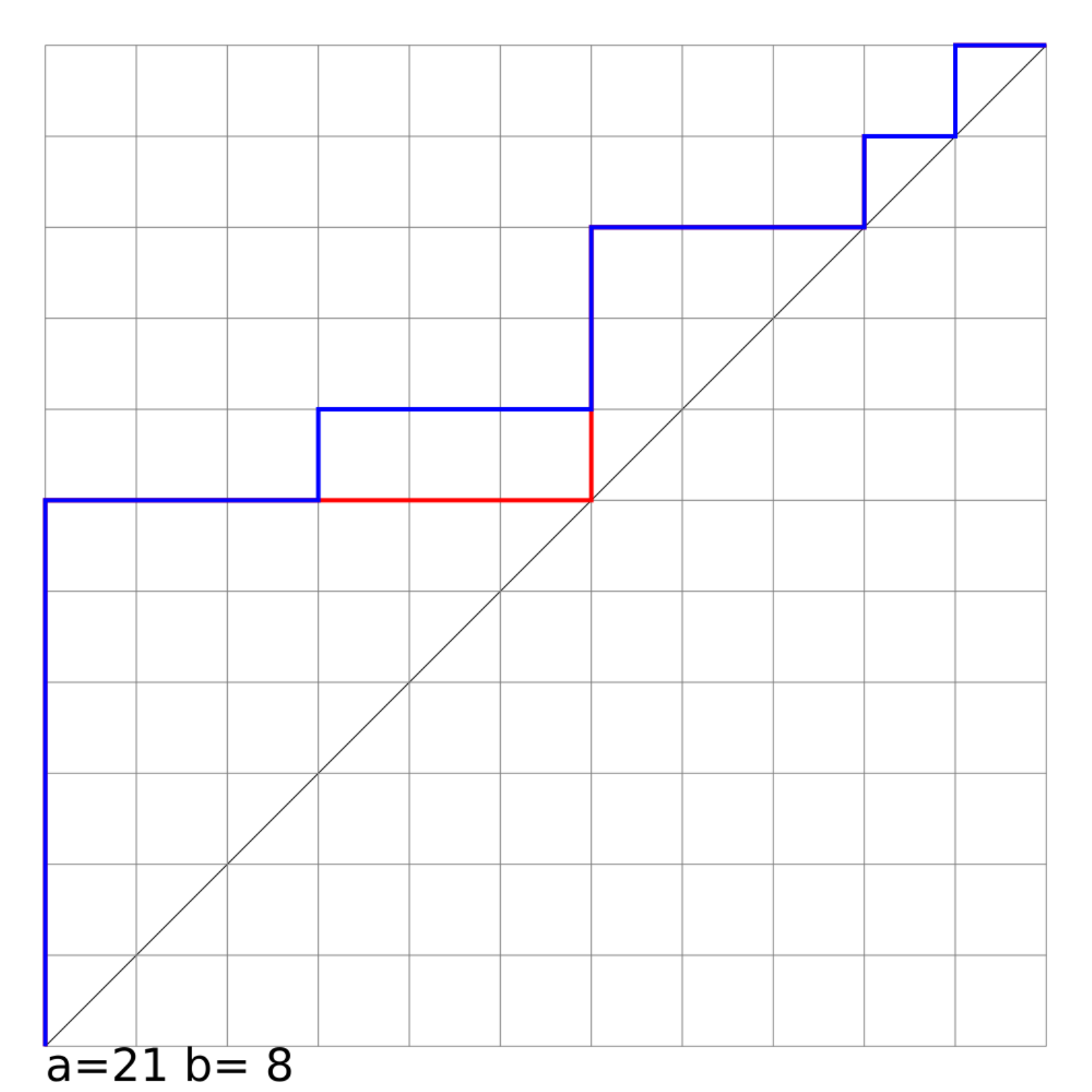}
    \caption{$p_{n,\lambda} \cdot A_7^3$}
  \end{subfigure}
  \hspace{10pt}
  \begin{subfigure}[t]{0.21\textwidth}
    \includegraphics[width=\textwidth]{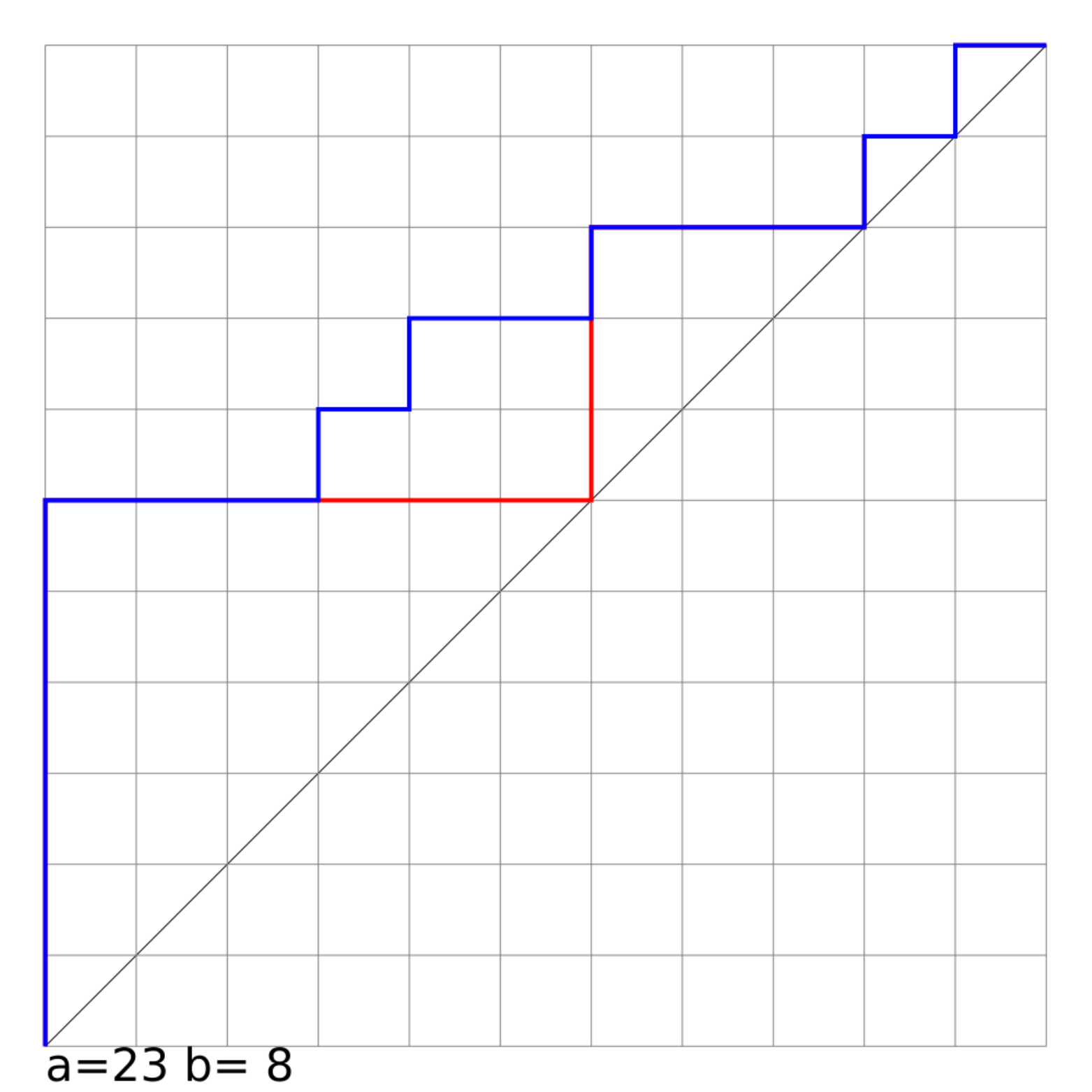}
    \caption{$p_{n,\lambda} \cdot A_7^3 A_8^2$}
  \end{subfigure}
  \hspace{10pt}
  \begin{subfigure}[t]{0.21\textwidth}
    \includegraphics[width=\textwidth]{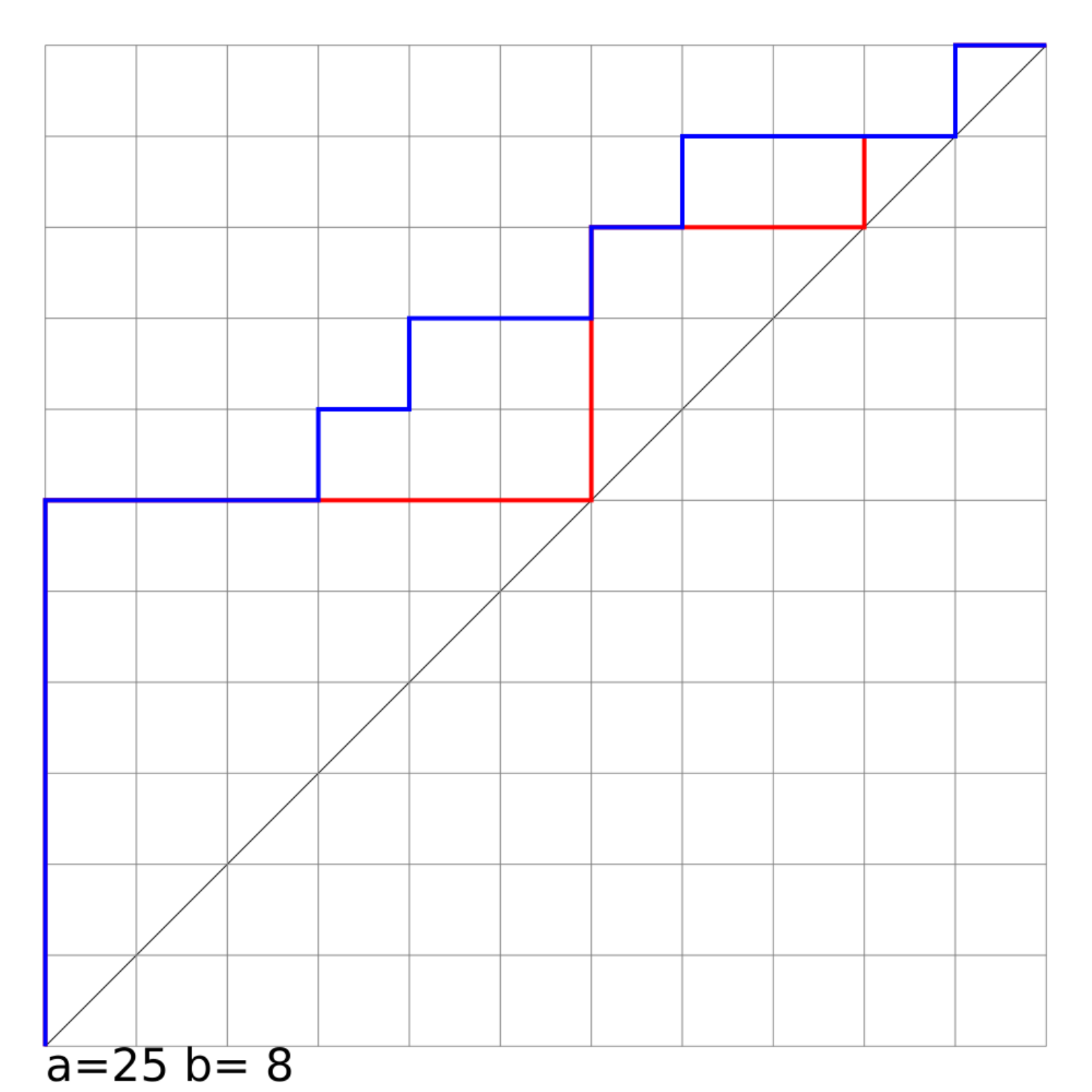}
    \caption{$p_{n,\lambda} \cdot A_7^3 A_8^2 A_{10}^2$}
  \end{subfigure}

  \begin{subfigure}[t]{0.21\textwidth}
    \includegraphics[width=\textwidth]{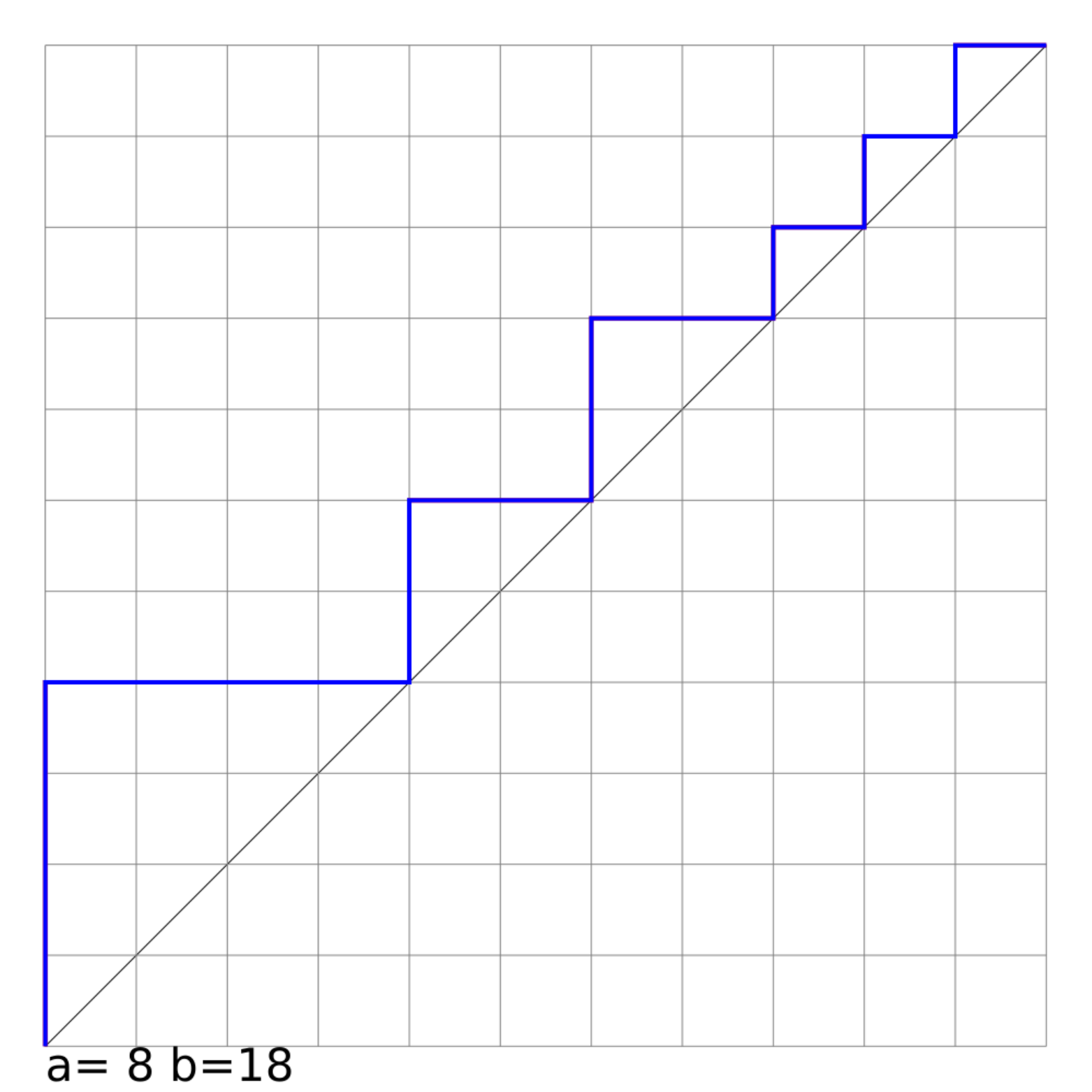}
    \caption{$p_{n,\lambda'}$}
  \end{subfigure}
  \hspace{10pt}
  \begin{subfigure}[t]{0.21\textwidth}
    \includegraphics[width=\textwidth]{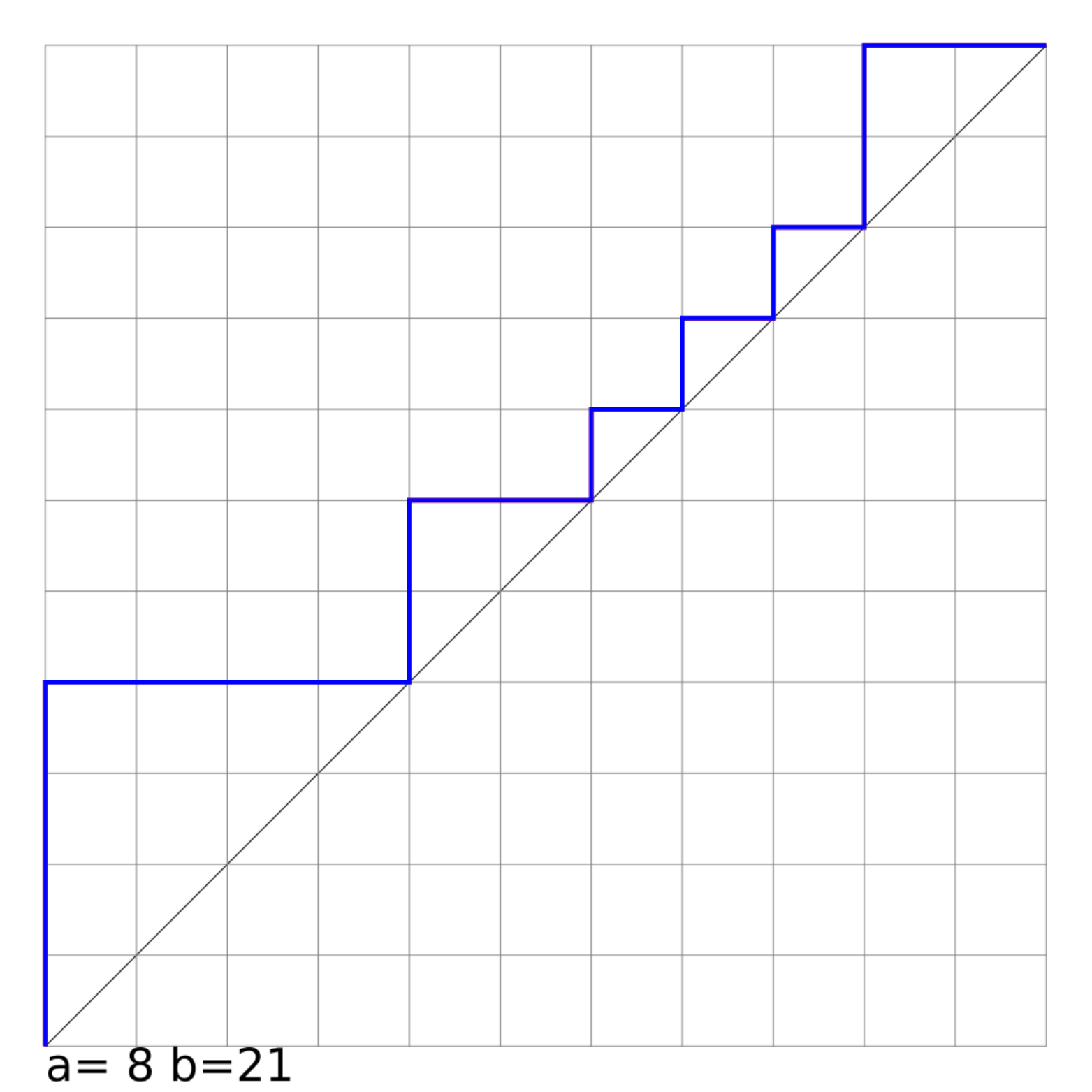}
    \caption{$p_{n,\lambda'} \cdot {B_{3,3}}$}
  \end{subfigure}
  \hspace{10pt}
  \begin{subfigure}[t]{0.21\textwidth}
    \includegraphics[width=\textwidth]{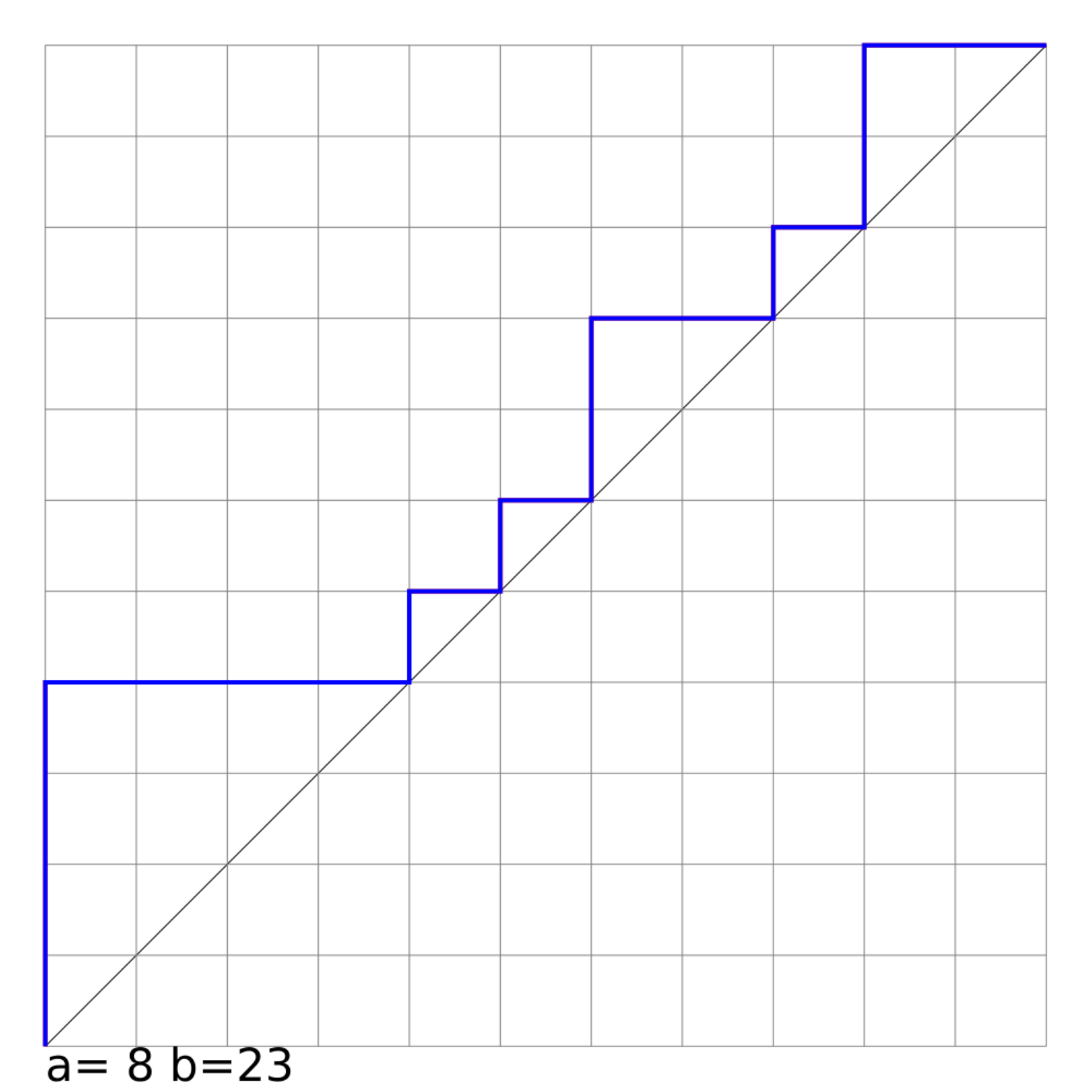}
    \caption{$p_{n,\lambda'} \cdot {B_{3,3} B_{2,2}}$}
  \end{subfigure}
  \hspace{10pt}
  \begin{subfigure}[t]{0.21\textwidth}
    \includegraphics[width=\textwidth]{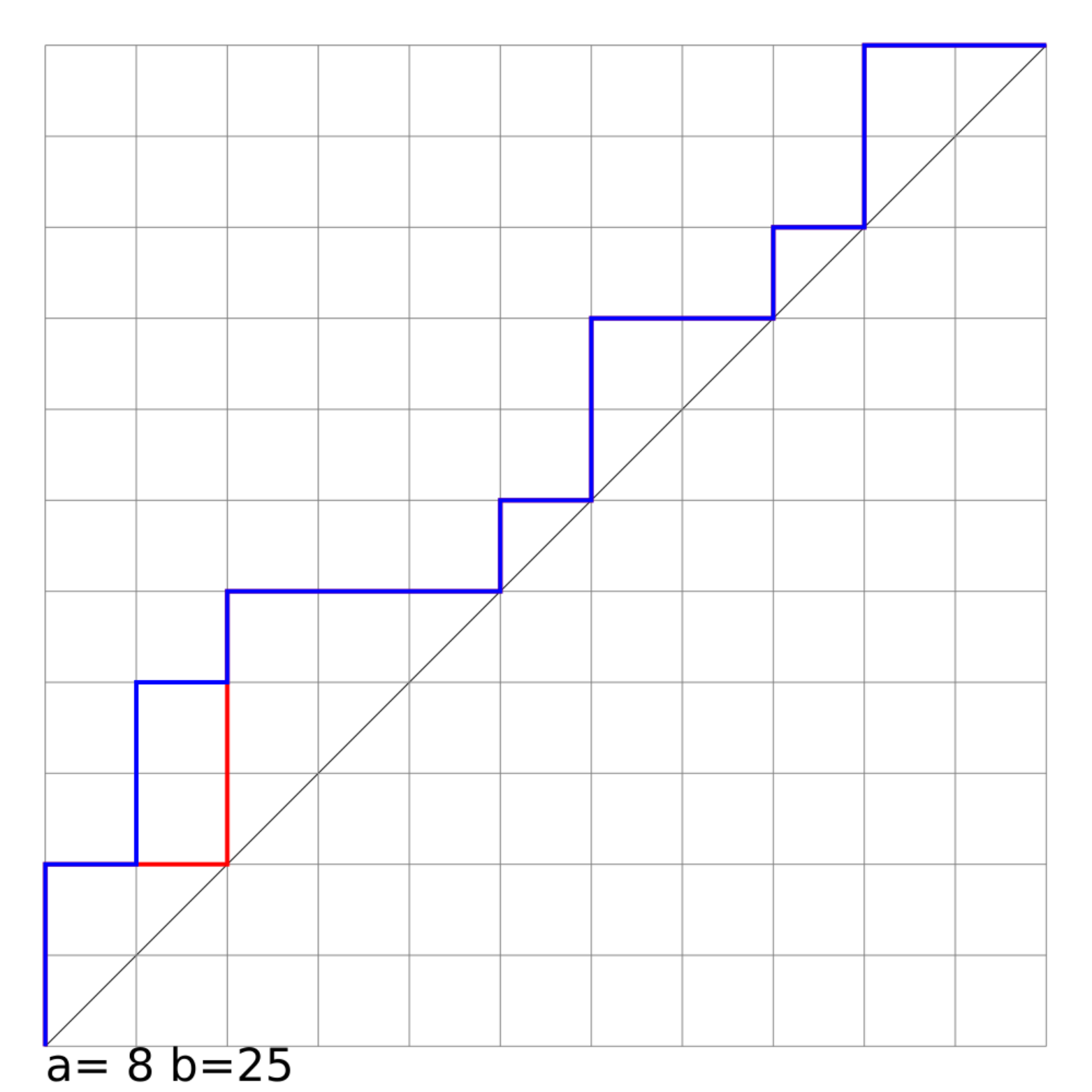}
    \caption{$p_{n,\lambda'} \cdot {B_{3,3} B_{2,2} B_{1,2}}$}
  \end{subfigure}
  \hspace{10pt}
  \caption{Each column forms an $\abp$-pair where $\lambda=(6,3,1,1)$.
    The count map for the last pair is $f(1,1)=3, f(1,2)=2, f(2,1)=2$.
    {The first three columns are pairs that live in $\Phi$ while the pair
    in the last column is a non-example for \cref{thm:Phi_bijection} as $p_{n,\lambda'} \cdot B_f
    \notin \mathcal{C}(n)$.}
    Note also that $(p_{11,\lambda} \cdot A_7^3 A_8^2 A_{10}^3, 
    p_{11,\lambda'} \cdot {B_{3,3} B_{2,2} B_{1,3}})$ do not form an $\ab$-pair
    although both are valid paths because the bounce changes in the former
    path. This is consistent with \cref{thm:ab_pairs}.}
  \label{fig:bijection_example}
\end{figure}

\cref{fig:bijection_example} gives some examples of $\abp$-pairs. Looking at a generic Dyck path,
it is not always easy to tell if it belongs to an $\abp$-pair; however,
we will now extract a subset of $\abp$-pairs which can be fully characterized.
Let $\mathcal{F}_n$ be the set of pairs $(f, \lambda)$, where $f$ is a count map
and $\lambda \in \Par_n$, that satisfies {the assumptions required in} \cref{thm:ab_pairs}, and such that $p_{n,\lambda'} \cdot B_f \in \cdycks(n)$. Define the sets
\begin{align*}
  \AF_n = \{ p_{n,\lambda} \cdot A^f \mid (f,\lambda) \in \mathcal{F}_n \}
  \quad \text{and} \quad 
  \BF_n = \{ p_{n,\lambda'} \cdot B_f \mid (f,\lambda) \in \mathcal{F}_n \}.
\end{align*}

\begin{thm}
  \label{thm:Phi_bijection}
  The map $\Phi : \AF_n \rightarrow \BF_n$ defined by
  \begin{equation}
    \Phi(p_{n,\lambda} \cdot A^f) = p_{n,\lambda'} \cdot B_f
  \end{equation}
  is an area and bounce flipping bijection.
\end{thm}

Note that if $f$ is the zero map, $\Phi$ maps $p_{n, \lambda}$ to $p_{n, \lambda'}$ in accordance with \cref{prop:partition}.

\begin{proof}
 Let $(f,\lambda) \in \mathcal{F}_n$, $\tau = p_{n,\lambda'} \cdot B_f$ and $p_{n,\alpha}$ be its bounce path.
  We can recover $\lambda$ and $f$ by the following procedure.
  The partition obtained by sorting its bounce path $\alpha$
  is $\lambda'$ (as $\tau \in \cdycks(n)$), which we denote $\mu$ for notational convenience. 
  Let $\ell = \ell(\bar{\mu})$. For $1 \leq i \leq \ell - 1$,
  recall that $\bar{\mu}_{\ell-i}$ is the $(i+1)$'th smallest part in $\bar{\mu}$
  and there are $d_i \equiv m_\mu(\bar{\mu}_{\ell-i})$ parts of that size.
  Let $\alpha_{j_1},\dots,\alpha_{j_{d_i}}$ be all those parts in $\alpha$ where the $j_i$'s are in decreasing order.
  Construct the count map $f$ on $\mathcal{I}_{\bnc}(\mu)$, for
  $1 \leq r \leq d_i$, by
  \begin{equation}
  	\label{eq:inv_Phi}
  f(i,r) = \sum_{j < j_r} \max\{0, \bar{\mu}_{\ell-i} - \alpha_{j} \}.
  \end{equation}

One can check that $p_{n, \mu} \cdot B_f = \tau$. The key idea is that $B_f$ simply reorders $\mu$ to get $\alpha$ by \cref{prop:si_not_bottom}. To complete the proof, we need to show that $p_{n, \lambda} \cdot A^f$ is a valid path. But this holds because we have chosen $\lambda$ to satisfy \cref{thm:ab_pairs}.
\end{proof}

As a consequence of \cref{thm:Phi_bijection}, we obtain this identity.

\begin{cor}
  Define
  \[
    G_n(q,t) = \sum_{\pi \in \AF_n \cup \BF_n} q^{\area(\pi)}t^{\bounce(\pi)}.
  \]
  Then $G_n(q, t) = G_n(t, q)$.
\end{cor}

\begin{rem}\label{rem:determine_ab_pair}
  To determine if a Dyck path $\pi \in \dycks(n)$ with bounce path $p_{n,\alpha}$  belongs to $\AF_n \cup \BF_n$, we have the following algorithm.
  \begin{enumerate}
  \item If $\alpha$ is a partition $\lambda$, determine the count map $f$ on
    $\mathcal{I}_{\bnc}(\lambda')$ by setting $f(i,r)$ to be the number of floating
    cells in the row $\row_\lambda(i,r)$; i.e.,
    \[
    f(i,r) = a_{\row_\lambda(i,r)} - r + 1.
    \]
    If there are floating cells in rows other than these or if $p_{n,\lambda'} \cdot B_f \notin \cdycks(n)$, then $\pi$ is not a member of $\AF_n$. Otherwise, $\pi \in \AF_n$.

  \item If $\alpha$ is not a partition and $\pi \in \cdycks(n)$, determine the count map $f$ as explained in the proof of \cref{thm:Phi_bijection}. If $f$ satisfies the conditions in \cref{thm:ab_pairs}, then $\pi \in \BF_n$.

  \item Otherwise, $\alpha$ is not a partition and $\pi \notin \cdycks(n)$, in which case $\pi \notin \AF_n \cup \BF_n$.
  \end{enumerate}
See \cref{fig:determine_ab_pair} for an example.
\end{rem}

\begin{figure}
  \captionsetup[subfigure]{labelformat=empty}
  \centering
  \begin{subfigure}[t]{0.2\textwidth}
    \includegraphics[width=\textwidth]{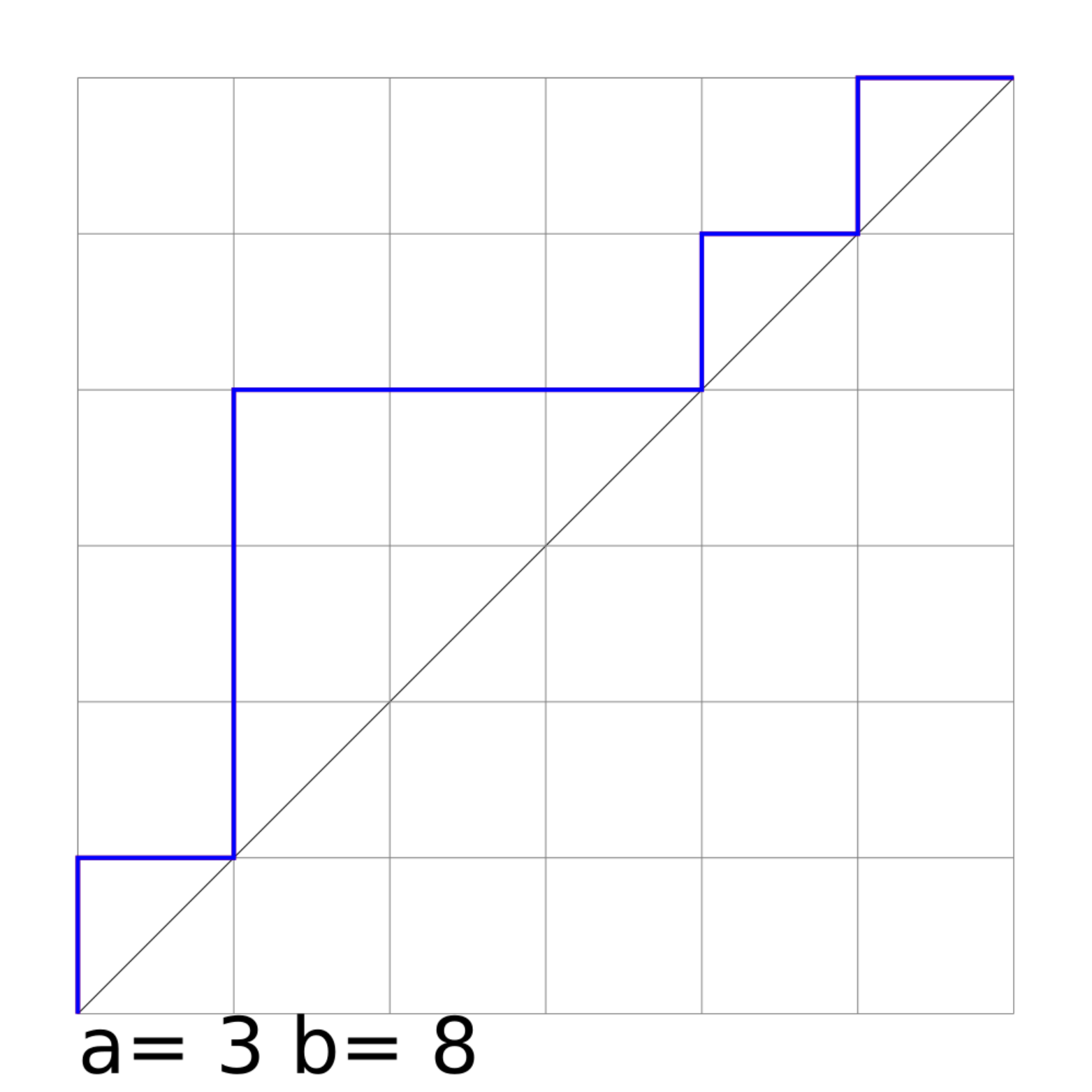}
    \caption{$\pi$ where $\alpha = (1, 3, 1, 1)$, $\lambda=(3,1,1,1)$, $f(1,1)=2 < \lambda'_1$. Hence, $\pi \in \BF_n$.}
  \end{subfigure}
  \hspace{30pt}
  \begin{subfigure}[t]{0.2\textwidth}
    \includegraphics[width=\textwidth]{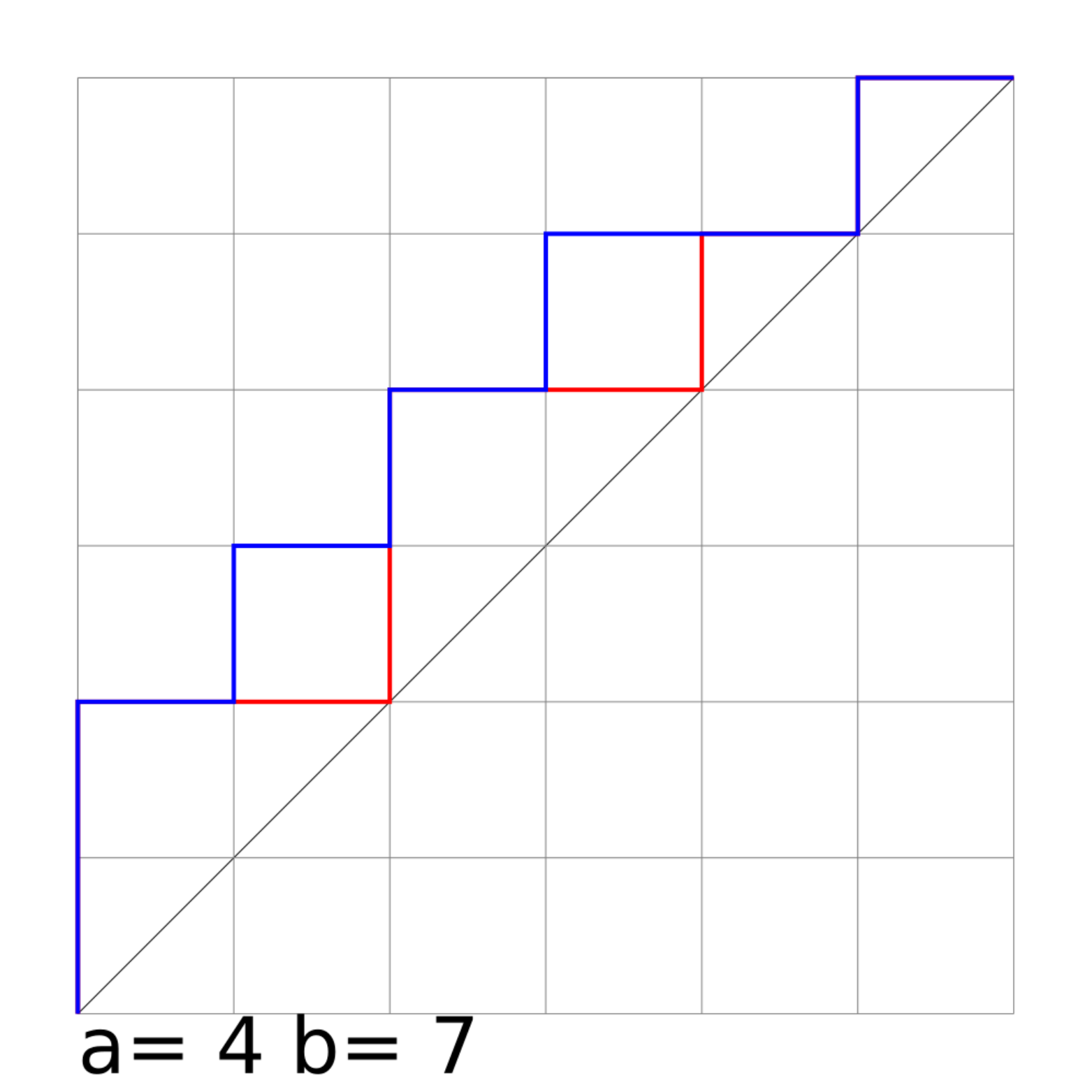}
    \caption{$\pi$ where $\alpha = \lambda=(2,2,1,1)$. $\pi \notin \AF_n$ as there is a floating cell in row $3$.}
  \end{subfigure}
  \hspace{30pt}
  \begin{subfigure}[t]{0.2\textwidth}
    \includegraphics[width=\textwidth]{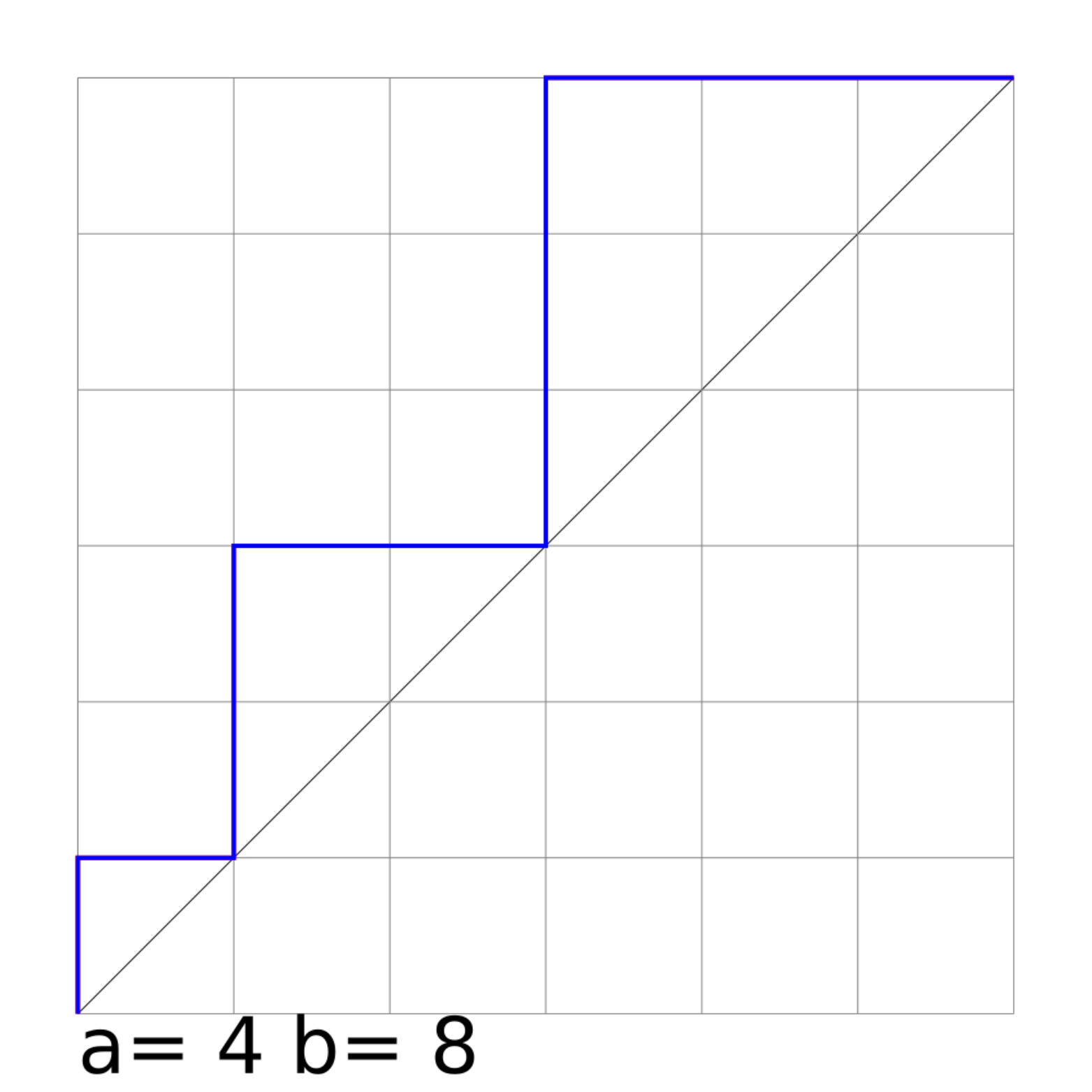}
    \caption{$\pi$ where $\alpha = (1, 2, 3)$ and $\lambda=(3,2,1)$. $\pi \notin \BF_n$ since $f(2,1)=3 > \lambda'_2$.}
  \end{subfigure}
  \caption{Determining if a path belongs to either $\AF_n$ or $\BF_n$.}
  \label{fig:determine_ab_pair}
\end{figure}

\begin{thm}
  \label{thm:num-bijection}
  Let $\phi$ be the golden ratio $(\sqrt{5}+1)/2$. For $n > 4$,
  \[
    1.23\phi^n \le | \AF_n \cup \BF_n| \le 2^n.
  \]
\end{thm}
\begin{proof}
  The upper bound is due to the fact that $|\BF_n| \le |\comp(n)|$
  and the number of compositions of $n$ is $2^{n-1}$. For the lower bound, consider partitions $\lambda$ of $n$ of the
  form $\langle 1^b, 2^a \rangle$ in frequency notation. If $a > 0$, then $b < \lambda'_1$. All assignments of the
  count map $f$ on $\mathcal{I}_{\bnc}(\lambda) = \{ (1,1), \dots, (1,a)\}$
  satisfying $b \ge f(1,1) \dots \ge f(1,b)$ lead to an $ab$-pair
  $(p_{n,\lambda'} \cdot A^f, p_{n,\lambda} \cdot B_f) \in (\AF_n,
  \BF_n)$ by \cref{thm:ab_pairs}. In particular, there are $\binom{a+b}{b}$ such pairs. When
  summed over all such partitions $\lambda$, it is a standard fact that the total number of such pairs is given by the $(n+1)$'th Fibonacci number (the number of sequences of $1$'s and $2$'s that sum to $n$).
  Using induction, one can show that the $n$'th Fibonacci number has a lower bound of $\phi^{n-2}$.
  The only such self-conjugate partitions are $(1), (2,1)$ and $(2,2)$. So, when $n > 4$, the two Dyck paths of each $\abp$-pair are distinct. Hence, there are a total of at least  $2\phi^{n-1} \approx 1.23\phi^n$ paths in $|\AF_n \cup \BF_n|$.
\end{proof}

\subsection{Including more $\abp$-pairs}
\label{sec:ext bij}

It does not seem very easy to extend the bijection $\Phi$ in \cref{thm:Phi_bijection} while still maintaining control on the size of the set in bijection. If we do not care for the latter, we can certainly extend the bijection. We illustrate the idea with an example.

For $n=11$ and $\lambda=(6,3,1,1)$, consider the count map $f$
defined by setting $f(1,1) = 3$ and zero everywhere else.
Then we have the following $\abp$-pair
in natural bijection:
\begin{center}
  \begin{tabular*}{0.5\textwidth}{c c}
    \includegraphics[width=0.2\textwidth]{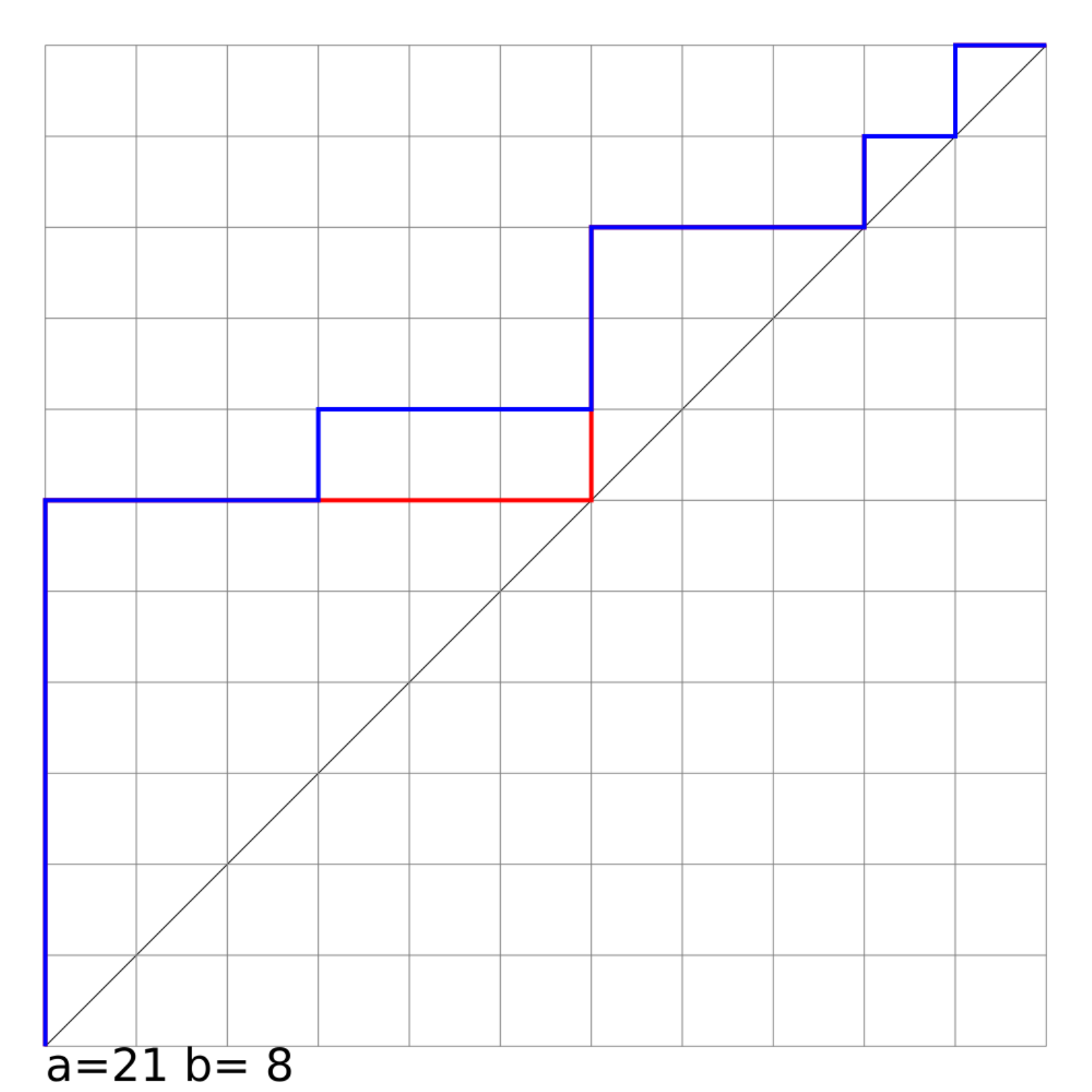} &
    \includegraphics[width=0.2\textwidth]{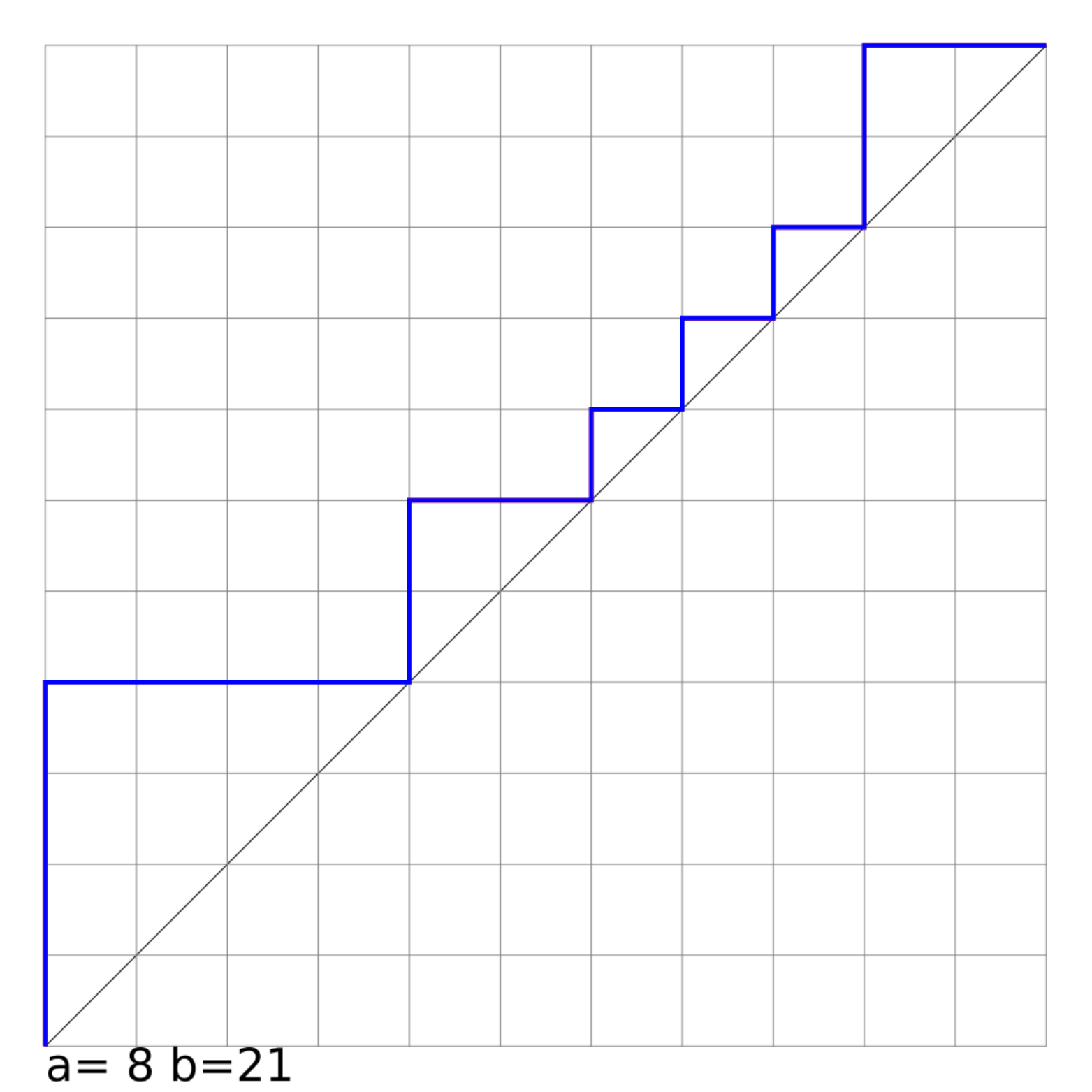} \\
    \CAP{$p_{11, (6,3,1,1)} \cdot A^f$} & \CAP{$p_{11, (4, 2, 2, 1, 1, 1)} \cdot B_f$}
  \end{tabular*}
\end{center}
The bounce points are unchanged and no floating cells are present above row $7$ in the left path, and the right path is unchanged below row $6$.
Because of this observation, we can apply 
another count map, $g$ say, such that $g(1,1)=2$ and zero everywhere else, to them.
Consider the paths $\sigma$ and $\tau$ below.
\begin{center}
  \begin{tabular*}{0.6\textwidth}{c c}
    \includegraphics[width=0.2\textwidth]{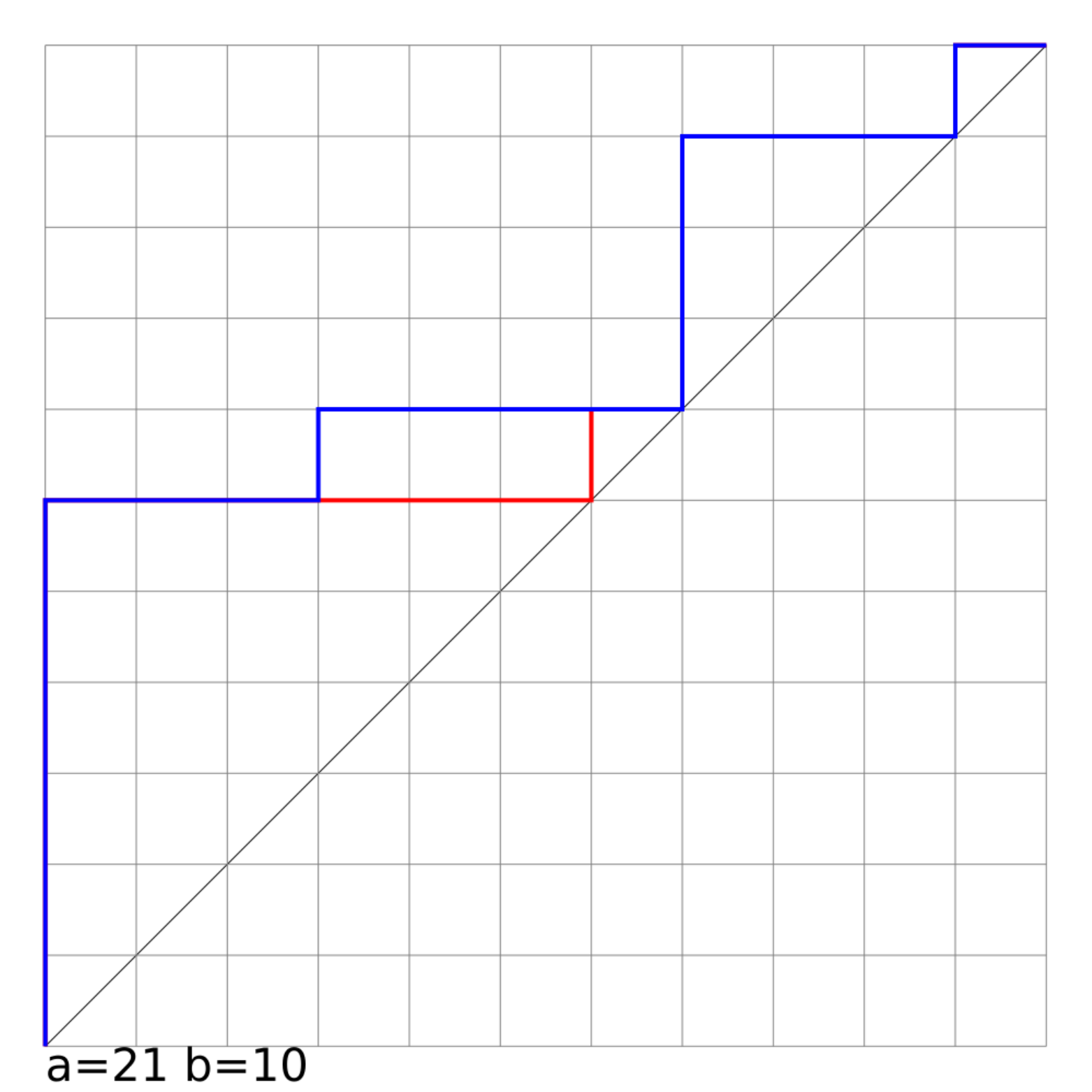} &
    \includegraphics[width=0.2\textwidth]{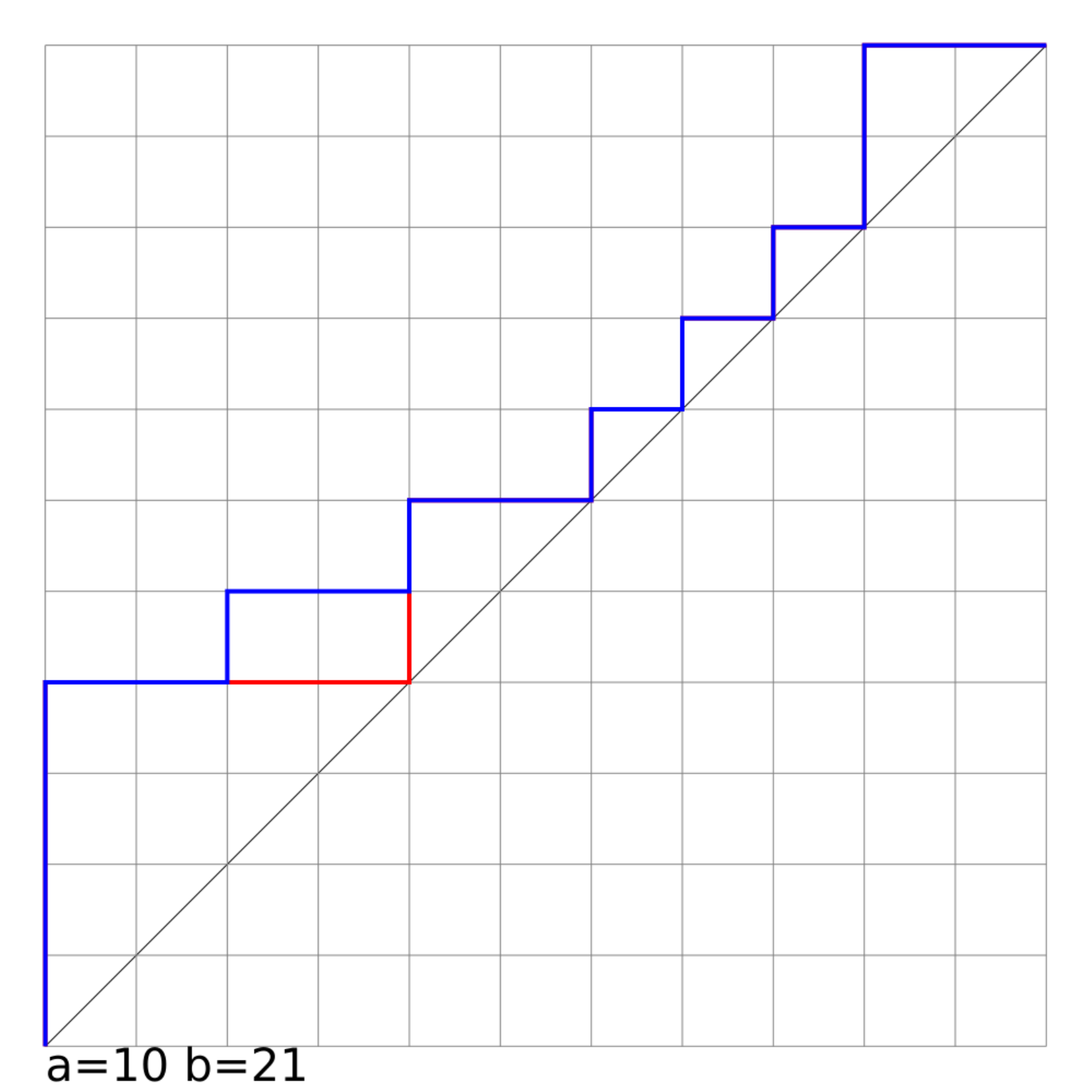} \\
    \CAP{$\sigma = p_{11, (6,3,1,1)} \cdot A^fB_g$} & \CAP{$\tau = p_{11,(4, 2, 2, 1, 1, 1)} \cdot B_fA^g$}
  \end{tabular*}
\end{center}
Then $(\sigma, \tau)$ also form an $\abp$-pair but they do not exist in $(\AF_n, \BF_n)$. 
The map $g$ in this example is chosen in such a way that $(p_{n, \lambda'} \cdot A^g, p_{n, \lambda} \cdot B_g)$ does belong to $(\AF_n, \BF_n)$, $p_{n,\lambda} \cdot A^f B_g = p_{n,\lambda} \cdot B_g A^f$ and $p_{n,\lambda'} \cdot B_fA^g
  = p_{n,\lambda'} \cdot A^gB_f$.
  Hence, we can recover $\lambda'$, $f$, and $g$ from $\tau$ as follows.
  \begin{enumerate}
  \item $\lambda'$ can be recovered from the bounce path $p_{n,\alpha}$ of $\tau$ by
    sorting $\alpha$,
  \item $p_{n,\lambda'} \cdot A^g$ is the path we get by modifying $\tau$ by
    changing its bounce path to $\lambda'$. Then $g$ can be recovered from this
    path by applying $(1)$ in \cref{rem:determine_ab_pair}, and
  \item $p_{n,\lambda'} \cdot B_f$ is the path we get by modifying $\tau$ by
    dropping all its floating cells. Thus $f$ is recovered from this path
    by applying $(2)$ in \cref{rem:determine_ab_pair}.
  \end{enumerate}
Similarly we can also recover $\lambda$, $f$, and $g$ from $\sigma$.
It's not hard to see that the above-mentioned commutativity is achieved only
if the count map $g$ takes on non-zero values over $(i,r)$ when
the row given by $\row_{\lambda'}(i,r)$ is below the rows changed by $f$ in
$p_{n,\lambda} \cdot B_f$.

\cref{thm:Gamma} gives the conditions on $f$ and $g$ to create such pairs.
The example here illustrates the salient points. Since the key ideas of the proof are similar to those in \cref{thm:Phi_bijection}, we simply state it here. 

\begin{thm}
\label{thm:Gamma}
  Let $\mathfrak{F}_n \subset \mathcal{F}_n^2$ be the set of pairs $((f,\lambda),(g,\mu))$ which satisfies
  \begin{enumerate}
  \item $\mu = \lambda'$,
  \item $|f| \ge |g|$,
  \item let $m = \ell(\bar\lambda)$ and $i = \max \{i : f(i,r) > 0\}$, then $g(j,r) = 0$ if $j > m-i$, and $g(m-i,r) = 0$ if $r > d$.
  \end{enumerate}
  Define the sets
  \begin{align*}
    \mathfrak{AF}_n = \{ p_{n,\lambda} \cdot A^f B_g : ((f,\lambda),(g,\lambda')) \in \mathfrak{F}_n\},\\
    \mathfrak{BF}_n = \{ p_{n,\lambda'} \cdot B_f A^g : ((f,\lambda),(g,\lambda')) \in \mathfrak{F}_n\}.
  \end{align*}
  Then the map $\Gamma : \mathfrak{AF}_n \rightarrow \mathfrak{BF}_n$ given by
  \begin{equation}
    \Gamma(p_{n,\lambda} \cdot A^f B_g) = p_{n,\lambda'} \cdot B_f A^g
  \end{equation}
  is an area and bounce flipping bijection.
\end{thm}

\section{Area-minimal and bounce-minimal paths}
\label{sec:minimal}

\subsection{The operators $U_i$ and $D_i$}
\label{sec:UD}

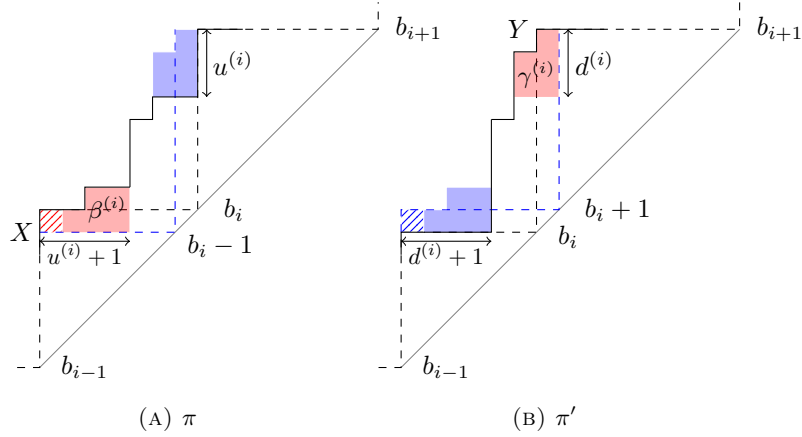
\begin{figure}[h!]
  \centering
  \begin{subfigure}{0.35\textwidth}
  \begin{tikzpicture}[scale=0.6]
    \draw[gray,very thin] (0.5,0.5) -- (8,8);

    \draw[dashed] (0, 0.5) --  (0.5,0.5) -- (0.5,4);
    \draw[dashed]  (0.5,4) -- (4,4) -- (4,8);
    \draw[dashed] (4,8) -- (8,8) -- (8, 8.6);

    \draw[color=white,pattern color=red,pattern=north east lines] (0.5,3.5) rectangle (1,4);

    \draw[dashed,color=blue] (0.5,3.5) -- (3.5,3.5) -- (3.5,8);

    \draw[color=white,fill=red,opacity=0.3] (1,3.5) -- (1,4) -- (1.5,4) -- (1.5,4.5) -- (2.5,4.5) -- (2.5,3.5);

    \draw[color=white,fill=blue,opacity=0.3] (3,6.5) -- (3,7.5) -- (3.5,7.5) -- (3.5,8) -- (4,8) -- (4,6.5);

    \draw (0.5,3) -- (0.5,4) -- (1.5,4) -- (1.5,4.5) -- (2,4.5) -- (2.5,4.5);
    \draw (2.5,4.5) -- (2.5,6) -- (3,6) -- (3,6.5) -- (4,6.5) -- (4,8) -- (5,8);

    \node at (4.8,4) {$b_i$};
    \node at (4.5,3.2) {$b_i - 1$};
    \node at (1.5,0.5) {$b_{i-1}$};
    \node at (8.9,8) {$b_{i+1}$};
    \node at (0.1,3.5) {$X$};
    \node at (2,4) {\small $\beta^{(i)}$};
    \draw[<->,thin] (4.2,6.5) -- (4.2,8) node[midway,xshift=10] {$u^{(i)}$};
    \draw[<->,thin] (0.5,3.3) -- (2.5,3.3) node[midway,yshift=-5] {\resizebox{1cm}{!}{$u^{(i)}+1$}};

  \end{tikzpicture}
  \caption{$\pi$}
  \label{fig:U_left}
  \end{subfigure}
  \hspace*{10pt}
  \begin{subfigure}{0.35\textwidth}
  \begin{tikzpicture}[scale=0.6]
    \draw[gray,very thin] (0.5,0.5) -- (8,8);

    \draw[dashed] (0, 0.5) --  (0.5,0.5) -- (0.5,3.5);
    \draw[dashed] (0.5,3.5) -- (3.5,3.5) -- (3.5,8);
    \draw[dashed] (3.5,8) -- (8,8) -- (8, 8.6);

    \draw[color=white,pattern color=blue,pattern=north east lines] (0.5,3.5) rectangle (1,4);

    \draw[dashed,color=blue] (0.5,3.5) -- (0.5,4) -- (4,4) -- (4,8);

    \draw[color=white,fill=blue,opacity=0.3] (1,3.5) -- (1,4) -- (1.5,4) -- (1.5,4.5) -- (2.5,4.5) -- (2.5,3.5);

    \draw[color=white,fill=red,opacity=0.3] (3,6.5) -- (3,7.5) -- (3.5,7.5) -- (3.5,8) -- (4,8) -- (4,6.5);

    \draw (0.5,3) -- (0.5,3.5) -- (2.5,3.5) -- (2.5,6);
    \draw (2.5,6) -- (3,6) -- (3,7.5) -- (3.5,7.5) -- (3.5,8) -- (5,8);

    \node at (5.3,4) {$b_i+1$};
    \node at (4.2,3.4) {$b_i$};
    \node at (1.5,0.5) {$b_{i-1}$};
    \node at (8.9,8) {$b_{i+1}$};
    \node at (3.1,8) {$Y$};
    \node at (3.5,7) {\small $\gamma^{(i)}$};
    \draw[<->,thin] (4.2,6.5) -- (4.2,8) node[midway,xshift=10] {$d^{(i)}$};
    \draw[<->,thin] (0.5,3.3) -- (2.5,3.3) node[midway,yshift=-5] {\resizebox{1cm}{!}{$d^{(i)}+1$}};

  \end{tikzpicture}
  \caption{$\pi'$}
  \label{fig:U_right}
  \end{subfigure}
  \caption{The diagram on the left is the relevant part of $\pi$ and that on the right is that of $\pi' = \pi \cdot U_i$.
On the diagram to the left, the solid black line is part of $\pi$, the dashed black line is its bounce path and the dashed blue line is the bounce path of $\pi'$. 
Note that $\pi$ and its bounce path must meet at ${X} = (b_{i-1}, b_i - 1)$. The red shaded region is the partition $\beta^{(i)}$ (read $(2, 2, 1)$ in this example) that is to be removed and the blue shaded region is the added partition after rotation and reflection. 
On the diagram to the right, the solid black line is part of $\pi'$, the dashed black line is its bounce path and the dashed blue line is the bounce path of $\pi$. 
Note that $\pi'$ must take at least two east steps after ${Y} = (b_i, b_{i+1})$ so that $\pi' \cdot D_i \neq \bot$. It is clear that $\pi' \cdot D_i = \pi$ in this example.}
  \label{fig:U}
\end{figure}

Given $\pi \in \dycks(n)$ with 
bounce points $(b_0 = 0, \dots, b_m)$, we now define $\pi \cdot U_i$ and $\pi \cdot D_i$ for $1 \leq i \leq m$, which are either new Dyck paths of the same semilength or the undefined path $\bot$. 
To define $\pi \cdot U_i$, we need some notation.
First, let $u^{(m)} = 0$ and $u^{(i)} = b_{i+1} - h_{b_i}$ be the height difference between the $(i+1)$'th bounce position and the height of the column at the $i$'th bounce position.
Now, let $\beta^{(i)}$ be the partition defined by $\beta^{(i)}_j = h_{b_{i-1} +
  u^{(i)} + 2-j} - b_i + 1$ for $1 \le j \le u^{(i)}$. {Informally, $U_i$
  will first remove a block of cells of shape $\beta^{(i)}$ in addition to removing an
  extra
  cell at column $b_{i-1} +1$. This step ensures that the bounce point at $i$
  shifts down by $1$. Then, the cells of shape $\beta^{(i)}$ are
  transposed (the height of the shape is now $u^{(i)}$) and added to the north
  of bounce point $i$. This step ensures that
  the next bounce point still appears at $b_{i+1}$ and that the area has
  decreased by one.
Formally,}
\begin{equation}\label{eq:U}
\pi \cdot  U_i = 
\begin{cases}
\bot &  h_{b_{i-1}} = b_i,\\
\ds \pi \cdot C^{-1}_{b_{i-1}+1} \left( \prod_{j\ge 1} C_{b_{i-1}+1+j}^{-\beta^{(i)}_{u^{(i)}-j+1}} \right) \left( \prod_{j\ge 1} A_{b_{i+1}-u^{(i)}+j}^{\beta_j} \right) & \text{otherwise}.
\end{cases}
\end{equation}
We can also have $\pi \cdot U_i = \bot$ in the second case and this can happen for two distinct reasons. To see this, it will help to look at \cref{fig:U}. Notice that if either $u^{(i)}$ is too high so that the row being removed (in red) crosses the diagonal, or $\beta^{(i)}_1$ is too high so that the row being added (in blue) overhangs the diagram, the operation will result in $\bot$.
\cref{fig:U_operator} shows a concrete example of $\pi \cdot U_i$.

\begin{figure}[h!]
  \centering
  \captionsetup[subfigure]{labelformat=empty}
  \begin{subfigure}[t]{0.25\textwidth}
    \includegraphics[width=\textwidth]{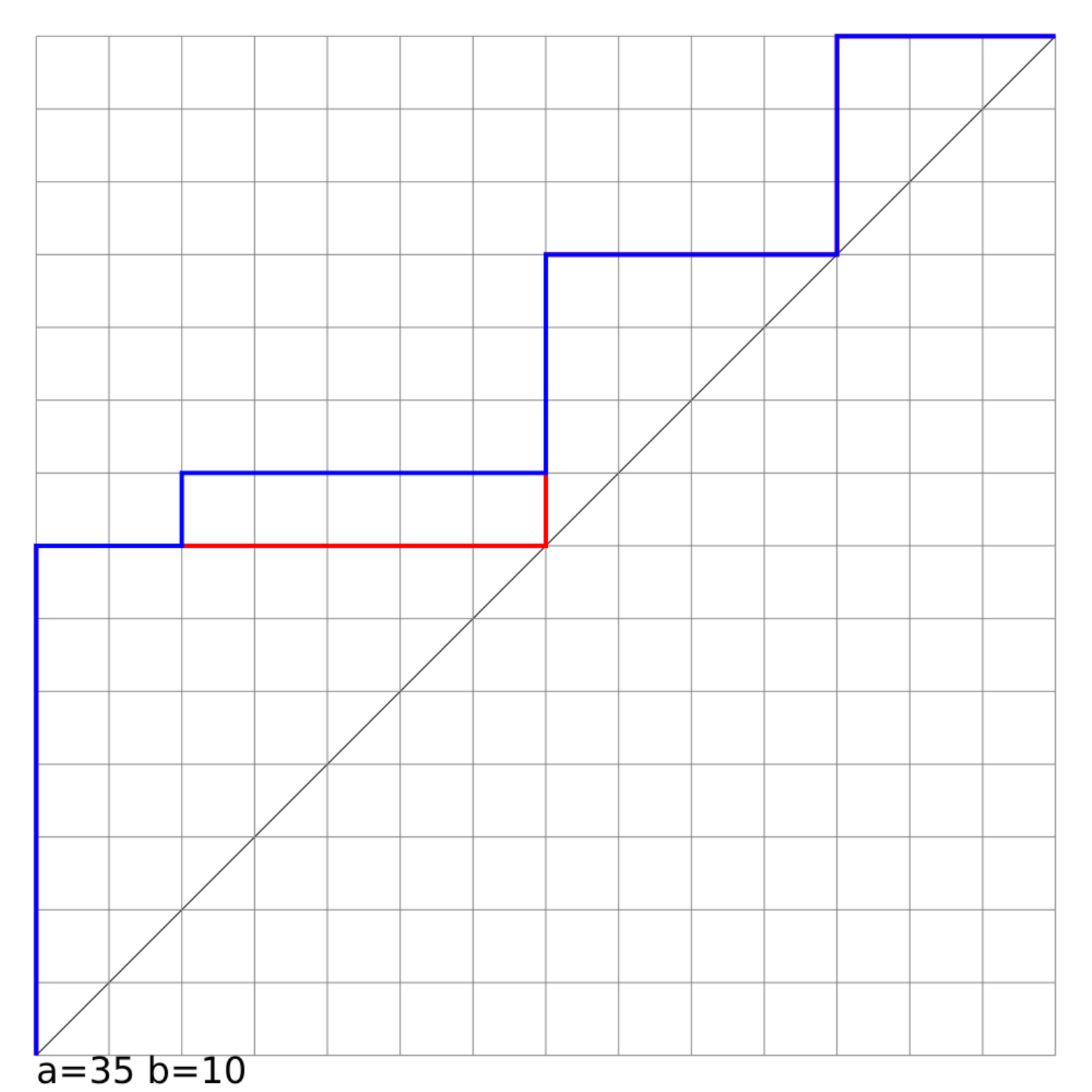}
  \end{subfigure}
  \hspace{10pt}
  \begin{subfigure}[t]{0.25\textwidth}
    \includegraphics[width=\textwidth]{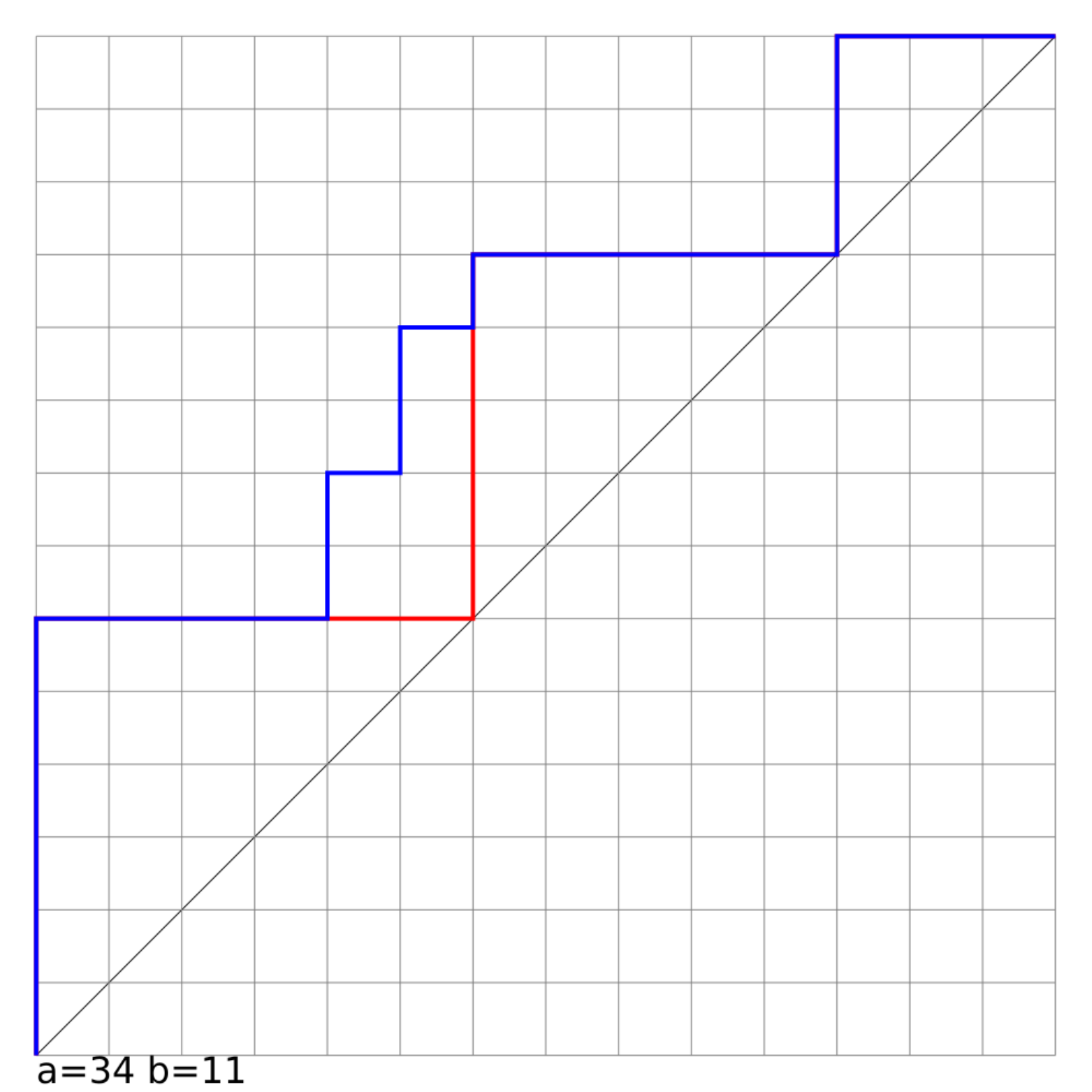}
  \end{subfigure}
  \caption{
    Example of $U_i$. 
    On the left is $\pi$ with bounce points $(0, 7, 11, 14)$, 
    $u^{(1)} = 3$, $\beta^{(1)} = (2,2,1)$, 
    $u^{(2)} = 3$, $\beta^{(2)} = (1,1,1)$.
    On the right is {$\pi' = \pi \cdot U_1$} where 
    $U_1 = C_1^{-1}C_2^{-1}C_3^{-2}C_4^{-2}A_9^2 A_{10}^2A_{11}^1$.
    Note that $\pi \cdot U_2 = \bot$ because $\pi \cdot C_{11}^{-1}= \bot$.
    }
  \label{fig:U_operator}
\end{figure}

\begin{prop}\label{prop:U_changes_ab}
If $\pi \cdot  U_i \neq \bot$, then $\area(\pi \cdot  U_i) = \area(\pi) - 1$ and $\bounce(\pi \cdot  U_i) = \bounce(\pi) + 1$.
\end{prop}

\begin{proof}
  It's clear from \eqref{eq:U} that $\area(\pi \cdot U_i) = \area(\pi) - 1$.
  Let $b_1,\dots,b_m$ be the bounce points of $\pi$.
  Clearly, the first $i-1$ bounce points of $\pi$ are also present in $\pi \cdot U_i$.
  Heading north from $(b_{i-1},b_{i-1})$,
  the first east step occurs at $(b_{i-1},b_i-1)$
  due to the removed cell at $X$ (see \cref{fig:U_left}).
  Therefore, the $i$'th bounce point must be $b_i-1$.
  Then, heading north from $(b_i - 1, b_i-1)$, the first east step occurs at
  $(b_i-1,b_{i+1})$ because $U_i$ always fills the $b_i$'th column of $\pi$
  up to height $b_{i+1}$. Therefore the remaining bounce points must be $b_{i+1},\dots,b_m$ and $\bounce(\pi \cdot U_i) = \bounce(\pi)+1$.
\end{proof}

Let $d^{(i)} = a_{b_i} - a_{b_i+1}$ be the difference between the number of
whole cells in the rows at the $i$'th bounce point. Let $h'_i = a_i + n +1- i$
{be the number of cells in the row $i$ to the right of the path}.
Define the partition $\gamma^{(i)}$ by {$\gamma^{(i)}_j =
h'_{b_{i+1} - d^{(i)} + j} - (n - b_i-1)$ for $1 \le j \le d^{(i)}$}.
Then set
\begin{equation}\label{eq:Di}
  \pi \cdot D_i =
  \begin{cases}
    \bot & h_{b_i+2} = b_{i+1}, \\
    \ds \pi \cdot \left(\prod_{j\ge 1} A_{b_{i+1}-j+1}^{-\gamma_{d^{(i)}-j+1}} \right) \left(\prod_{j\ge 1} C_{b_{i-1}+d^{(i)}+2-j}^{\gamma_j}
    \right)C_{b_i+1} & \text{otherwise}.
  \end{cases}
\end{equation}
{For the example in \cref{fig:U_operator}, one can check that $\pi' \cdot
  D_1 = \pi$, where $d^{(1)} = a_6 - a_7 = 5 - 2 = 3$, $\gamma_1^{(1)} =
  h'_9 - 7 = 2$, $\gamma_2^{(1)} = h'_{10} - 7 = 2$, and $\gamma_3^{(1)} = h'_{11} - 7 = 1$. }
The following shows that $D_i$ and $U_i$ are inverses of each other whenever their actions are nontrivial.

\begin{prop}
  If $\pi \cdot  D_i \neq \bot$, then $\area(\pi \cdot  D_i) = \area(\pi) + 1$ and $\bounce(\pi \cdot  D_i) = \bounce(\pi) - 1$ and $\pi \cdot D_i  U_i = \pi$. Similarly, if $\pi \cdot  U_i \neq \bot$, then $\pi \cdot U_i  D_i = \pi$.
\end{prop}

\begin{proof}
  Like in \cref{prop:U_changes_ab}, let $b_1,\dots,b_m$ be the
  bounce points of $\pi$. Then $\pi \cdot D_i$ shares the first $i-1$ bounce points of $\pi$.
  We then trace the bounce path of $\pi \cdot D_i$ starting from $(b_{i-1},b_{i-1})$.
  Heading north, the first east step occurs at $(b_{i-1},b_i+1)$ due to the cell being added at column $b_{i-1}+1$
  (see \cref{fig:U_right}). So, the $i$'th bounce point is $b_i+1$. Finally, heading north from $(b_i+1,b_i+1)$,
  the first east step occurs at $(b_i+1,b_{i+1})$ due to the condition in $\pi$ that $h_{b_i+2}= b_{i+1}$.
  Thus, the remaining bounce points must be $b_{i+1},\dots,b_m$ and $\bounce(\pi \cdot D_i) = \bounce(\pi) -1$.

  Suppose $\tau = \pi \cdot D_i \ne \bot$.
  Then $\tau$ satisfies $h_{b_{i-1}} < b_i+1$.
  Additionally, $d^{(i)}$ of $\pi$ and $u^{(i)}$ of $\tau$ are equal
  since the height of column $b_i$ of $\tau$ is $b_{i+1} - d^{(i)}$.
  Thus, the partition $\beta^{(i)}$ of $\tau$ is equal to $\gamma^{(i)}$ of $\pi$ and $\tau \cdot U_i = \pi$.

Suppose $\sigma = \pi \cdot U_i \ne \bot$. Then $\sigma$ satisfies $h_{b_i+2} = b_{i+1}$. In
addition, $d^{(i)}$ of $\sigma$ and $u^{(i)}$ of $\pi$ are equal since the removal of $\beta^{(i)}$ ensures that
$\sigma$ has $u^{(i)}+1$ east steps followed by one north step starting from
$(b_{i-1},b_i-1)$. Consequently, $\gamma^{(i)}$ of $\sigma$ and $\beta^{(i)}$ of $\pi$ are equal
and $\sigma \cdot D_i = \pi$.
\end{proof}

The following remark illustrates two simple cases where the action of $U_i$ is simple.

\begin{rem}\label{rem:simple_U}
  For any Dyck path $\pi$ with bounce path $p_{n,\alpha}$,
  the following two special conditions ensure that $\pi \cdot U_i \ne \bot$.
  
  \begin{enumerate}
  \item Suppose $u^{(i)} = 0$, $h_{b_{i-1}} < b_i$, and $a_{b_i} > 1$. In this case, removing the cell $X$ as shown in the left diagram of \cref{fig:U} will do the job.
  
  \item Suppose $\alpha_i > \alpha_{i+1}+1$ and $a_{b_i+1} = 0$. In this case, there are no floating cells in the region between the bounce points $b_i$ and $b_{i+1}$. Thus, $\beta^{(i)}$ is a single row of length $u^{(i)} = \alpha_{i+1}$ and this can be moved as shown in \cref{fig:U}.
  \end{enumerate}
\end{rem}

The following remark gives the inverse cases of \cref{rem:simple_U}.

\begin{rem}\label{rem:simple_D}
  For any Dyck path $\pi$ with bounce path $p_{n,\alpha}$,
  the following two special conditions ensure that $\pi \cdot D_i \ne \bot$.
  
  \begin{enumerate}
  \item Suppose $i < \ell(\alpha)$, $d^{(i)} = 0$ and either $b_i = n-1$ or $h_{b_i+2} = b_i+1$. In this case, adding the cell $A$ as shown in the left diagram of \cref{fig:U} will do the job.
  
  \item Suppose $\alpha_i \le \alpha_{i+1}$, $a_{b_i+1}=0$, and $h_{b_i+2} = b_i+1$. In this case, there are no floating cells in the region between the bounce points $b_i$ and $b_{i+1}$. Thus, $\gamma^{(i)}$ is a single column of height $u^{(i)} = \alpha_{i+1}$ and this can be moved as shown in \cref{fig:U}.
  \end{enumerate}
\end{rem}

We now have the existence of certain $U_i$ and $D_i$ operators as a
consequence of the remarks above.

\begin{lem}
  \label{lem:existence_of_D}
Let $\pi \in \dycks(n)$ with bounce points $0=b_0,b_1,\dots,b_m=n$
  such that row $b_i+1$ has $b_i-b_{i-1}-1$ cells
  for some $1 \le i \le m-1$. Then, there exists $i \le j \le m-1$ such that
  $\pi \cdot D_j \ne \bot$.
\end{lem}
\begin{proof}
Suppose $h_{b_i+2} = b_{i+1}$. Since row $b_i+1$ has $b_i-b_{i-1}-1$ cells, $d^{(i)} = 0$ and by \cref{rem:simple_D}(1), $\pi \cdot D_i \ne
\bot$.
If $h_{b_i+2} > b_{i+1}$, consider $\pi \cdot D_{i+1}$ instead. 
If $h_{b_{i+1} + 2} = b_{i+2}$, then $\pi \cdot D_{i+1} \ne \bot$. 
Otherwise, check $b_{i+2}$.
If $b_{m-1}$ is reached without success, 
then again by \cref{rem:simple_D}(1), $\pi \cdot D_{m-1} \ne \bot$.
This is illustrated below where $\pi \cdot D_1 = \bot$ since $h_4 > 4$ and
  $\pi \cdot D_2 = \bot$ since $h_6 > 5$; 
  whereas, $\pi \cdot D_3 \ne \bot$.
  \begin{center}
    \begin{tabular*}{0.5\textwidth}{c c}
      \includegraphics[width=0.2\textwidth]{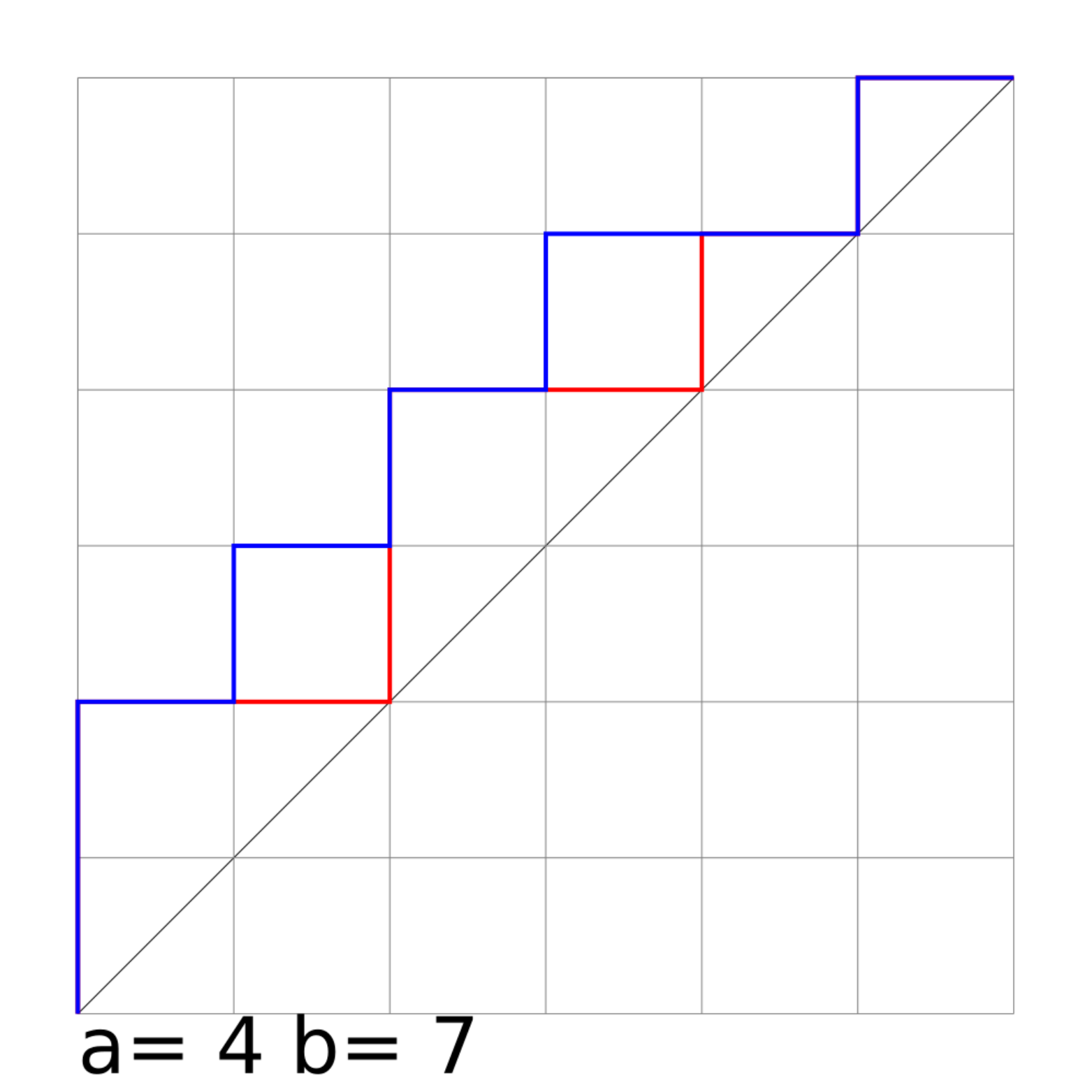} &
                                                                           \includegraphics[width=0.2\textwidth]{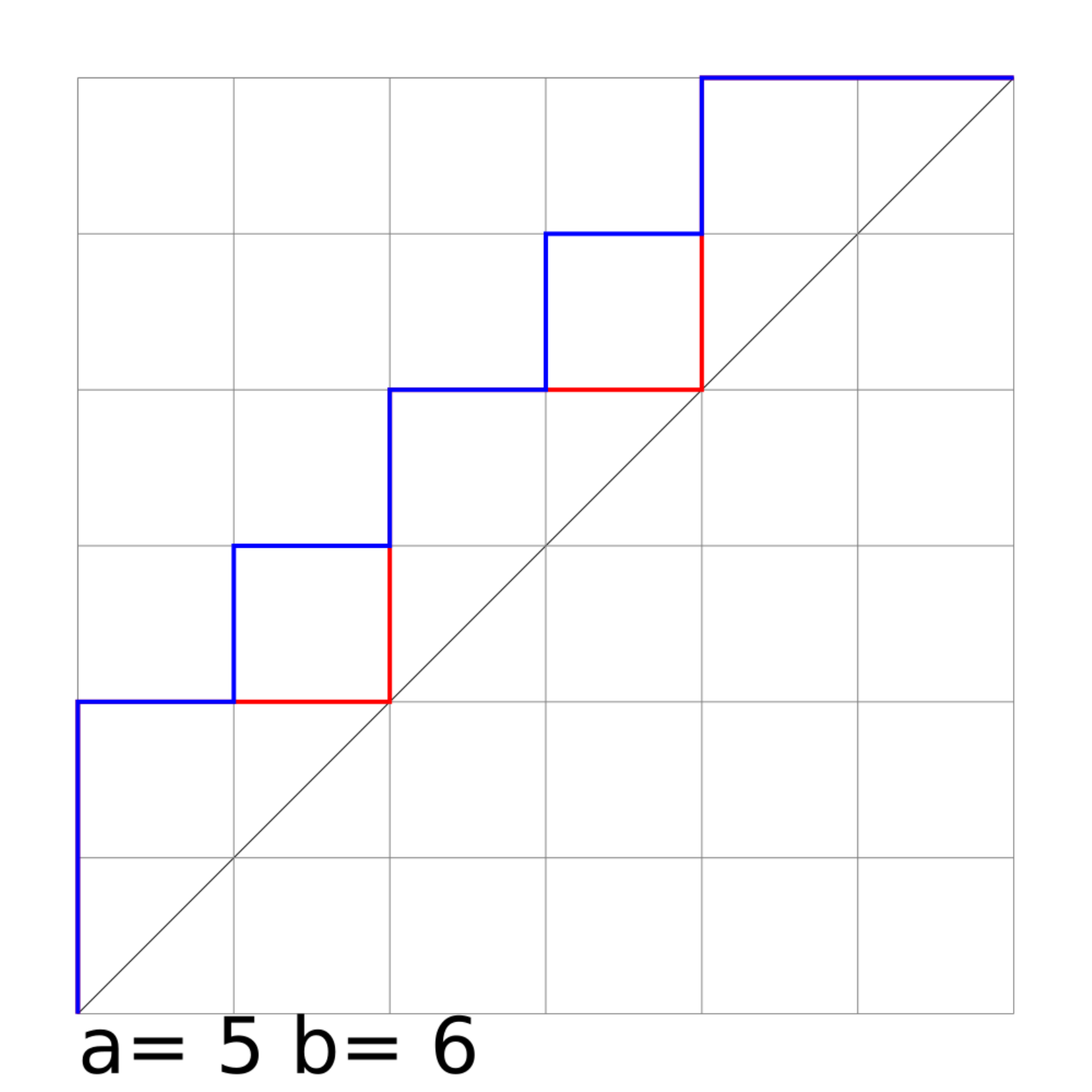} \\
      \CAP{$\pi$} & \CAP{$\pi \cdot D_3$}
    \end{tabular*}
  \end{center}
\end{proof}

\begin{lem}
  \label{lem:existence_of_U}
Let $\pi \in \mathcal{D}(n)$ with bounce points $0=b_0,b_1,\dots,b_m=n$ such
  that column $b_i$ has height $b_{i+1}$ for some $1 \le i \le m-1$. Then there
  exists $1 \le j \le i$ such that $\pi \cdot U_j \ne \bot$.
\end{lem}

\begin{proof}
  Suppose $h_{b_{i-1}} < b_i$. Since column $b_i$ has height $b_{i+1}$, $u^{(i)} = 0$ and by \cref{rem:simple_U}(1), $\pi \cdot U_i \ne
  \bot$. If $h_{b_{i-1}} = b_i$, consider $\pi \cdot U_{i-1}$
  instead. If $h_{b_{i-2}} < b_{i-1}$, then $\pi \cdot U_{i-1} \ne \bot$.
  Otherwise, check $b_{i-2}$. If $b_{1}$ is reached without success then, 
  again by \cref{rem:simple_U}(1), $\pi \cdot U_{1} \ne \bot$.
\end{proof}

The following two lemmas restrict the shape of the bounce path of $\pi$ if
  no $U_i$ or $D_i$ operations on $[\pi]$ lead to a valid path.
Recall that a partition is said to be \emph{strict} if all its parts are distinct.

\begin{lem}
  \label{lem:no_down_conditions}
  Let $\pi \in \dycks(n)$ with bounce path $p_{n,\alpha}$. If for all $\tau \in
  [\pi]$, $\tau \cdot D_i = \bot$ for all $i$, then $\alpha$ is a
  strict partition.
\end{lem}

\begin{proof}
  We show the contrapositive. Assume that $\alpha$ is not a strict partition;
  i.e., there exists $1 \le i < \ell(\alpha)$ such that $\alpha_i \le
  \alpha_{i+1}$. We show that there exists a $\tau \in [\pi]$ where $\tau \cdot D_i
  \ne \bot$. There are two cases.
  If $h_{b_i+2} > b_{i+1}$, then there are at least two more bounce
   points following $b_i$. In particular, row $b_{i+1}+1$ has 
   $b_{i+1}-b_i-1$ cells, and so $\pi$ satisfies
  \cref{lem:existence_of_D}. We can therefore
  conclude that $\pi \cdot D_j \ne \bot$ for some $j \ge i+1$.
  Now suppose $h_{b_i+2} = b_{i+1}$.
  Construct $\tau \in [\pi]$ by rearranging the floating cells between
  $b_i$ and $b_{i+1}$ by first filling row $b_i+1$ maximally then row
  $b_i+2$ and so on.
  This ensures that the partition to delete in $\tau$ is $\gamma^{(i)} = (1)^{d^{(i)}}$
  because $d^{(i)} \le b_{i+1} - h_{b_i}$. A special case of this
  occurs when there are no floating cells to rearrange, in
  which case we can apply \cref{rem:simple_D}(2).
  Thus $\tau \cdot D_i \ne \bot$.
  The figure below shows an example.

  \begin{center}
    \begin{tabular*}{0.7\textwidth}{c c c}
      \includegraphics[width=0.2\textwidth]{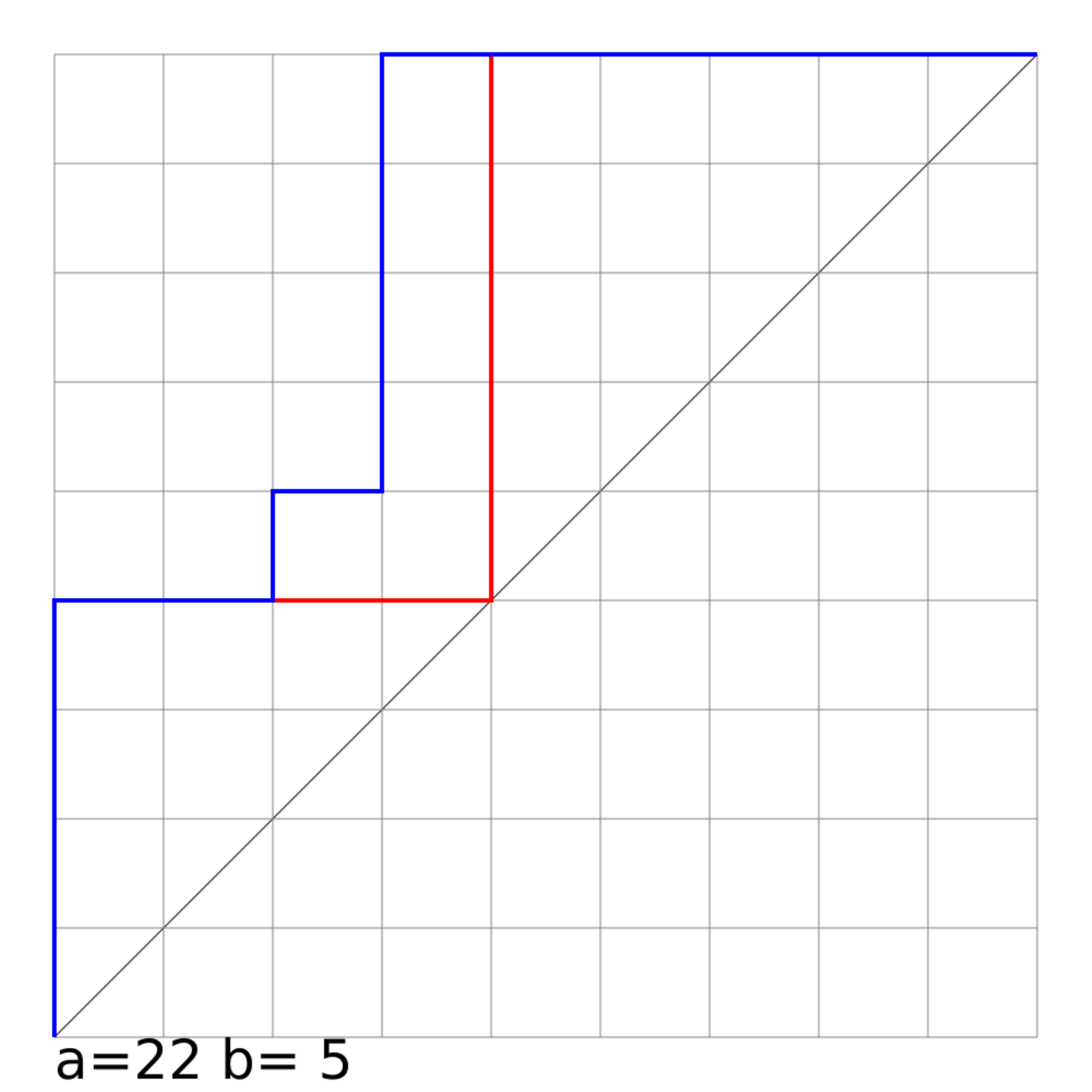} &
      \includegraphics[width=0.2\textwidth]{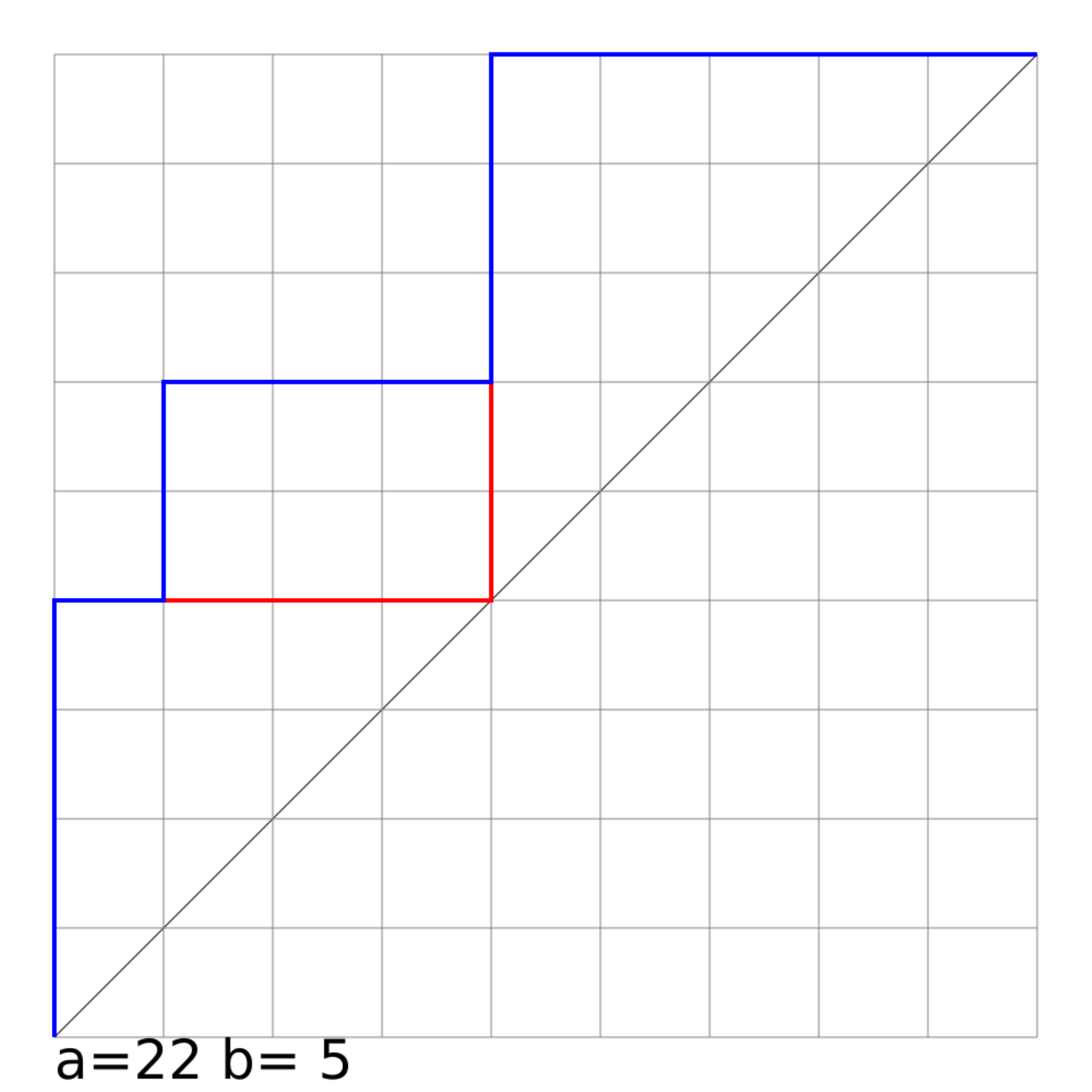} &
      \includegraphics[width=0.2\textwidth]{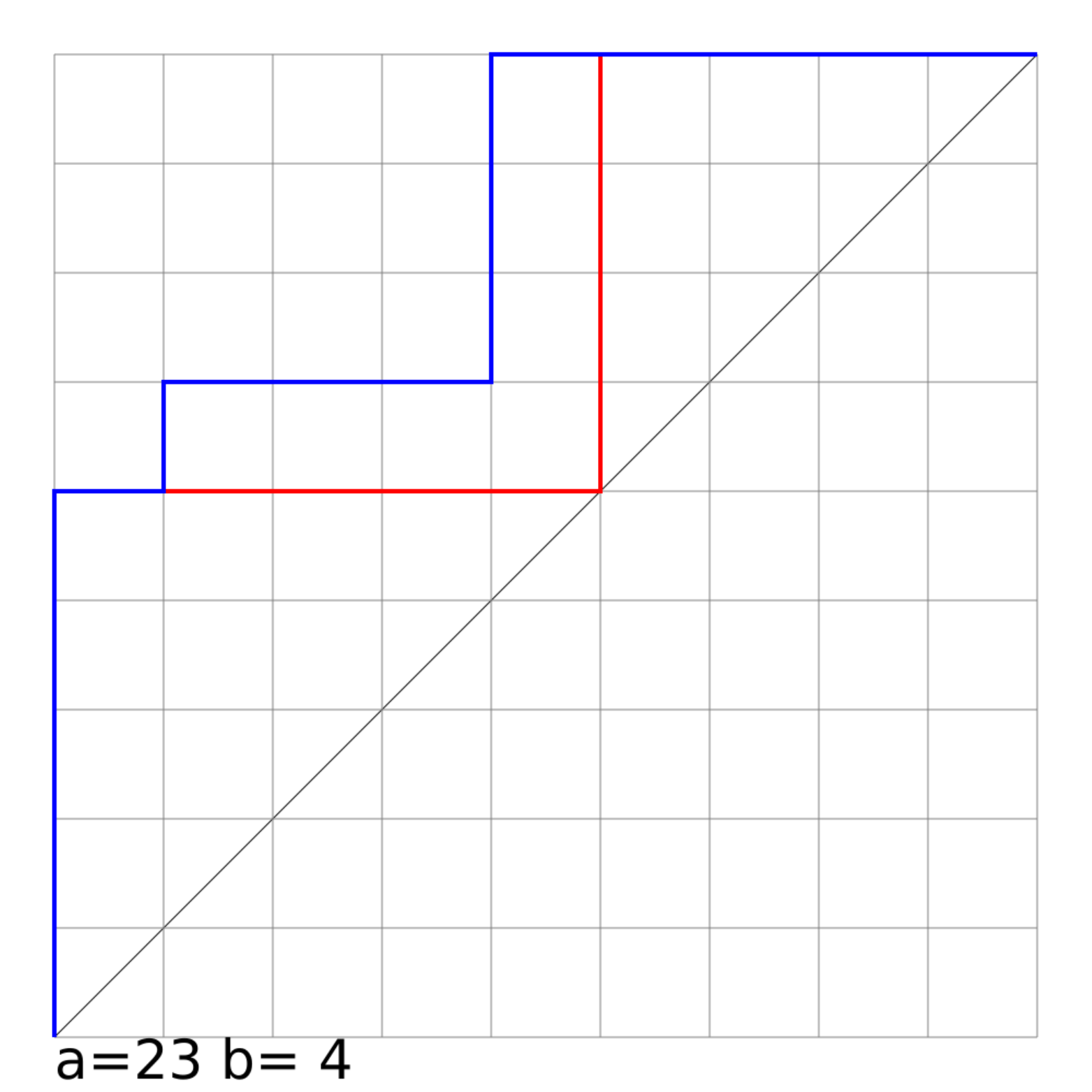} \\
      \CAP{$\pi$} & \CAP{$\tau$} & \CAP{$\tau \cdot D_1$}
    \end{tabular*}
  \end{center}
  Thus $\alpha$ must be a strict partition.
\end{proof}

\begin{lem}
  \label{lem:no_up_conditions}
  Let $\pi \in \dycks(n)$ with bounce path $p_{n,\alpha}$ with $\ell(\alpha) = \ell$. If for all $\tau \in
  [\pi]$, $\tau \cdot U_i = \bot$ for all $i$, then $\pi \in \cdycks(n)$, $\alpha_{\ell}=1$, and
  $\alpha_i - \alpha_{i+1} \le 1$ for all $1 \le i \le \ell(\alpha)-1$.
\end{lem}

\begin{proof}
  We first show that $\alpha_{\ell} = 1$. Suppose that $\alpha_{\ell} > 1$.
  If $h_{b_{\ell-1}} = n$, then by \cref{lem:existence_of_U} there exists a $j \le
  \ell-1$ such that $\pi \cdot U_j \ne \bot$; otherwise,
  $\pi \cdot U_{\ell(\alpha)} \ne \bot$ by \cref{rem:simple_U}(1).

  To show $\pi \in \cdycks(n)$, suppose that $\pi$ contains floating cells.
  We know $\alpha_\ell = 1$.
  So there must be an $\alpha_i > 1$
  such that $\alpha_j = 1$ for all $j > i$. 
  Let $\tau \in [\pi]$ be the path
  containing floating cells at row $b_i+1$.
  If $h_{b_{i-1}} = b_i$, then by \cref{lem:existence_of_U} there exists a
    $j \le i-1$ such that $\tau \cdot U_j \ne \bot$. Otherwise,
  by \cref{rem:simple_U}(1), $\tau
  \cdot U_i \ne \bot$.

Finally, given $\alpha_l = 1$ and $\pi \in \cdycks(n)$, suppose there exists an
$i$ such that $\alpha_i - \alpha_{i+1} > 1$. Then, by \cref{rem:simple_U}(2),
$\pi \cdot U_i \ne \bot$.
\end{proof}

\subsection{Area-minimal and bounce-minimal paths}
\label{sec:area-bounce-minimal}
For $\pi \in \dycks(n)$, set $\ab(\pi) := \area(\pi) + \bounce(\pi)$.
Let $P_n(a, b)$ be the set of Dyck paths $\pi \in \dycks(n)$ such that $\area(\pi) = a$ and $\bounce(\pi) = b$.
Here, we focus on all Dycks paths $\pi$ having a fixed value of $\ab(\pi)$. 

\begin{defn}
  Let $\mathcal{B}(n) \subseteq \dycks(n)$ be the set of \emph{bounce-minimal} Dyck paths; i.e., for every $\pi \in \mathcal{B}(n)$ and $\tau \in \dycks(n)$ such that $\ab(\tau)=\ab(\pi)$, $\bounce(\tau) \ge \bounce(\pi)$.
  Similarly,  let $\mathcal{A}(n) \subseteq \dycks(n)$ be the set of \emph{area-minimal}
  Dyck paths; i.e, for every $\pi \in \mathcal{A}(n)$ and
  $\tau \in \dycks(n)$ such that $\ab(\tau) = \ab(\pi)$, $\area(\tau) \ge \area(\pi)$.
\end{defn}

The next two lemmas describe properties of bounce-minimal and area-minimal Dyck paths.

\begin{lem}
  \label{lem:bounce_minimal_conditions}
  Let $\pi \in \mathcal{B}(n)$ and $p_{n,\alpha}$ be its bounce path.
  Then $\alpha$ is a strict partition and $\area(\pi) \le \area(p_{n,\alpha}) + \min \{ \alpha_i - \alpha_{i+1} - 1 : 1 \le i < n\}$.
\end{lem}
\begin{proof}
Any path $\pi \in \mathcal{B}(n)$ must satisfy $\tau \cdot D_i = \bot$ for all
  $\tau \in [\pi]$ for all $i$; otherwise, $\tau \cdot D_i$ would have a smaller
  bounce than $\pi$. By \cref{lem:no_down_conditions}, $\alpha$ must be a strict
  partition. To show the second condition holds, assume that
  $\area(\pi) - \area(p_{n,\alpha}) \ge \alpha_i - \alpha_{i+1}$ for
  the $1 \le i < \ell(\alpha)$ such that
  $\alpha_i - \alpha_{i+1}$ has the smallest value. 
  If $h_{b_i+2} > b_{i+1}$, then row $b_{i+1}+1$ has $b_{i+1} - b_i - 1$ cells. Therefore, we can apply \cref{lem:existence_of_D} at row $i+1$ to guarantee the existence of some $D_j$ for $j \geq i+1$ such that $\pi \cdot D_j \neq \bot$.
  So assume that $h_{b_i+2} = b_{i+1}$.
  Now construct $\tau \in [\pi]$ by rearranging the floating cells between $b_i$ and
  $b_{i+1}$ by first filling row $b_i+1$ maximally, then row $b_i+2$ and so on. If row $b_i+1$ is full, then by \cref{rem:simple_D}(1), $D_i$ does the job by adding one cell. Otherwise, row $b_i+1$ will have at least $\alpha_i-\alpha_{i+1}$ cells, $d^{(i)}(\tau) \le \alpha_{i+1} - 1$ and $b_i+1 - h_{b_i} = \alpha_{i+1} - 1$. Therefore $\gamma^{i}(\tau) = (1)^{d^{(i)}}$ 
  and $\tau \cdot D_i \ne \bot$, completing the proof.
\end{proof}

\begin{lem}
  \label{lem:area_minimal_conditions}
  Let $\pi \in \mathcal{A}(n)$ with bounce path $p_{n,\alpha}$. Then
  \begin{enumerate}
  \item $\pi \in \cdycks(n)$,
  \item $|\alpha_i - \alpha_{i+1}| \le 1$ for all $1 \le i < \ell(\alpha)$ and $\alpha_{\ell(\alpha)} = 1$, and
  \item for all $i < j < k$, we cannot have $\alpha_i < \alpha_j < \alpha_k$.
  \end{enumerate}
\end{lem}
\begin{proof}
Any path $\pi \in \mathcal{A}_n$ must satisfy $\tau \cdot U_i = \bot$ for all
  $\tau \in [\pi]$ for all $i$; otherwise $\tau \cdot U_i$ would have a smaller area than $\pi$ for some $\tau \in [\pi]$. By \cref{lem:no_up_conditions}, $\pi \in \cdycks(n)$, 
  $\alpha_i - \alpha_{i+1} \le 1$ for all $1 \le i \le \ell(\alpha)-1$ and $\alpha_{\ell(\alpha)} = 1$.

  Suppose $\alpha_i - \alpha_{i+1} < -1$. Then, since there are no floating cells, we can apply \cref{rem:simple_D} to conclude that $\pi \cdot D_i \ne \bot$.
  Let $\pi' = \pi \cdot D_i$.
  Then, $\pi'$
  will have at least two floating cells at column $b_i$.
  If $\alpha' = (\dots,2,1)$,
  we can move one floating cell to row $n$ to create $\tau' \in [\pi']$
  which guarantees
  $\tau' \cdot U_{\ell(\alpha')-1} \cdot U_{\ell(\alpha')} \ne \bot$.
  The following diagram illustrates this case.
  \begin{center}
    \begin{tabular*}{1\textwidth}{c c c c c}
      \includegraphics[width=0.15\textwidth]{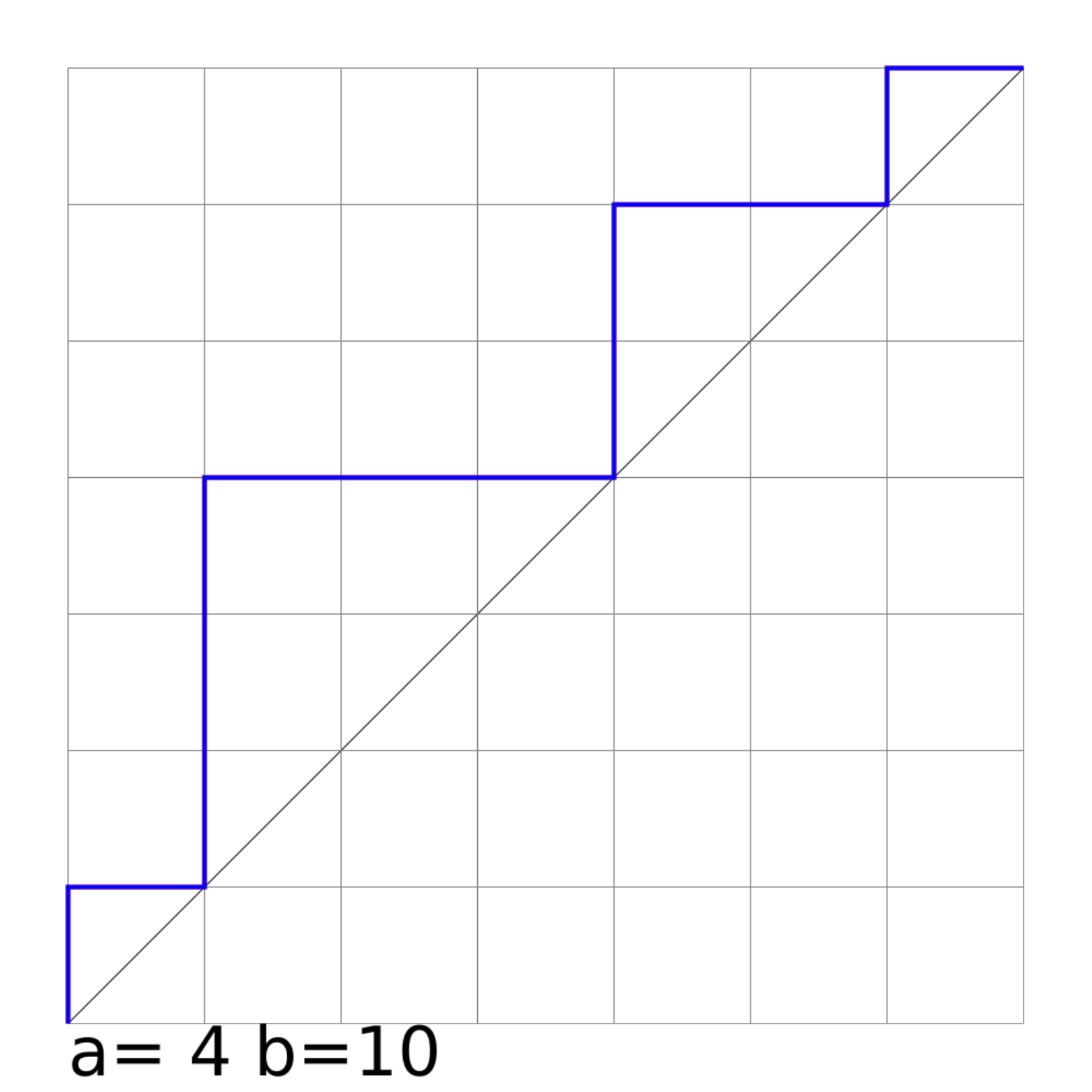} &
      \includegraphics[width=0.15\textwidth]{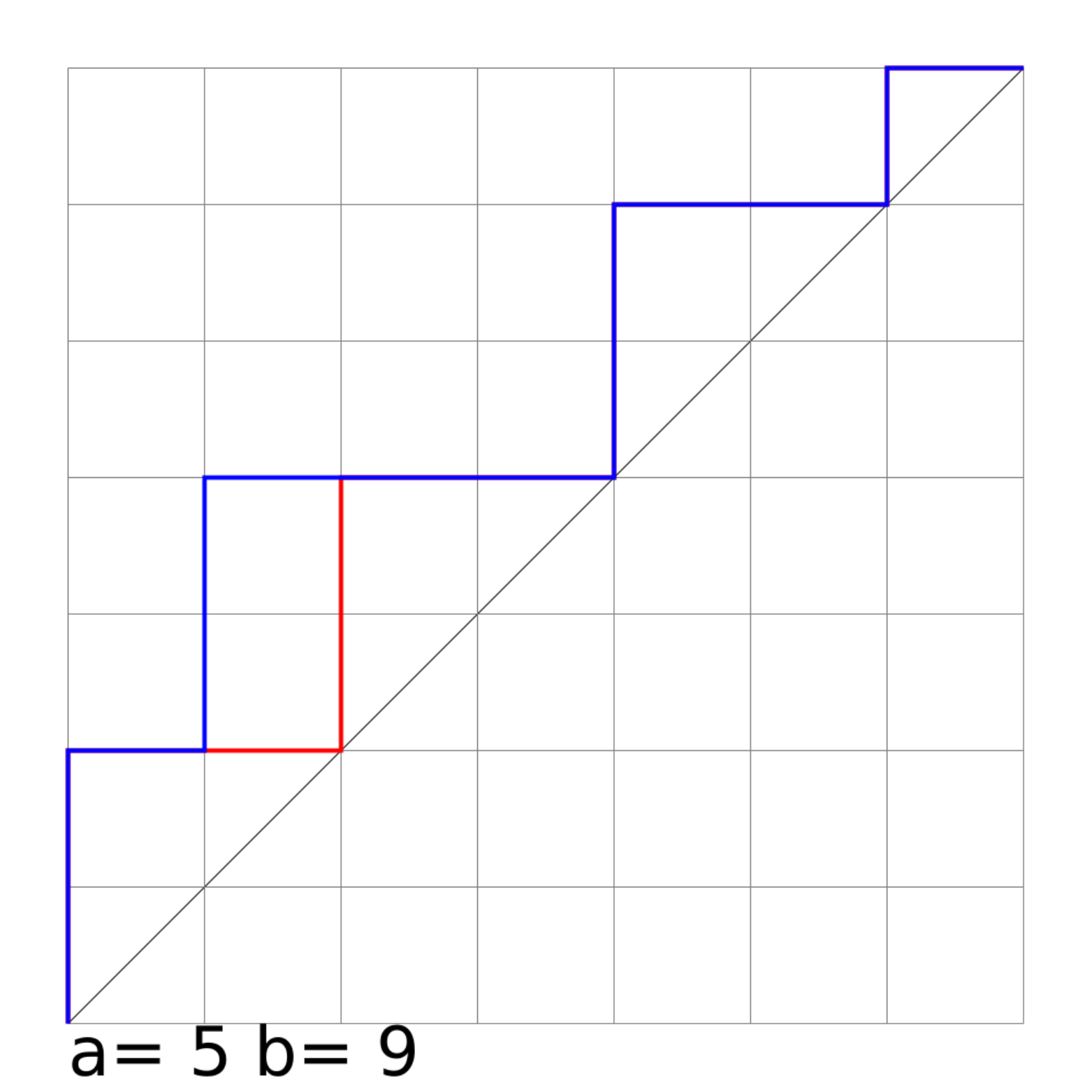} &
      \includegraphics[width=0.15\textwidth]{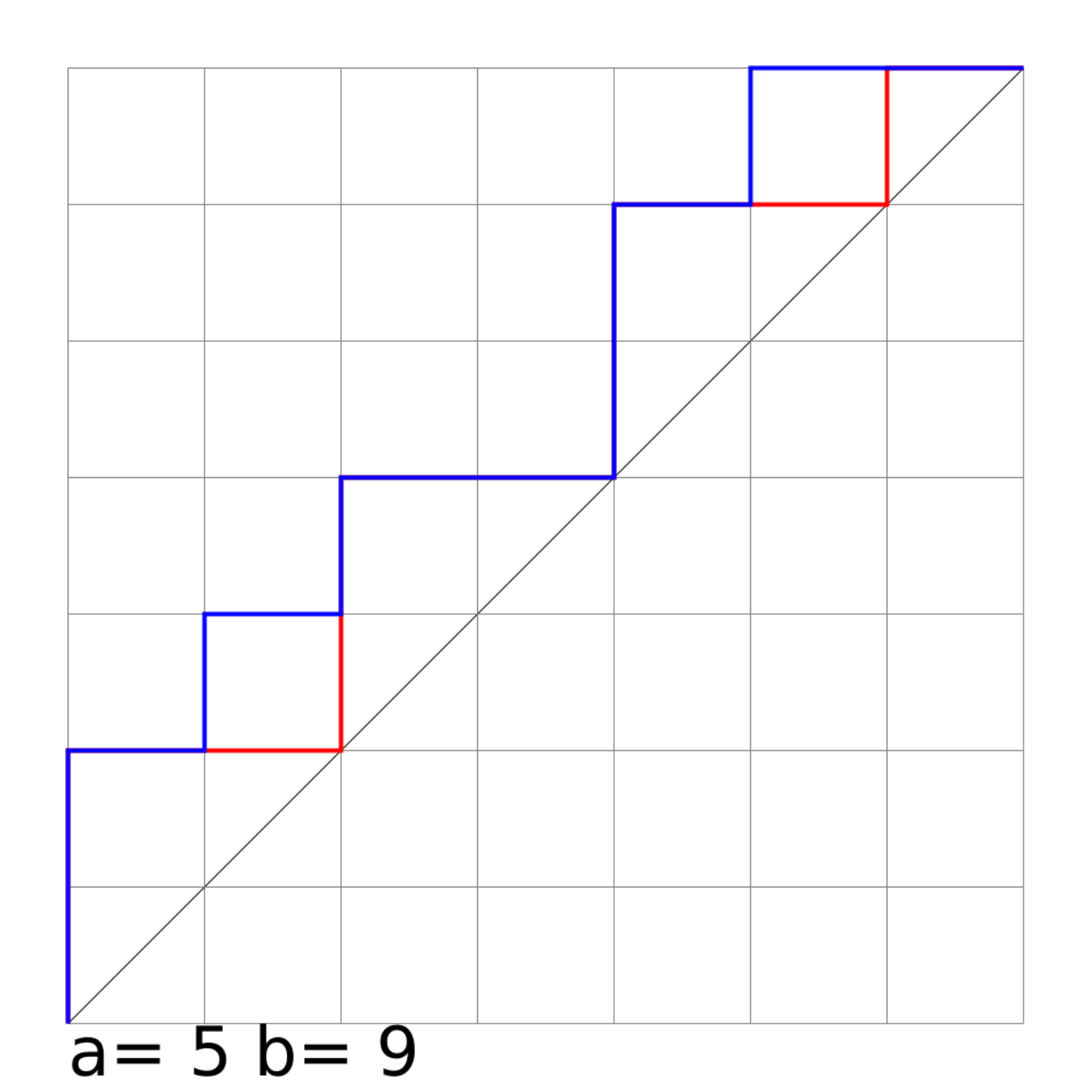} &
      \includegraphics[width=0.15\textwidth]{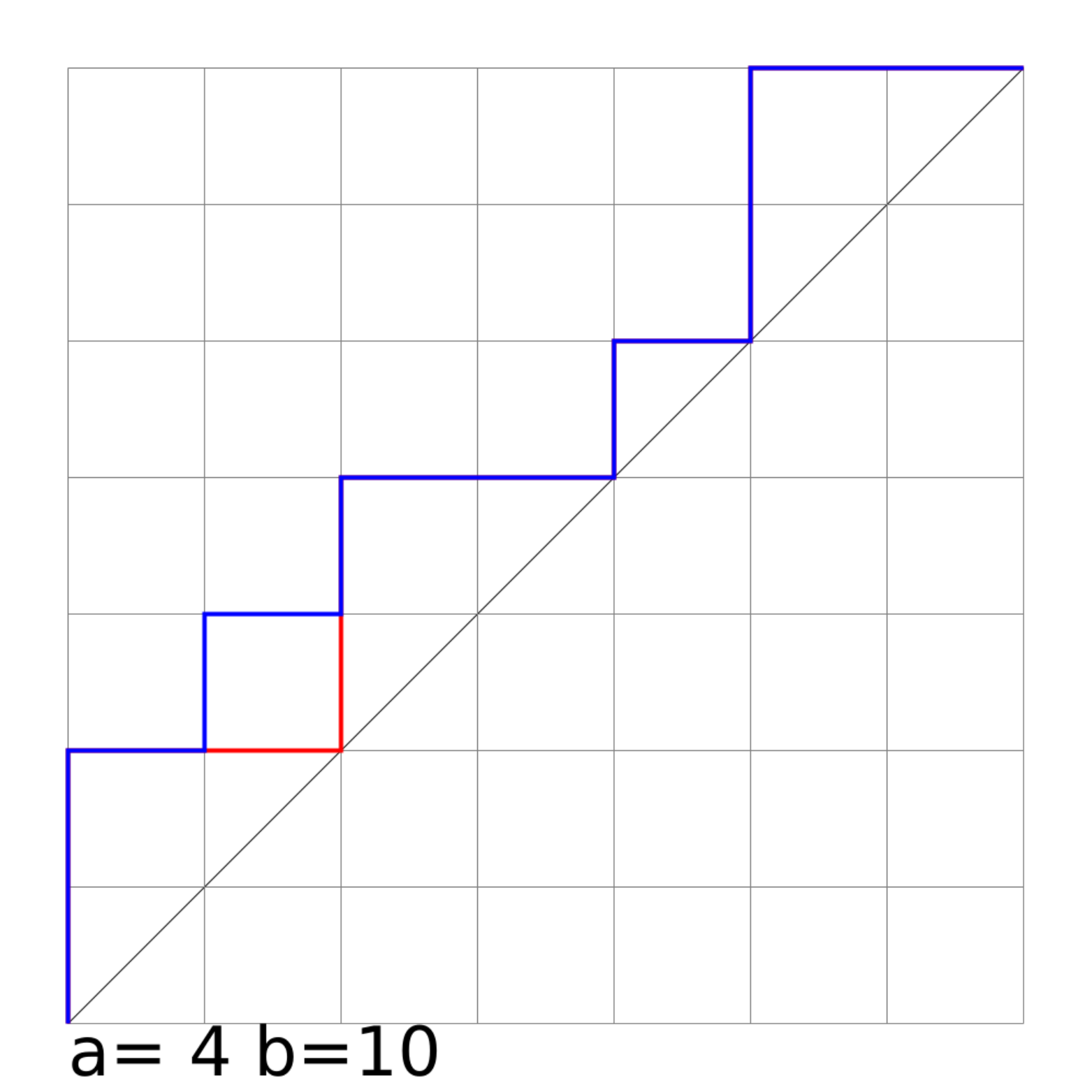} &
      \includegraphics[width=0.15\textwidth]{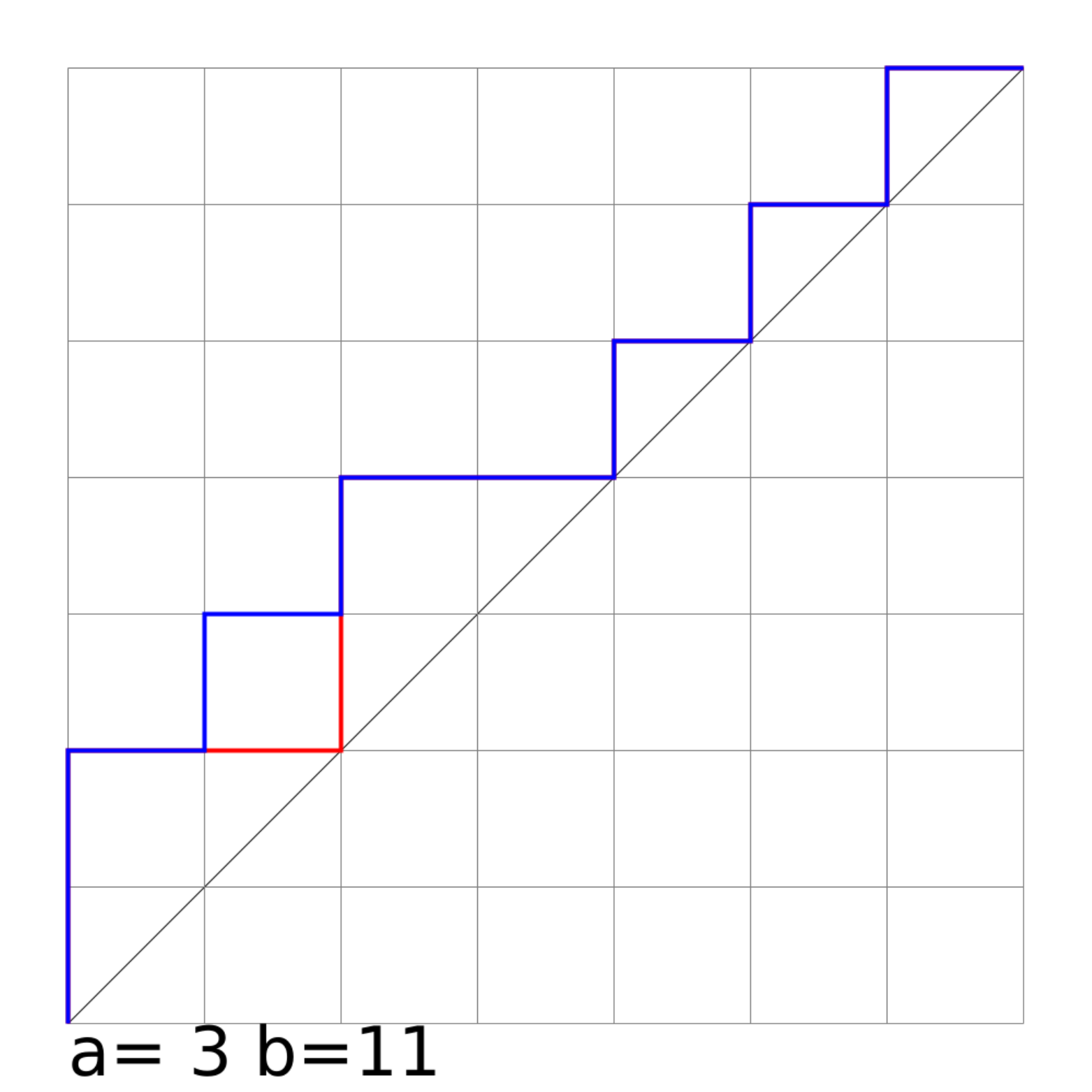} \\
      \CAP{$\pi$} & \CAP{$\pi' = \pi \cdot D_1$} & \CAP{$\tau'$} & \CAP{$\tau' \cdot U_3$} & \CAP{$\tau' \cdot U_3 U_4$}
    \end{tabular*}
  \end{center}
  Otherwise, pick the maximal $i$
  such that $\alpha'_i = 2$ and move a floating cell to row $b_i + 1$ to create
  $\tau' \in [\pi']$.
  Then $\pi'' = \tau' \cdot U_i \ne \bot$.
  We create $\tau'' \in [\pi'']$ by moving another floating cell to $b_{i+1}+1$
  and apply $U_{i+1}$.
  See the following figure.
  Therefore, $(2)$ holds.
  \begin{center}
    \begin{tabular*}{1\textwidth}{c c c c c c}
      \includegraphics[width=0.14\textwidth]{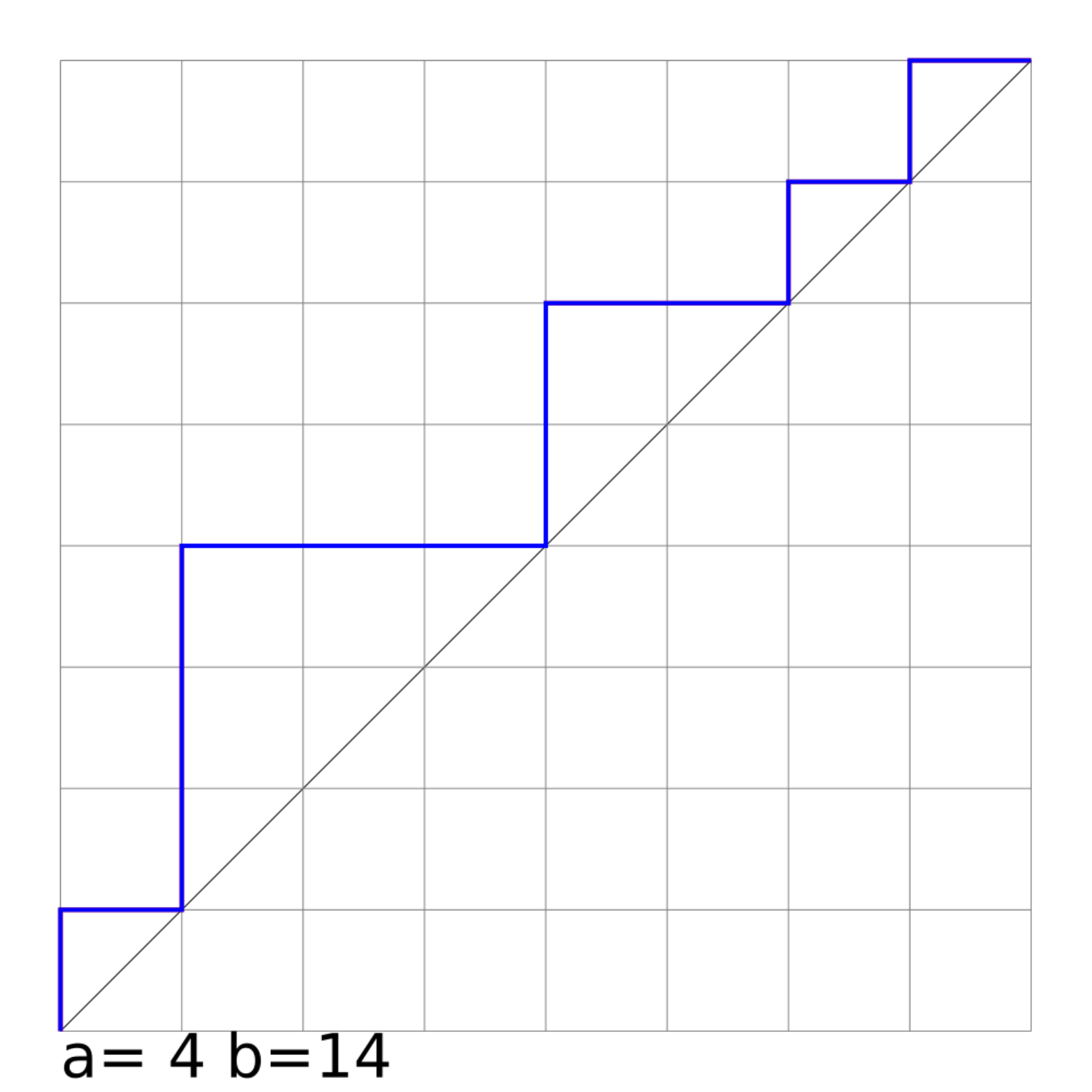} &
      \includegraphics[width=0.14\textwidth]{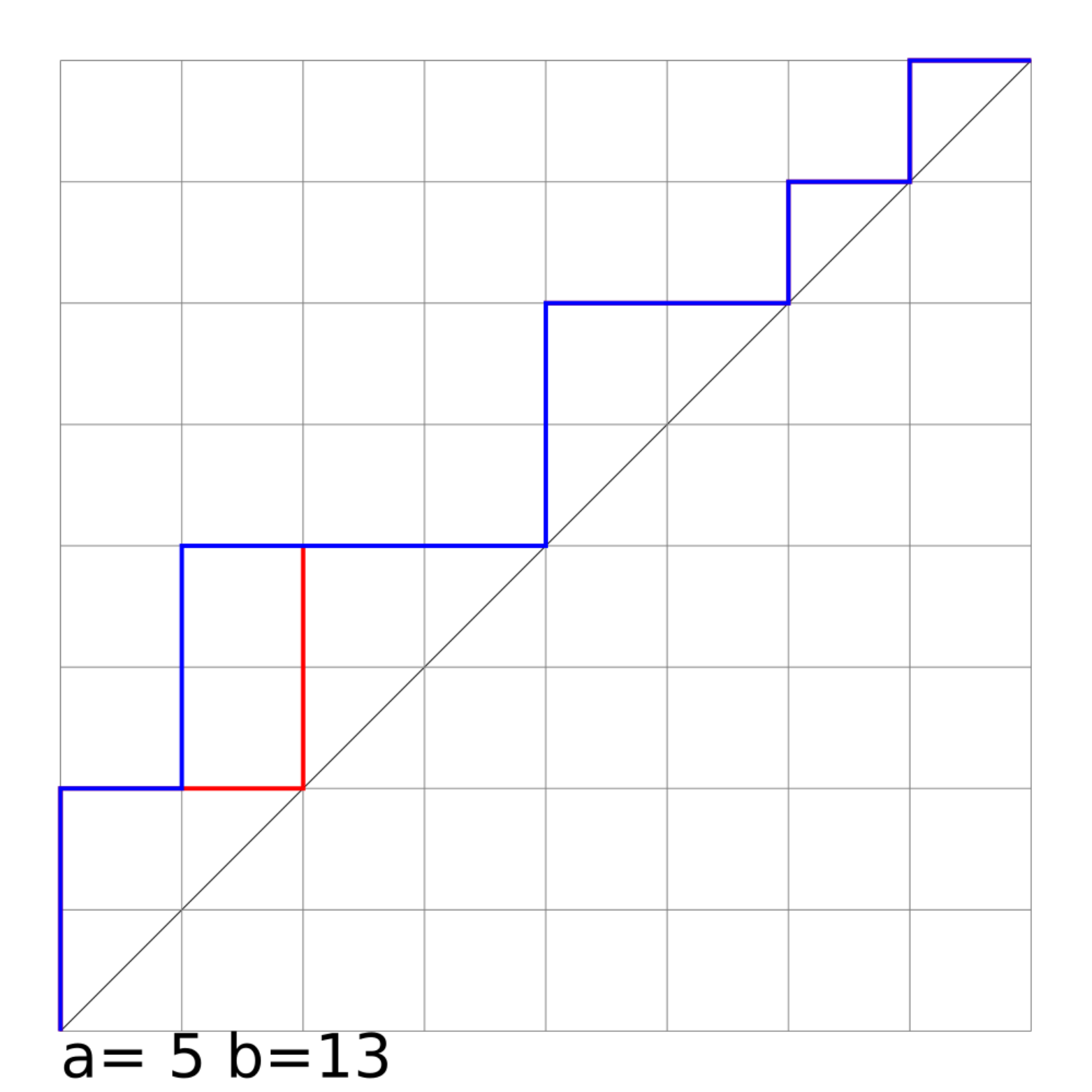} &
      \includegraphics[width=0.14\textwidth]{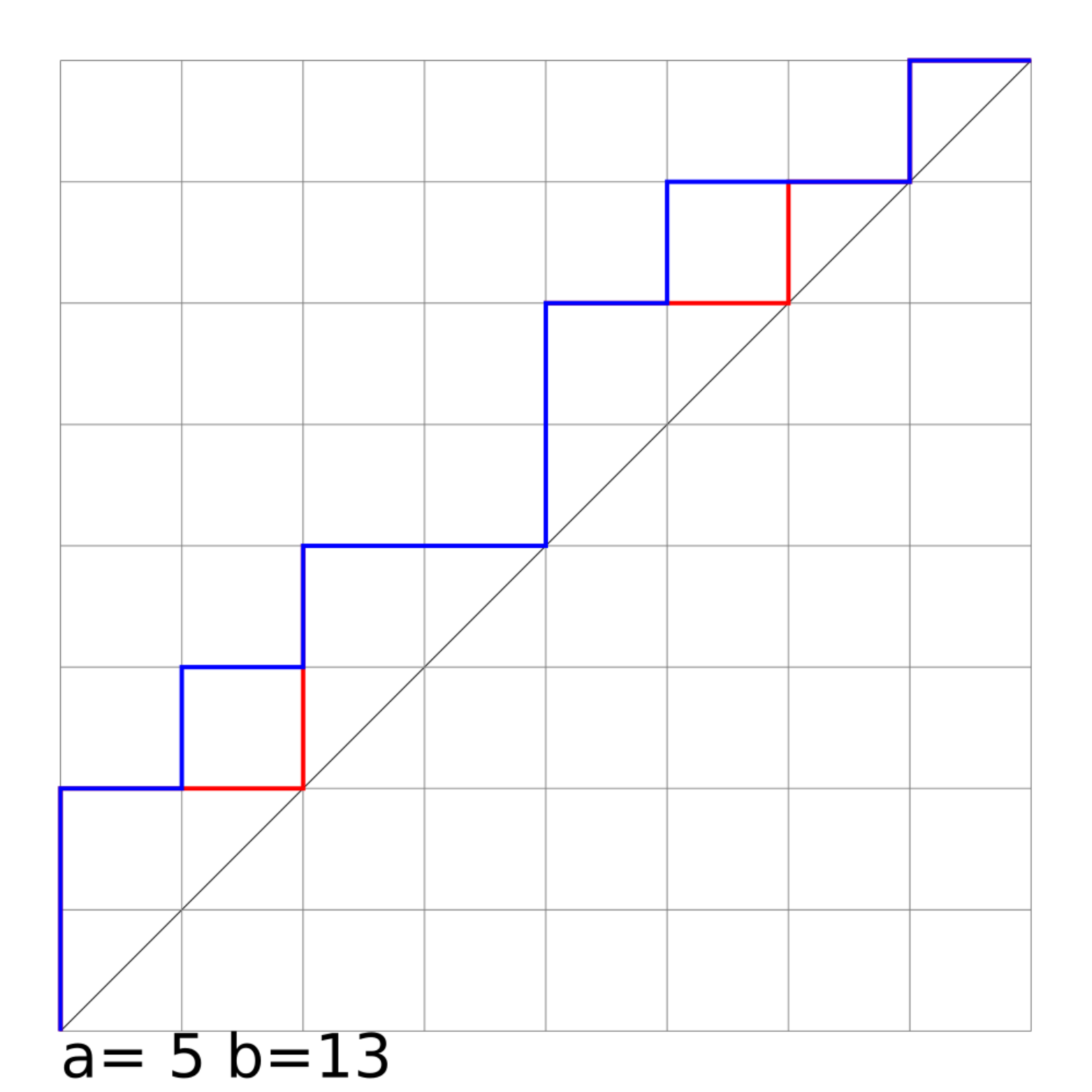} &
      \includegraphics[width=0.14\textwidth]{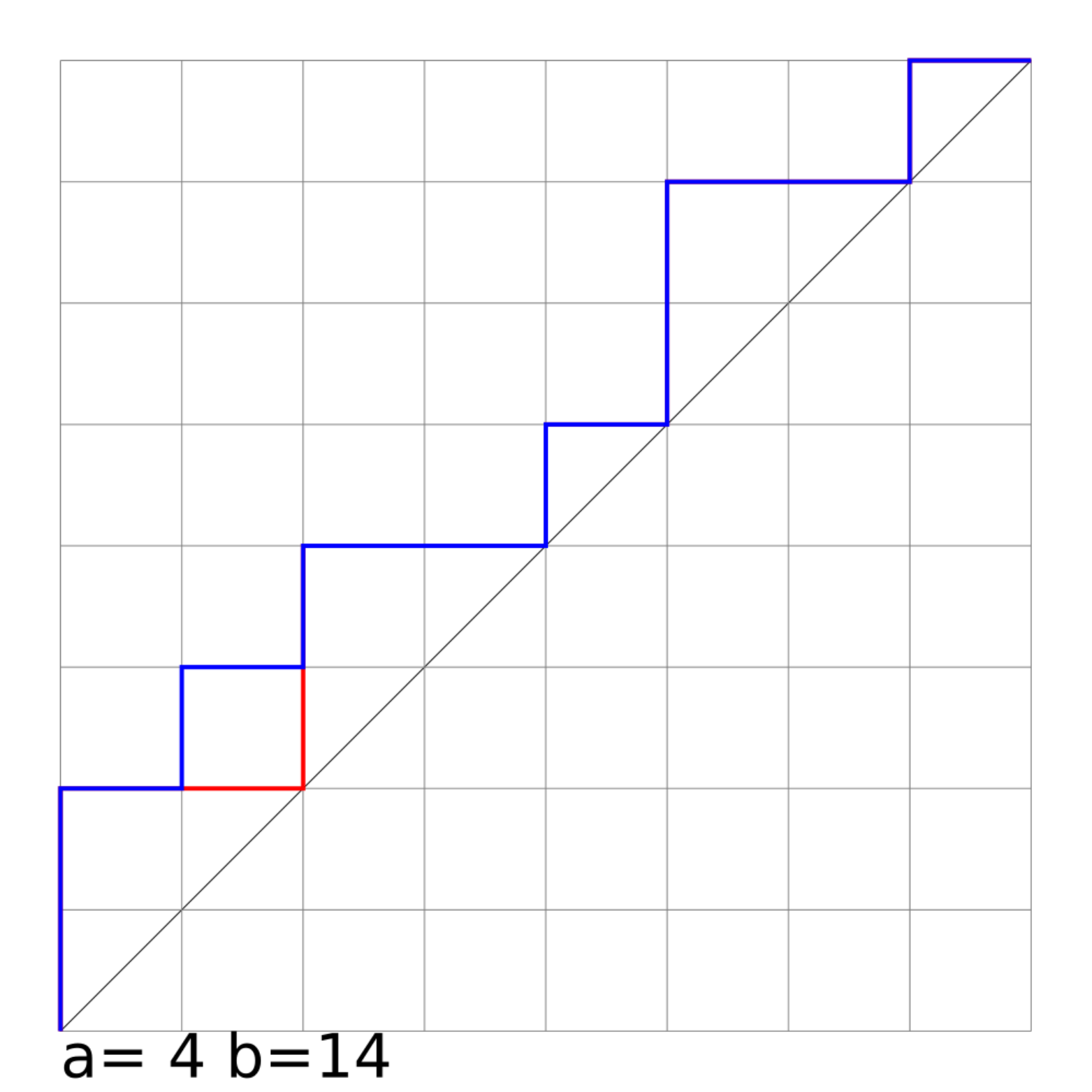} &
      \includegraphics[width=0.14\textwidth]{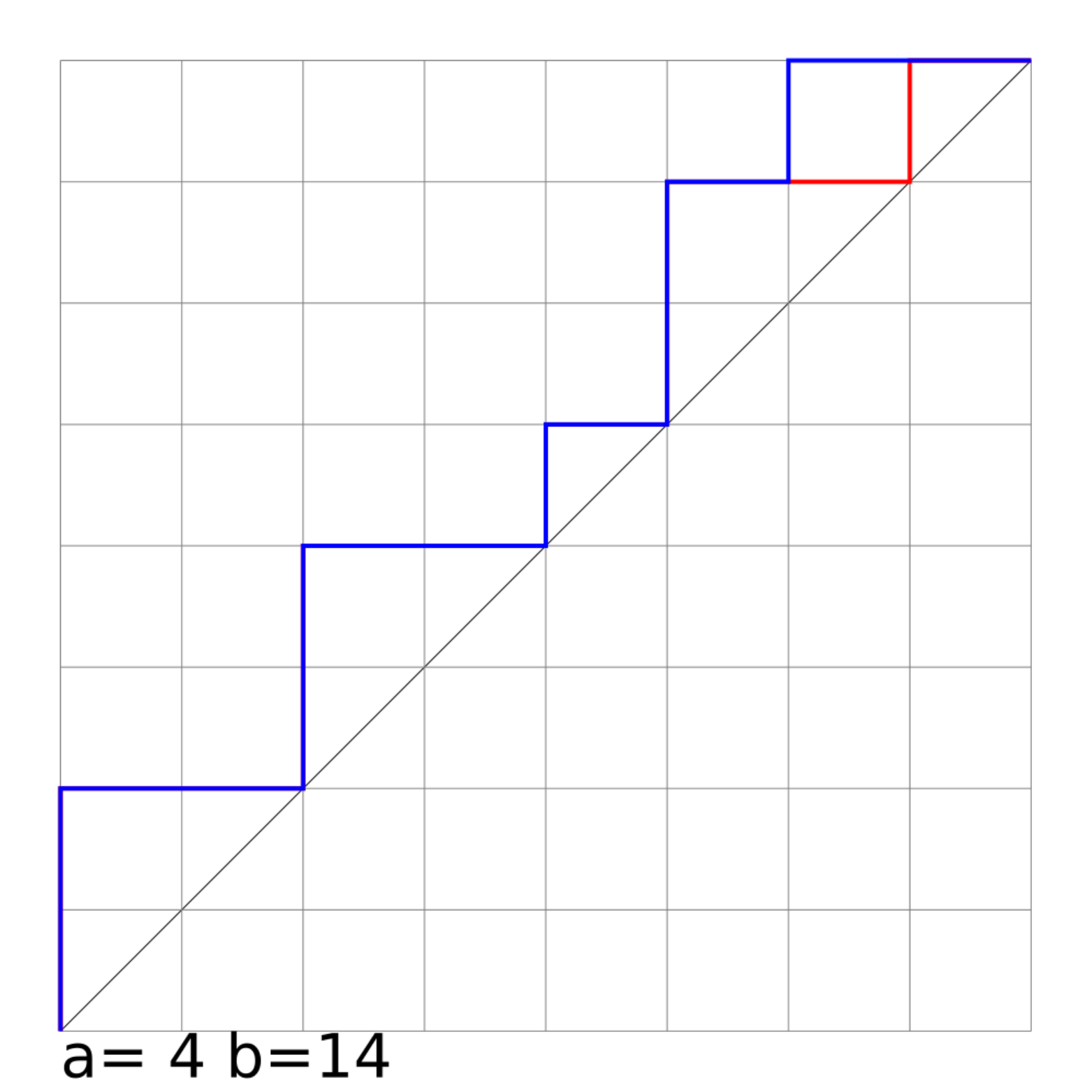} &
      \includegraphics[width=0.14\textwidth]{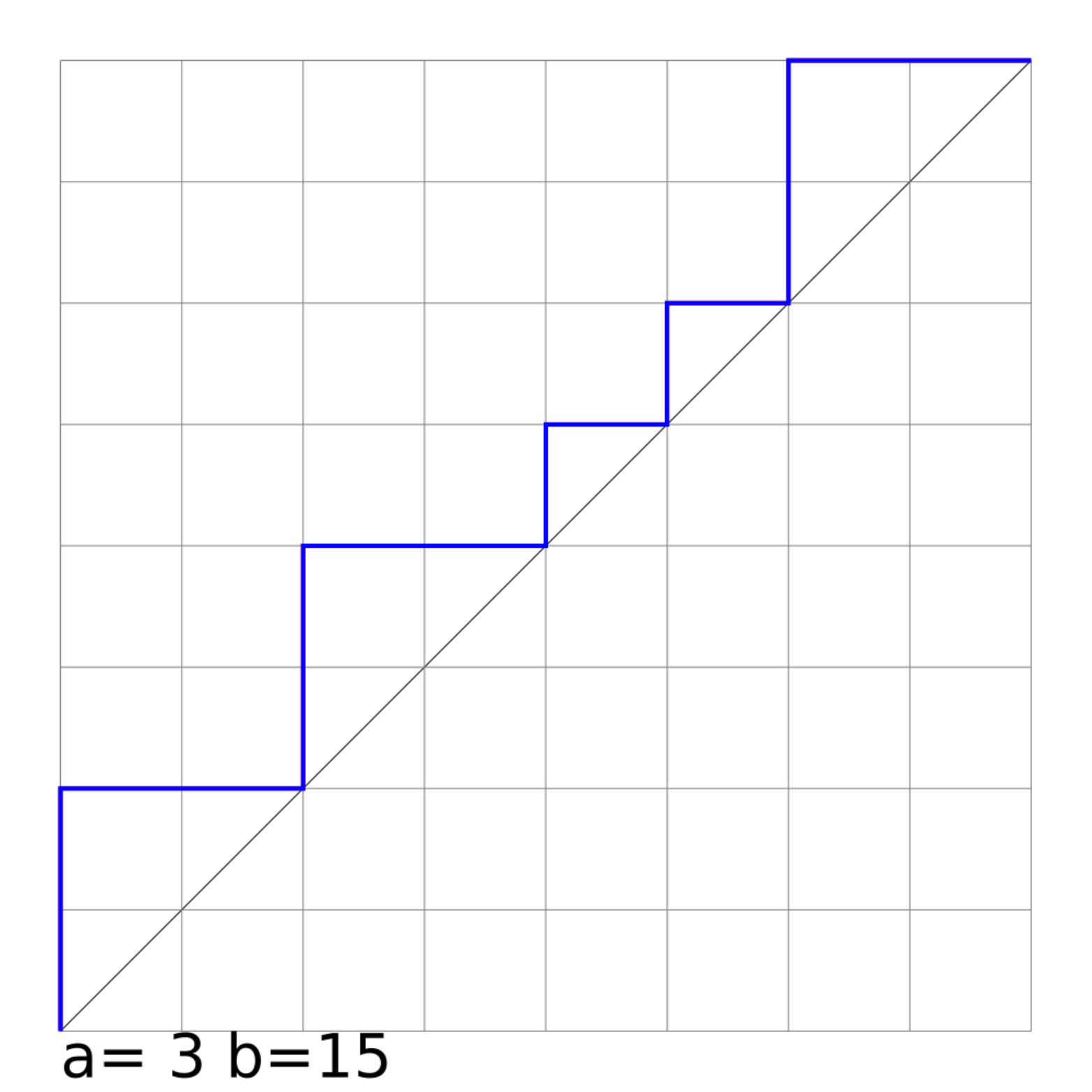} \\
      \CAP{$\pi$} & \CAP{$\pi' = \pi \cdot D_1$} & \CAP{$\tau'$} & \CAP{$\pi'' = \tau' \cdot U_3$} & \CAP{$\tau''$} & \CAP{$\tau'' \cdot U_4$}
    \end{tabular*}
  \end{center}

Finally, suppose $\alpha_i < \alpha_j < \alpha_k$ and $\alpha_i \le \dots \le
  \alpha_k$ for a
  maximal choice of indices $i < j < k$.
  Then, by (2), there exist maximal indices $r < s < t$ for $k \le r$ where
  $\alpha_r = 3, \alpha_s=2$, $\alpha_t=1$, and $\alpha_r \le \dots \le \alpha_t$.
  Construct $\alpha'$ by permuting the indices of $\alpha$ according to the transpositions
\[
    \alpha' =
    \begin{cases}
      \alpha \cdot (r,\ s)(i,\ i+s-r), &\text{if } \alpha_k = 3 \text{ or } s-r \le j-i+2,\\
      \alpha \cdot (r,\ r+j-i+2)(i,\ j+1) &\text{otherwise}.
    \end{cases}
\]

Clearly, $p_{n,\alpha'}$ has the same area as $p_{n,\alpha}$.
It is a somewhat tedious case exercise to check that the first transposition increases the bounce by $s-r$ (resp. $j-i+2$) and the second decreases the bounce by the same amount in the first (resp. second) case.
Therefore, $p_{n,\alpha'} \in P_n(\area(\pi),\bounce(\pi))$.
For example, if
\[
  \alpha = (2,3,2,3,3,4,4,3,4,3,3,2,2,1,1),
\]
we have $i,j,k = 3,5,7$ and $r,s,t=11,13,15$. Since $s-r = 2 \le j-i+2 = 4$, by the first case we apply the transpositions $(11,\ 13)(3,\ 5)$ to get
\[
\alpha' = (2,3,\underline{3},3,\underline{2},4,4,3,4,3,\underline{2},2,\underline{3},1,1),
\]
where we have underlined the elements being swapped.
We now set $\tau = p_{n,\alpha'} \cdot U_s$ in the first case and $\tau = p_{n,\alpha'} \cdot U_{j+1}$ in the second case. Both lead to a valid path by \cref{rem:simple_U}(2).
\end{proof}

\begin{thm}
\label{thm:area-bounce-minimal}
The map $\Phi$ is a bijection from $\mathcal{B}(n)$ and $\mathcal{A}(n)$.
\end{thm}

\begin{figure}[h!]
  \centering
  \includegraphics[width=0.9\textwidth]{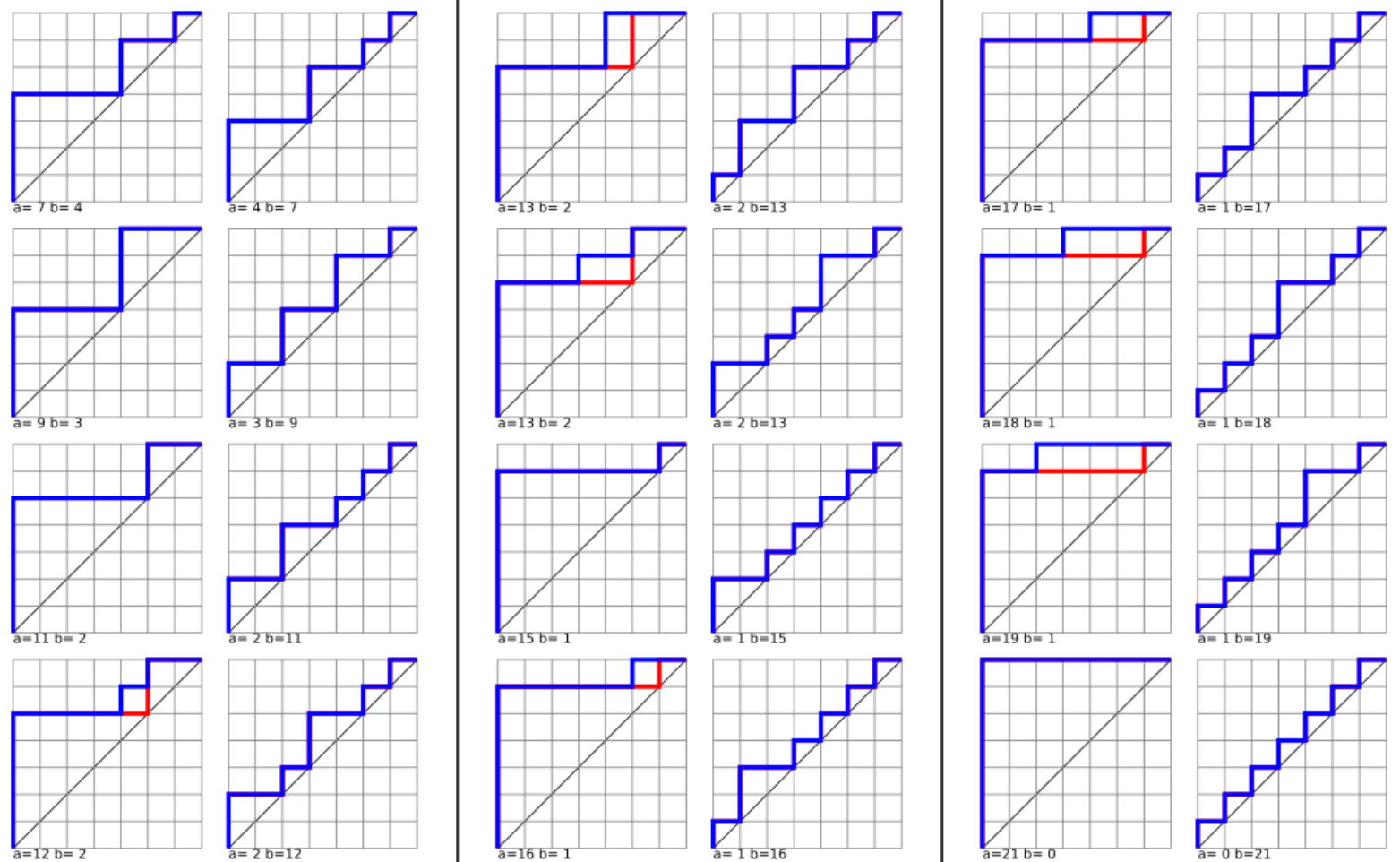}
  \caption{The bijection from $\mathcal{B}(7)$ and $\mathcal{A}(7)$ shown in three columns. In each column, the element in $\mathcal{B}(7)$ is on the left and that in $\mathcal{A}(7)$ is on the right. Note that there are only 11 distinct values of $\ab$; there are two paths each in $P_7(2, 13)$ and $P_7(13, 2)$.}
  \label{fig:minimal_bijection}
\end{figure}

\begin{proof}
  We first show that $\pi \in \mathcal{B}(n)$ implies $\pi \in \AF_n$.
  By \cref{lem:bounce_minimal_conditions}, {the bounce path of $\pi$
  is $p_{n,\lambda}$ for some strict partition $\lambda$}.
  Thus, we have that $\bar{\lambda} = \lambda$ and $\lambda'$ is a
  gapless partition; i.e., $0 \le \lambda'_i - \lambda'_{i+1} \le 1$
  for all $1 \le i < \ell(\lambda')$ and $\lambda'_{\lambda_1} = 1$.
  In addition, $\area(\pi) - \area(p_{n,\lambda}) \le m$ where
  $m = \min \{ \lambda_i - \lambda_{i+1} - 1 \mid 1 \le i < \ell(\lambda)\}$.
  Let $f$ such that $\pi = p_{n,\lambda} \cdot A^f$.
  Then, it is easily seen that $f$ satisfies the conditions in \cref{thm:ab_pairs}; i.e.,
  \begin{enumerate}
  \item $f(i,r) =0$ for $r \ge \lambda_{i+1} - \lambda_{i+2} = d_{\lambda'}(i+1)$ since $f(i,r) = 0$ when $r > m$,
  \item $f(i,1) < \lambda_i - \lambda_{i+1} \le \bar{\lambda}'_i$ since $f(i,1) \le m$.
  \end{enumerate}
  We now show that $p_{n,\lambda'} \cdot B_f \in \cdycks(n)$ by showing that every step of this operation
  leads to a bounce path whose composition only permutes adjacent part sizes in
  $\lambda'$.
  First consider, $\pi^{(1,1)} = p_{n,\lambda'} \cdot B_{\bnc_{\lambda'}(1,1),f(1,1)}$.
  If $f(1,1) > 0$, this is clearly a bounce path since this swaps the last part of size $2$ in $\lambda'$
  with a part of size $1$.
  Assume that $\pi^{(i,r)} \in \cdycks(n)$.
  The next valid operation is either $\pi^{(i,r)} \cdot B_{\bnc_{\lambda'}(i,r+1),f(i,r+1)}$ or
  $\pi^{(i,r)} \cdot B_{\bnc_{\lambda'}(i+1,1),f(i+1,1)}$. 
The former has to be a bounce path since the application of
  $B_{\bnc_{\lambda'}(i,r),f(i,r)}$ was successful and $f(i,r+1) \le f(i,r)$.
  For the latter, we have to show that there are at least $f(i+1,1)$ parts of size $i+1$
  in $\alpha^{(i,r)}$ immediately following the last part of size $i+2$.
  This holds since only a maximum of $m - f(i+1,1)$ parts of size $i+1$ could have
  been moved which is $\le \lambda_{i+1}-\lambda_{i+1}-1-f(i+1,1)$.
  Thus, $p_{n,\lambda'} \cdot B_f \in \BF_n$ and so $\pi \in \AF_n$.

  Now we show that $\pi \in \mathcal{A}(n)$ implies $\pi \in \BF_n$.
  Let $p_{n,\alpha}$ be the bounce path of $\pi$.
  By \cref{lem:area_minimal_conditions} (2) and (3), $\alpha$ is formed from its
  underlying gapless partition, {$\lambda$ (by sorting the parts of $\alpha$
    in decreasing order)},
  by only permuting adjacent part sizes.
  In particular, starting from $\lambda$, a part of size $i+1$ for $1 \le i < \ell(\bar{\lambda})$
  can be swapped with any part of size $i$ except for the last one
  (since that would violate condition $(2)$ in \cref{lem:area_minimal_conditions}).
  Therefore, if $\pi = p_{n,\lambda} \cdot B_f$ for some $f$, then
  $f(i,1) < \bar{\lambda}'_i - \bar{\lambda}'_{i+1} \le \bar{\lambda}'_i$.
  Thus, $p_{n,\lambda'} \cdot A^f \in \AF_n$.

  We can now show that $\Phi$ sends bounce-minimal paths to area-minimal paths.
  Let $\pi \in \mathcal{B}(n)$ satisfy $\ab(\pi) = s$ and $\bounce(\pi) = b$.
  We showed that $\pi \in \AF_n$ and so $\Phi(\pi) \in \BF_n$.
  Then $\ab(\Phi(\pi)) = s$ and $\area(\Phi(\pi)) = b$.
  Suppose there exists $\tau \in \mathcal{A}(n)$ with $\ab(\tau)=s$ and
  $\area(\tau) = a$ such that $a < b$.
  We have shown that $\tau \in \BF_n$ and so $\Phi^{-1}(\tau) \in \AF_n$ with
  $\ab(\Phi^{-1}(\tau)) = s$ and $\bounce(\Phi^{-1}(\tau)) = a$.
  However, this contradicts the bounce-minimality of $\pi$.
\end{proof}

\cref{fig:minimal_bijection} gives an example of this bijection.

\subsection{Paths with intermediate values of area and bounce}
\label{sec:intermediate}

\begin{lem}\label{lem:ab_inequality}
  For all $\tau \in [\pi]$ with bounce path $p_{n,\alpha}$ and $1 \le i \le \ell(\alpha)$,
  \begin{enumerate}
  \item if $\tau \cdot D_i = \bot$ then $\area(\pi) \ge \bounce(\pi)$, and
  \item if $\tau \cdot U_i = \bot$ then $\area(\pi) \le \bounce(\pi)$.
  \end{enumerate}
\end{lem}
\begin{proof}
  Consider the first statement. By \cref{lem:no_down_conditions},
  $\alpha$ must be a strict partition. This condition forces $\area(\pi) \ge \bounce(\pi)$.
  This is seen by induction on the length of $\alpha$ for a bounce path
  $p_{n,\alpha}$.
  When $\ell(\alpha)=1$, $\alpha = (1)$ and $\area(p_{1,(1)}) = \bounce(p_{1,(1)})$.
  Given $\ell(\alpha)=l$, let $\alpha' = (k,\alpha_1,\dots)$ and $n = k+|\alpha|$.
  We have
  \begin{align*}
    \area(p_{n,\alpha'})  = \binom{k}{2} + \area(p_{n-k,\alpha})
    &\ge \binom{\alpha_1+1}{2} + \area(p_{n-k,\alpha})\\
    &\ge \binom{\alpha_1+1}{2} + \bounce(p_{n-k,\alpha})\\
    &\ge n-k + \bounce(p_{n-k,\alpha}) = \bounce(p_{n,\alpha'}).
  \end{align*}
In general, for a path $\pi$ having bounce path $p_{n,\alpha}$, the extra
  floating cells only increase the area, proving the first part.
  
For the second part, \cref{lem:no_up_conditions} implies that $\pi
  \in \cdycks(n)$, $\alpha_i - \alpha_{i+1} \le 1$ for all $i < \ell(\alpha)$ and $\alpha_{\ell(\alpha)} = 1$. 
Thus, the maximum area is attained when $n$ is a triangular number and the reverse of $\alpha$ is of the form $(1, 2, 3, \dots)$. In that case, we know that
$\area(\pi) = \bounce(\pi)$.
\end{proof}

\begin{figure}[h]
  \centering
  \captionsetup[subfigure]{labelformat=empty}
  \begin{subfigure}[t]{0.25\textwidth}
    \includegraphics[width=\textwidth]{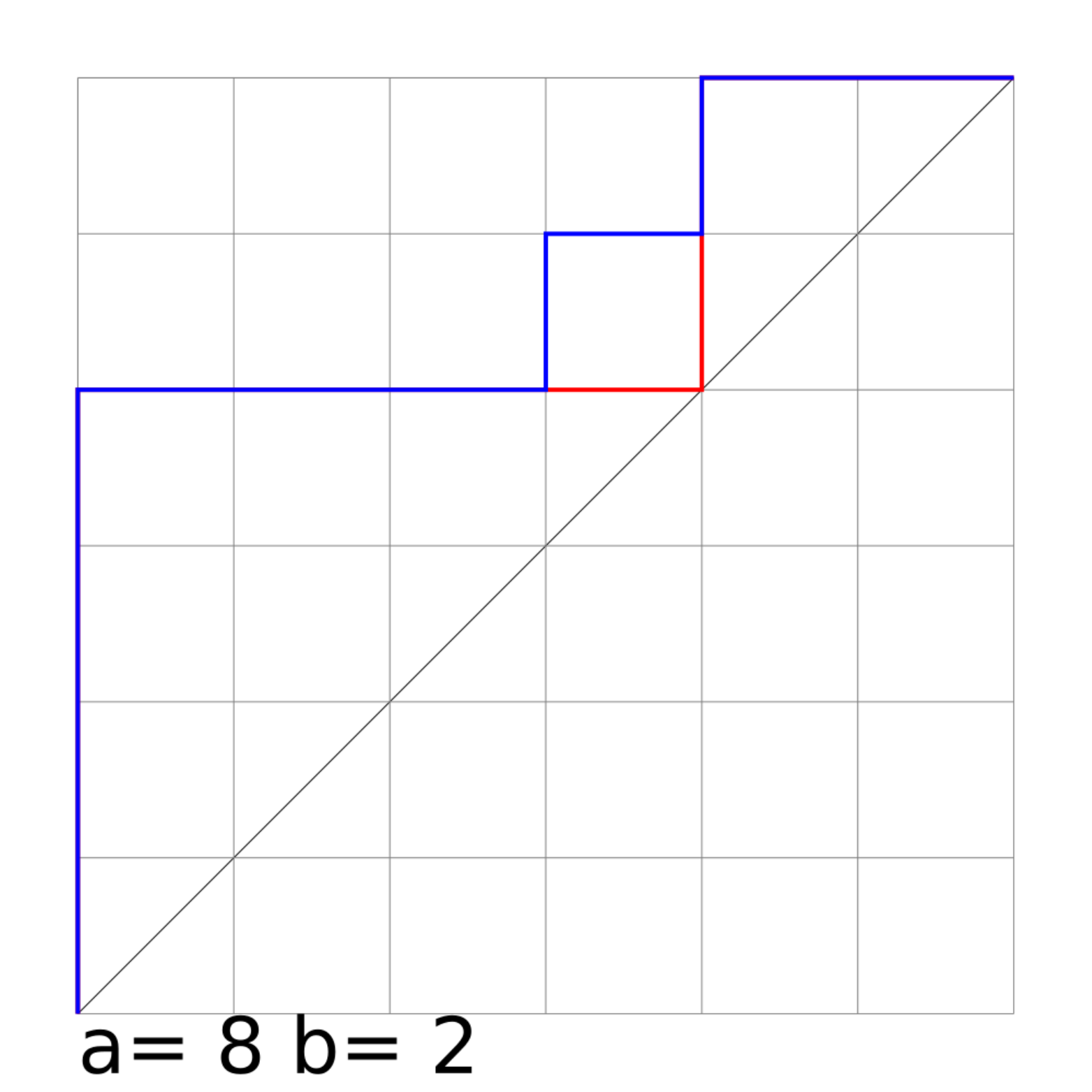}
    \caption{$\pi$}
  \end{subfigure}
  \hspace{10pt}
  \begin{subfigure}[t]{0.25\textwidth}
    \includegraphics[width=\textwidth]{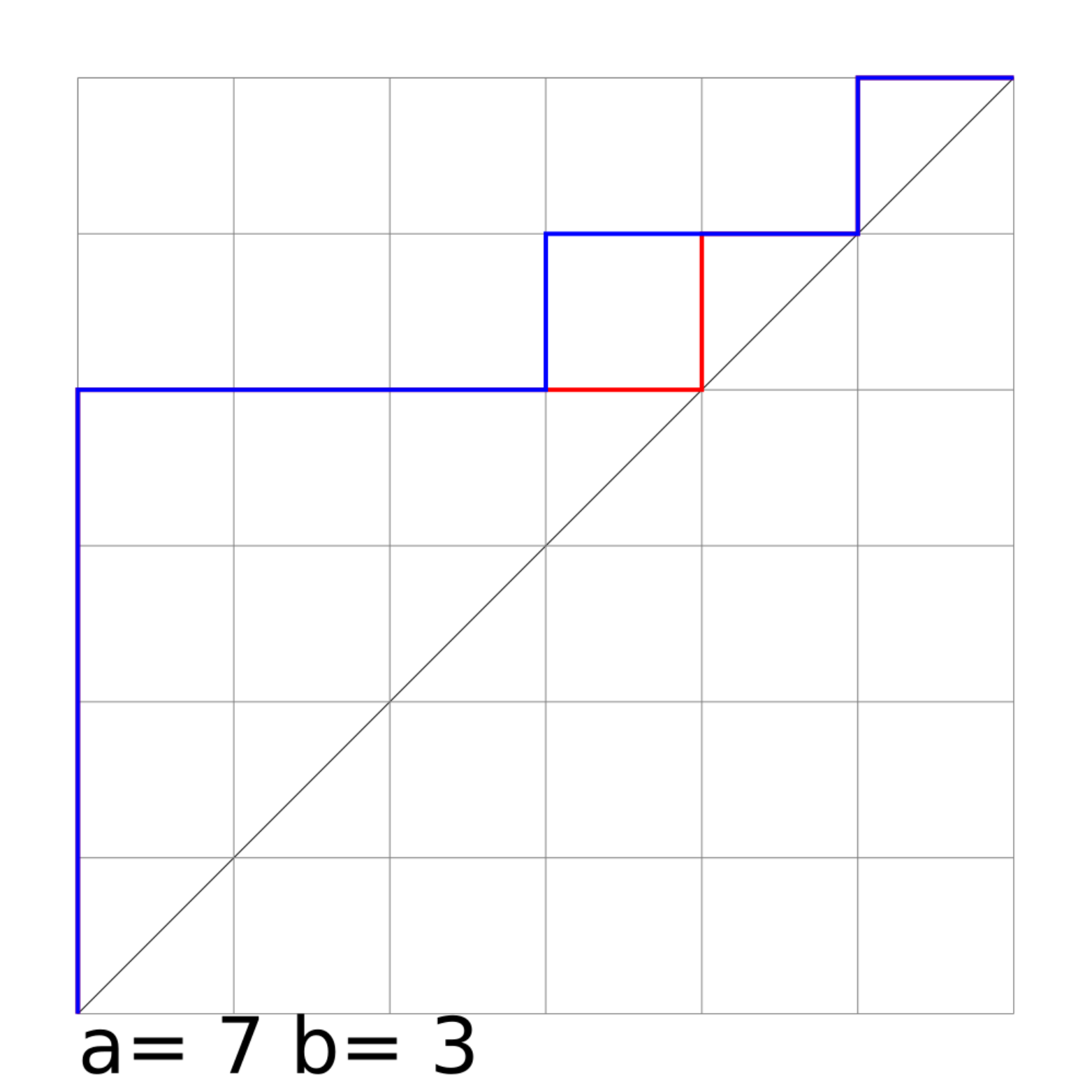}
    \caption{$\tau^{(1)} = \pi \cdot U_1$}
  \end{subfigure}
  \hspace{10pt}
  \begin{subfigure}[t]{0.25\textwidth}
    \includegraphics[width=\textwidth]{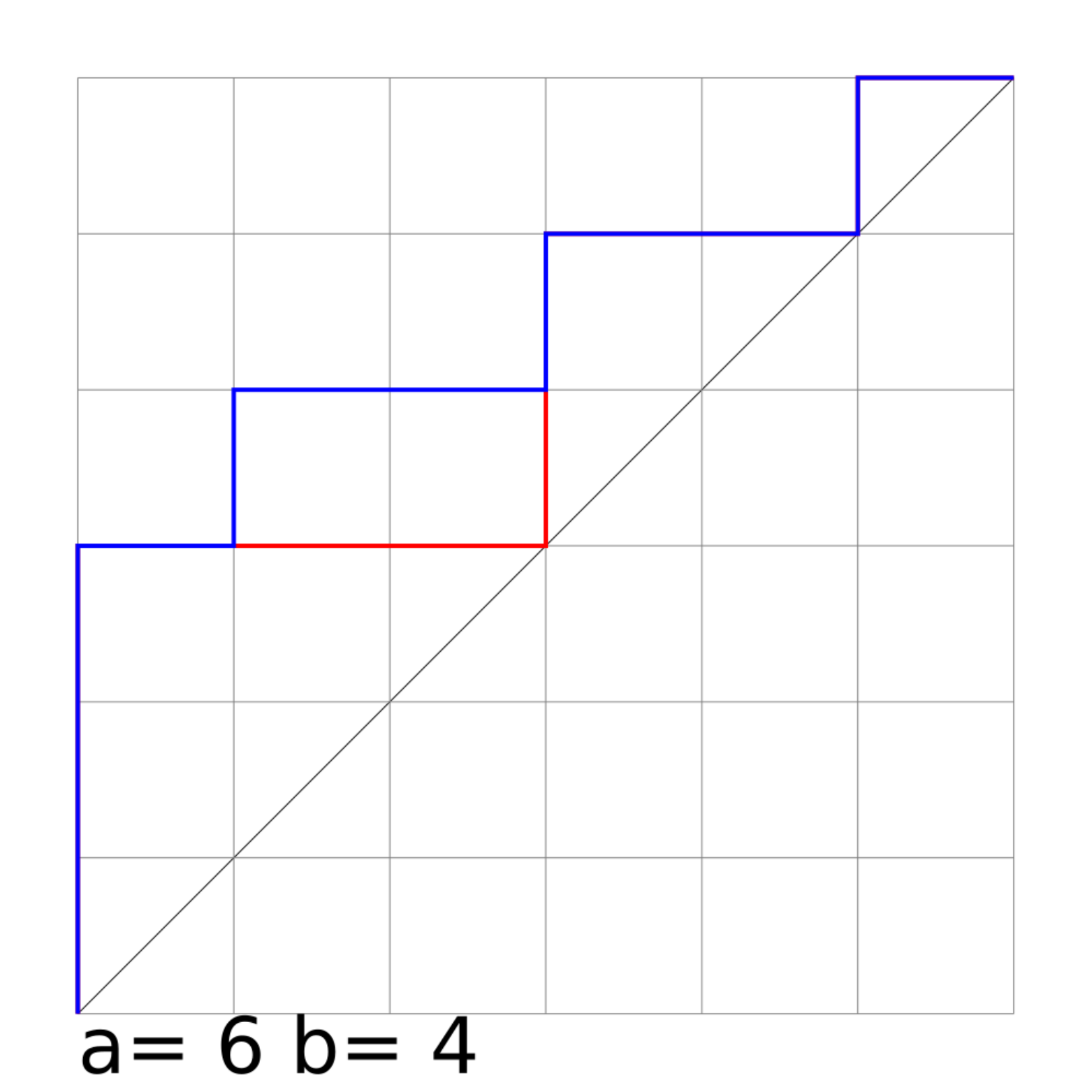}
    \caption{$\tau^{(2)} = \tau^{(1)} \cdot U_1$}
  \end{subfigure}

  \begin{subfigure}[t]{0.25\textwidth}
    \includegraphics[width=\textwidth]{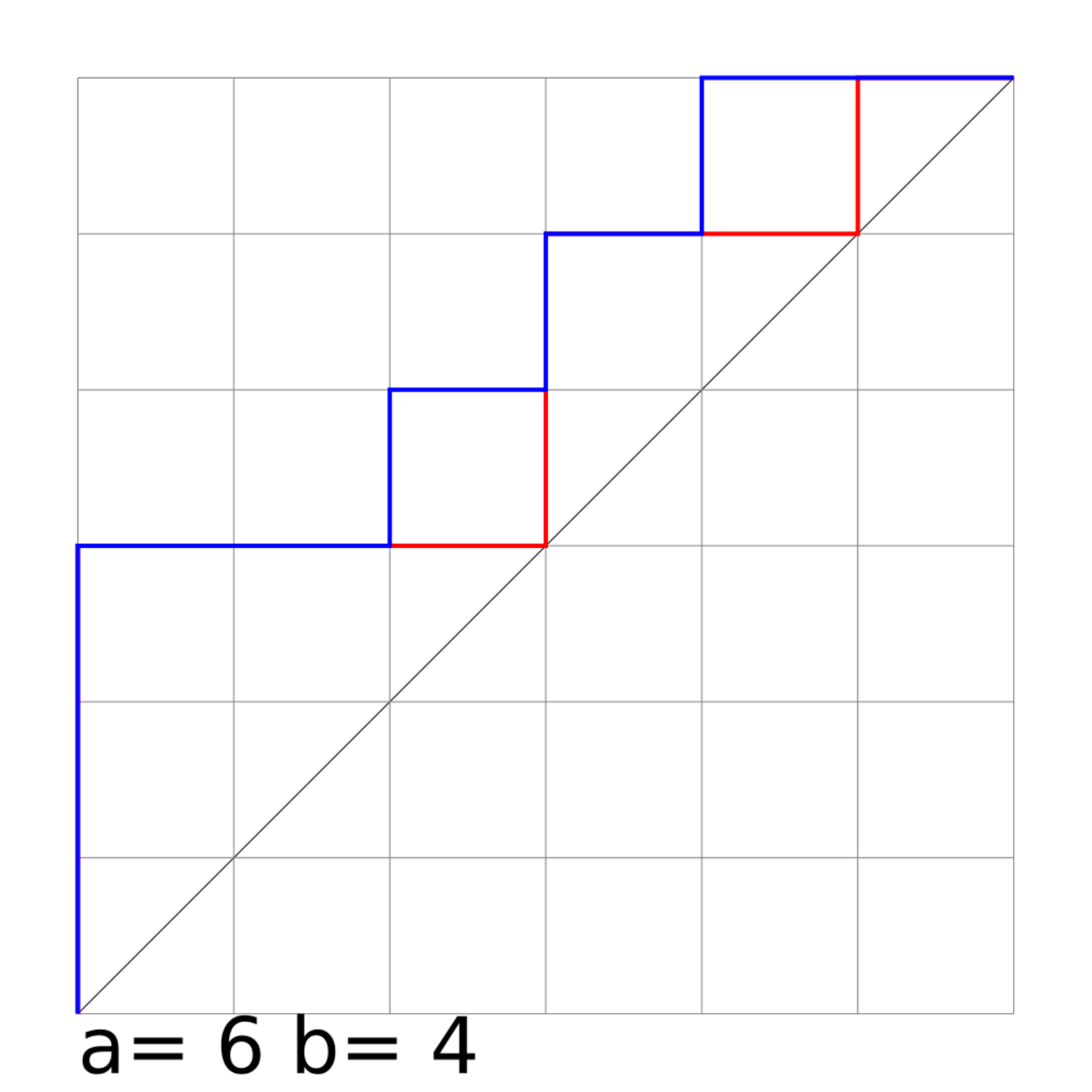}
    \caption{$\tau^{(2)} \cdot A_4^{-1}A_6$}
  \end{subfigure}
  \hspace{10pt}
  \begin{subfigure}[t]{0.25\textwidth}
    \includegraphics[width=\textwidth]{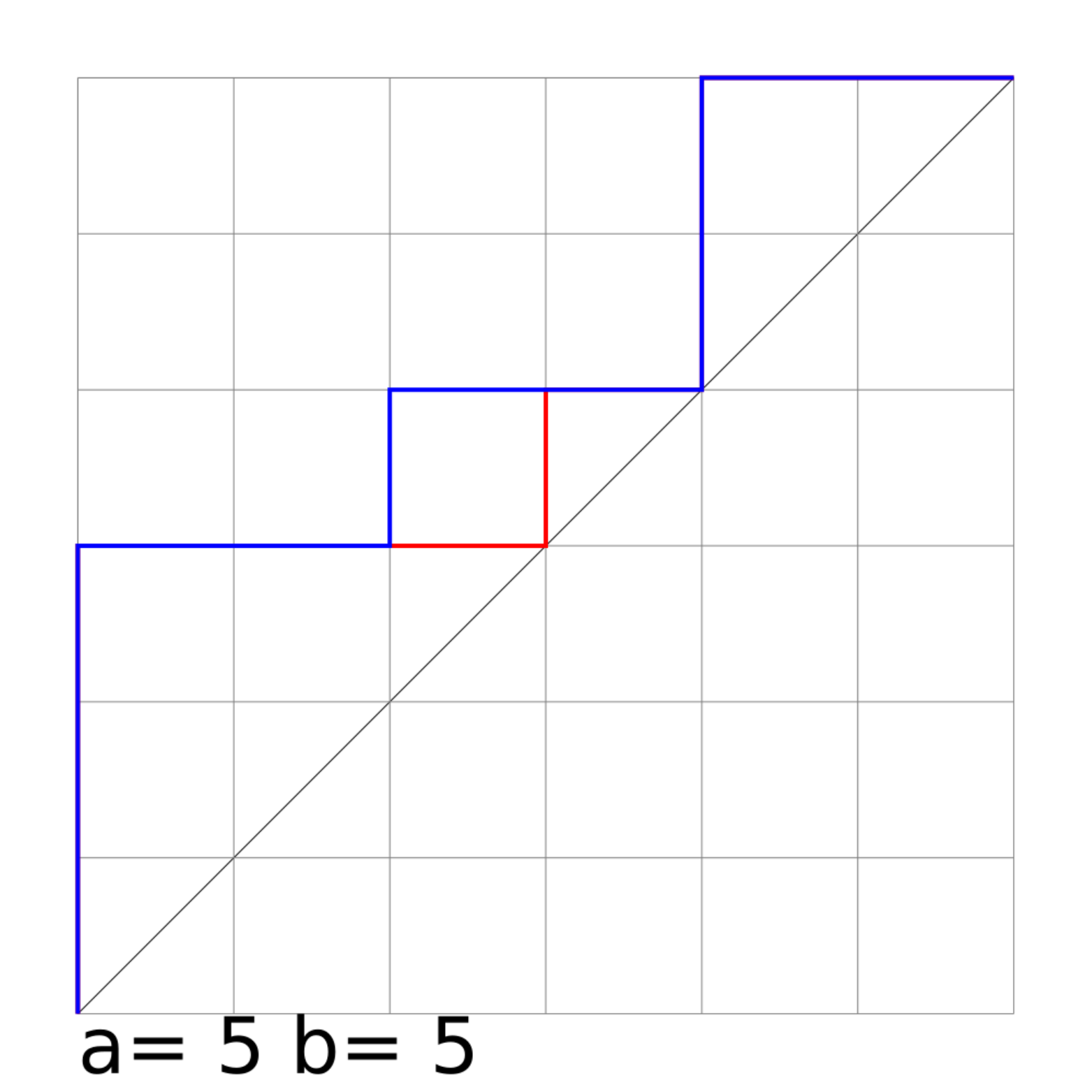}
    \caption{$\tau^{(3)} = \tau^{(2)} \cdot A_4^{-1}A_6U_2$}
  \end{subfigure}
  \caption{Construction of $\pi' \in P_6(5,5)$. We start with the bounce-minimal path $\pi$ to construct the sequence $\tau^{(1)}$, $\tau^{(2)}$, and $\tau^{(3)}=\pi'$.}
  \label{fig:arbitraryPathConstruction}
\end{figure}

\begin{thm}
  \label{thm:degree_symmetry}
  For all $n \ge 1$, $P_n(a,b) \ne \emptyset$ if and only if $P_n(b,a) \ne \emptyset$.
\end{thm}

\begin{proof}
  Suppose we take $a + b = s$ where $\area(\pi) \le a \le s-\area(\pi)$ for
  some $\pi \in \mathcal{A}(n)$ with bounce path $p_{n,\alpha}$.
  We can construct a path $\pi' \in P_n(a,b)$ as follows.
  If $a < b$, by \cref{lem:ab_inequality} $(1)$, we can construct a sequence
  $\pi = \tau^{(0)}, \tau^{(1)}, \dots, \tau^{a-\area(\pi)} = \pi'$
  where $\tau^{(i)} \in [\tau^{(i-1)} \cdot D_j]$ for some $j$.
  To construct $\tau^{(1)}$ for starters, first check if there exists an $i$
    such that $\pi \cdot D_i \ne \bot$. If so, let $\tau^{(1)} = \pi \cdot D_i$.
    Otherwise, since $a < b$, the contrapositive of \cref{lem:ab_inequality}
    says that there exists a path $\pi' \in [\pi]$
    such that $\pi' \cdot D_i$ for some $i$. We let $\tau^{(1)} = \pi' \cdot D_i$.
  Similarly, when $a \ge b$, statement $(2)$ guarantees a sequence
  to a $\pi' \in P_n(a,b)$ starting from $\Phi^{-1}(\pi)$.
  \cref{fig:arbitraryPathConstruction} shows an example construction.
\end{proof}

From the proof of \cref{thm:degree_symmetry}, we have the following corollary.

\begin{cor}
\label{cor:Pn nonempty}
Suppose $\pi \in \mathcal{B}(n)$ has area $a$ and bounce $b$. Then
$P_n(a-i,b+i) \ne \emptyset$ for all $0 \le i \le b-a$. 
\end{cor}

\begin{rem}
We think something stronger than \cref{cor:Pn nonempty} is true. For a fixed value of area plus bounce, say $m$, let $b_{\min}(m)$ be the minimal value of bounce over all paths $\pi$ such that $\area(\pi) + \bounce(\pi) = m$. Then we believe that $P_n(m - b, b) \ne \emptyset$ for $b_{\min}(m) \leq b \leq m - b_{\min}(m)$. See \cref{lem:continuously_increase_ab} in the next section for a result in this direction.
\end{rem}

The following proposition fully characterizes paths in $P_n(a,b)$ where $a+b$
takes on the two largest possible values. Moreover, by \cref{thm:degree_symmetry}
we can construct $P_n(b,a)$.

\begin{prop}\label{prop:top_two_ab}
  For all  $n \ge 3$, let $s=\binom{n}{2}$ and $t=\binom{n}{2}-1$. Then $|P_n(s-i,i)|=1$ for
  all $0 \le i \le s$ and $|P_n(t-j,1+j)|=1$ for all $0 \le j \le t-1$.
\end{prop}

\begin{proof}
  For any $\pi \in \dycks(n)$ with bounce points $0=b_0,b_1,\dots,b_m=n$, the
  number of cells outside the path is at least $\sum_{i=1}^{m-1} (n-b_i)$ because
  the height of column $b_{i-1}+1$ for $1 \le i \le m-1$ must be $n-b_i$.
  Therefore, if $\pi \in P_n(s-i,i)$ it must have exactly $n-m+1$ columns of
  full height. 
  There is only one such path and it has bounce path $p_{n,\alpha}$ such that
  $\alpha = (1,1,\dots,1,k,k')$ and has maximal floating cells.
  For instance, the fifth path in the top row of \cref{fig:top_two_ab} has
  $\alpha=(1,2,1)$, where $k=2$ and $k'=1$, 
  and contains all possible floating cells.

If $\pi \in P_n(t-j,1+j)$, it must have exactly $n-m-1$ columns of full height
  and exactly one column of height $n-1$. There is only such path here as well.
  We can get this path in exactly one of two ways. Start with the sole path
    $\tau \in P_n(\binom{n}{2}-j,j)$ for $1 \le j \le \binom{n}{2}-2$.
  If there are floating cells in this path, there is only one
  way to remove exactly one floating cell: $\pi = \tau \cdot A_n^{-1}$. Otherwise,
  there are no floating cells, $b_1=1$ and $b_{m-1} \le n-3$. In this case,
   $\pi = \tau \cdot A_n^{-2}A_{b_{m-1}+1}$. For instance, the first
    path in the bottom row of \cref{fig:top_two_ab} is constructed from the
    second path in the top row by removing a floating cell. The third path in
    the bottom row is constructed by from the path above it by
    deleting two cells in the top row and adding one cell to the second row.
\end{proof}

\begin{figure}[h!]
  \centering
  \begin{subfigure}[t]{0.13\textwidth}
    \includegraphics[width=\textwidth]{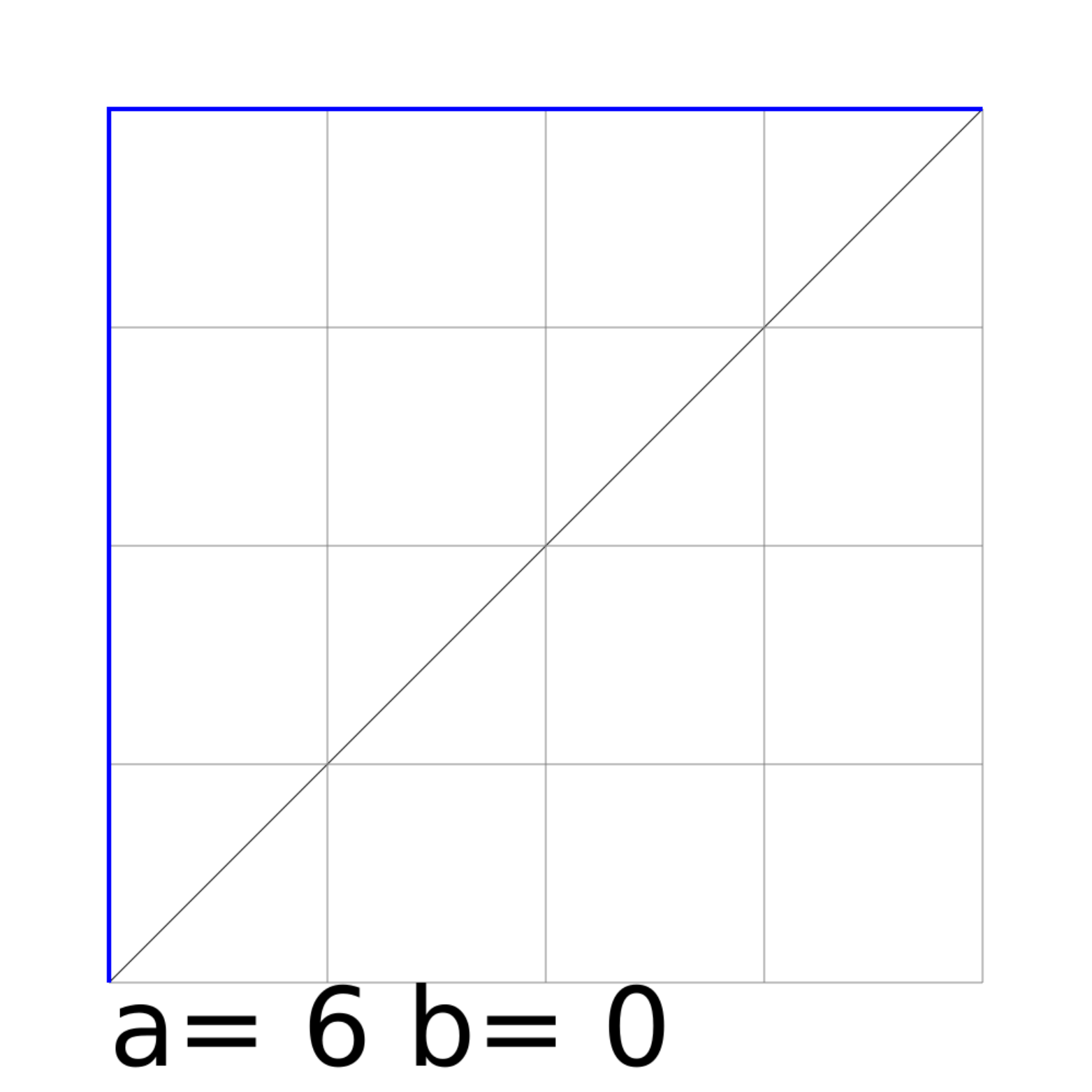}
  \end{subfigure}
  \begin{subfigure}[t]{0.13\textwidth}
    \includegraphics[width=\textwidth]{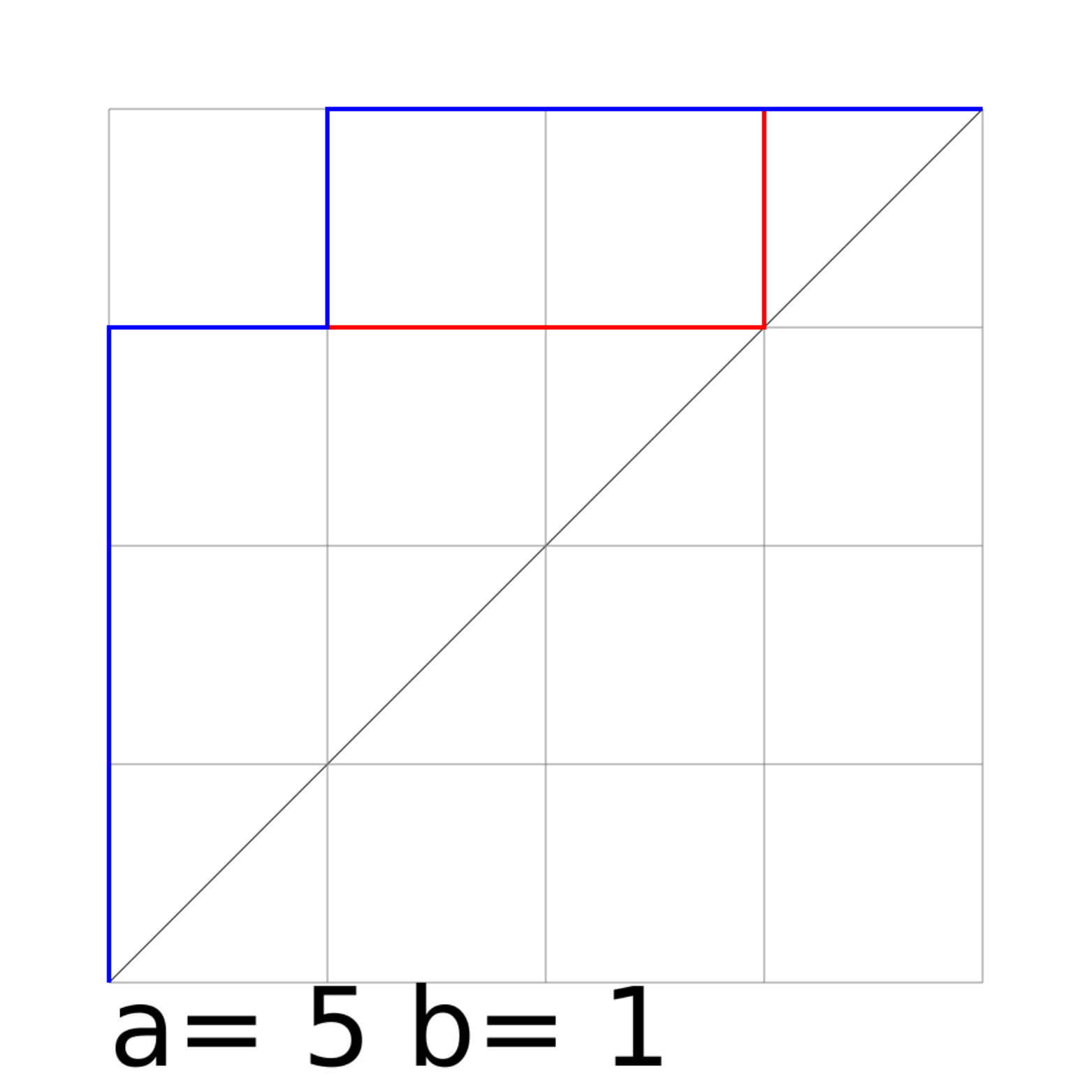}
  \end{subfigure}
  \begin{subfigure}[t]{0.13\textwidth}
    \includegraphics[width=\textwidth]{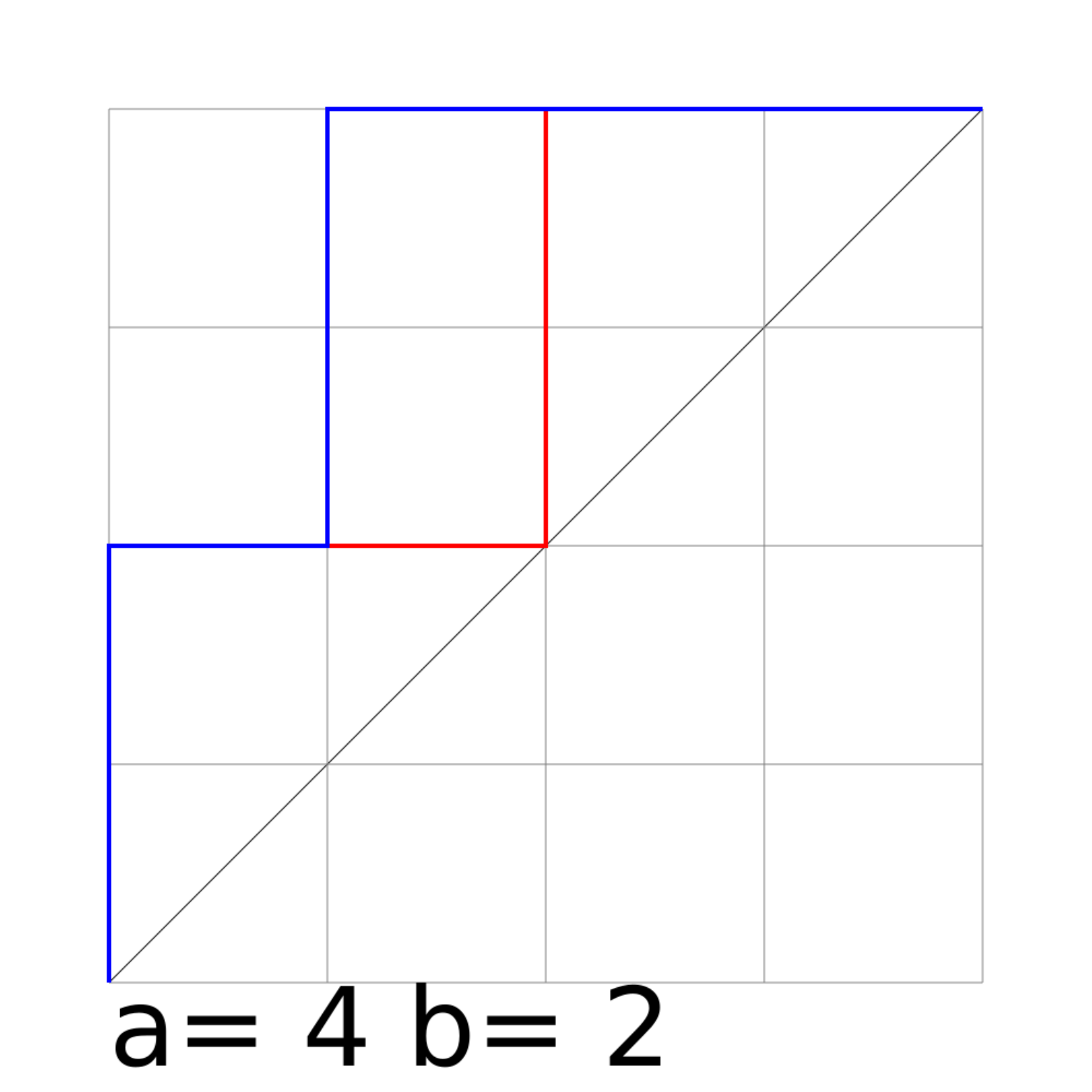}
  \end{subfigure}
  \begin{subfigure}[t]{0.13\textwidth}
    \includegraphics[width=\textwidth]{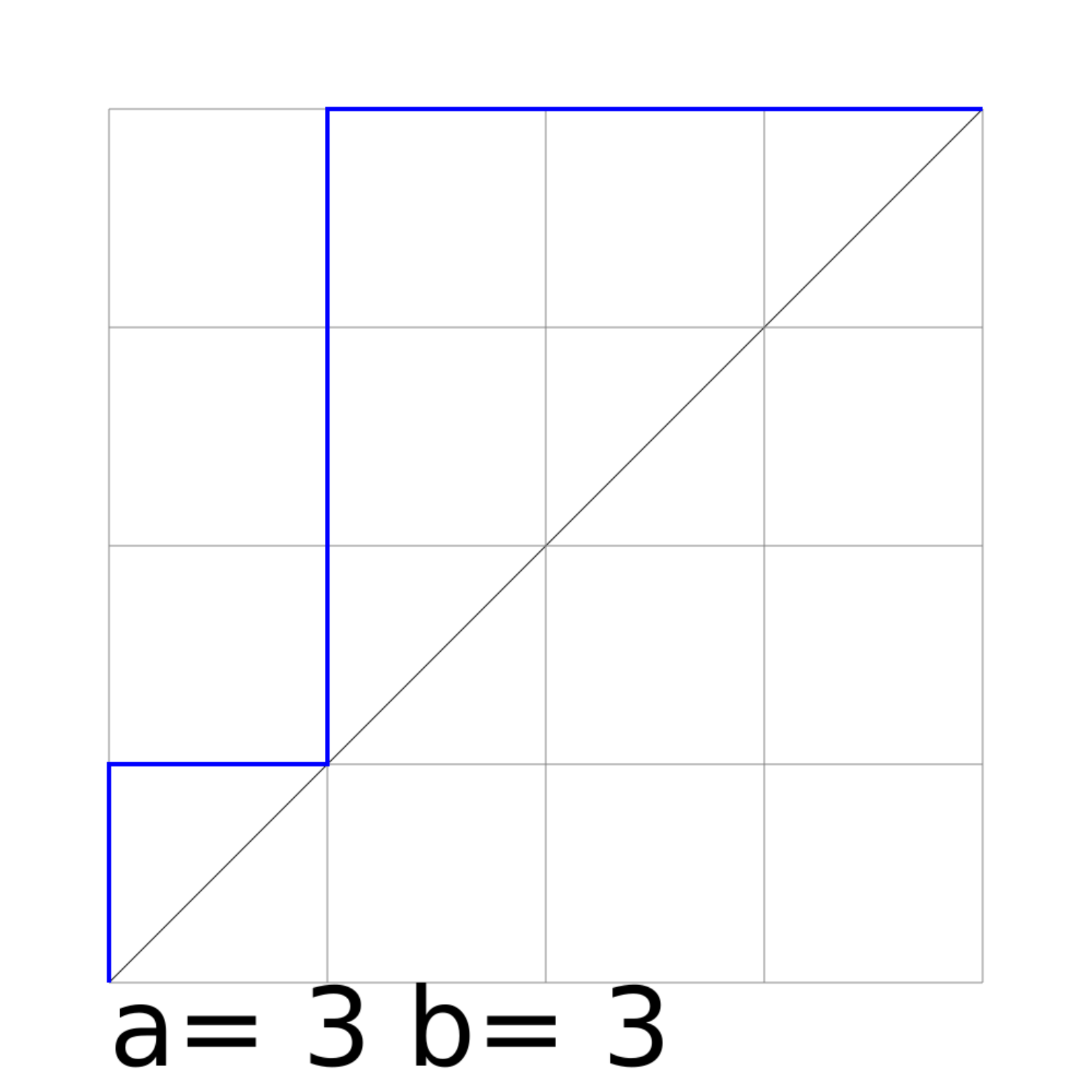}
  \end{subfigure}
  \begin{subfigure}[t]{0.13\textwidth}
    \includegraphics[width=\textwidth]{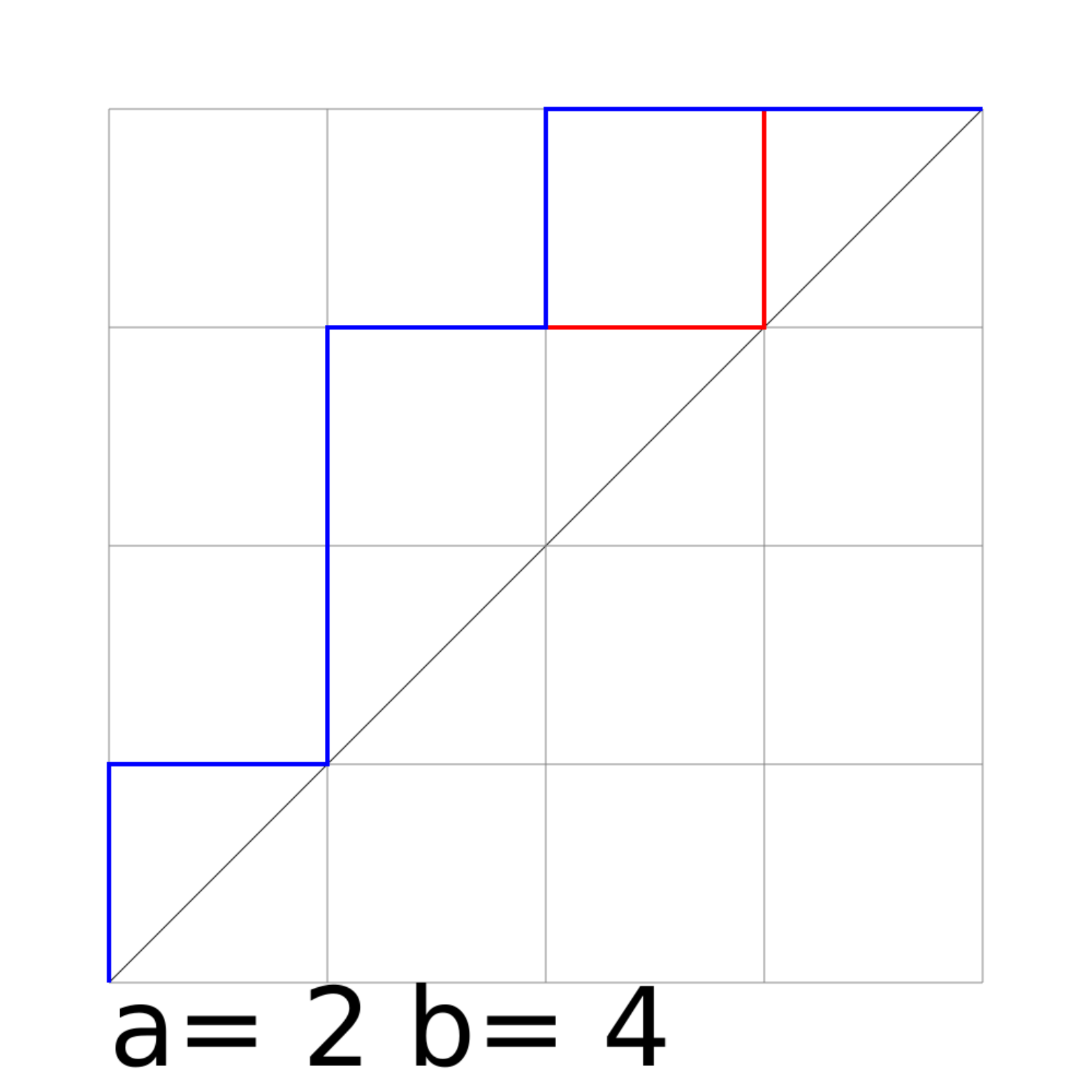}
  \end{subfigure}
  \begin{subfigure}[t]{0.13\textwidth}
    \includegraphics[width=\textwidth]{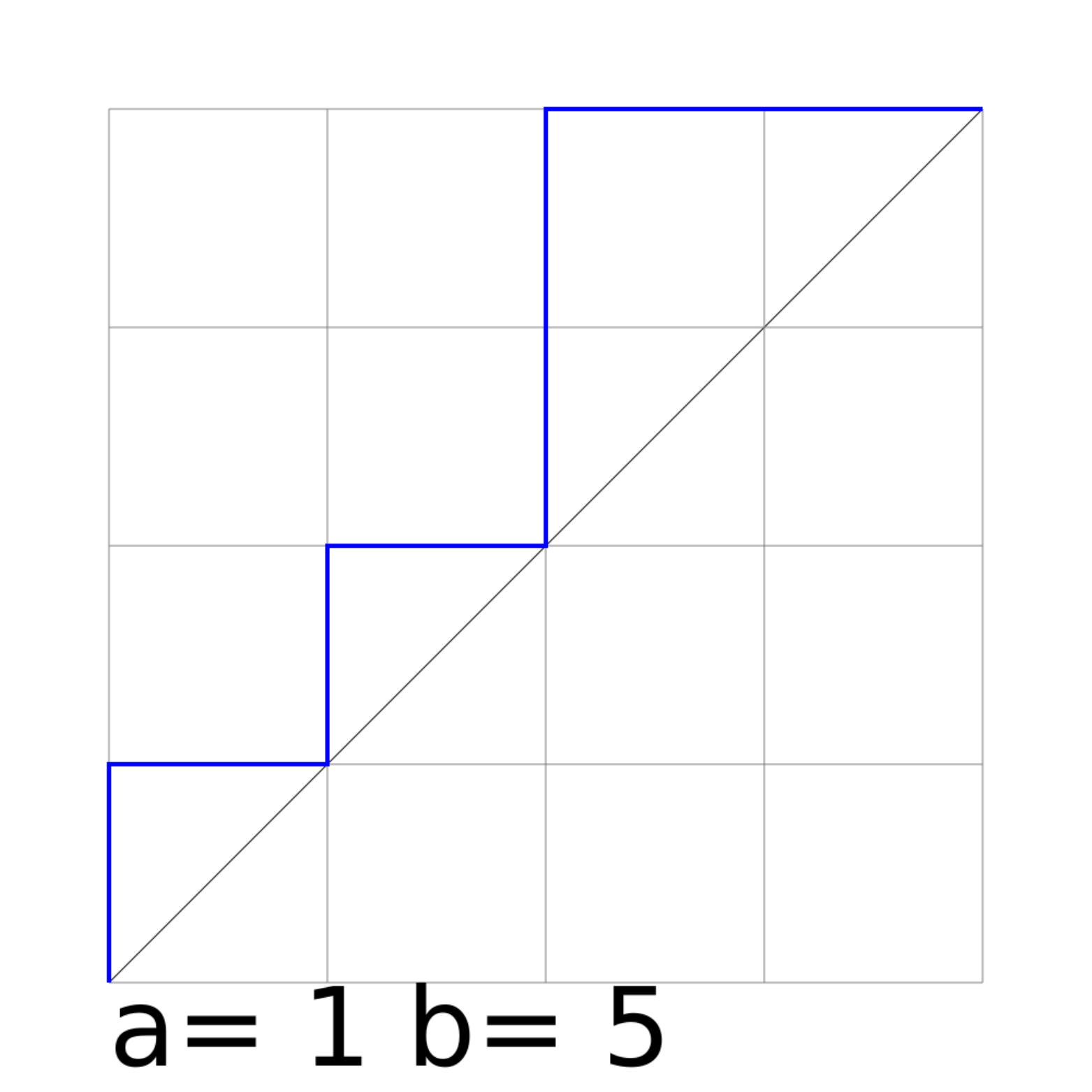}
  \end{subfigure}
  \begin{subfigure}[t]{0.13\textwidth}
    \includegraphics[width=\textwidth]{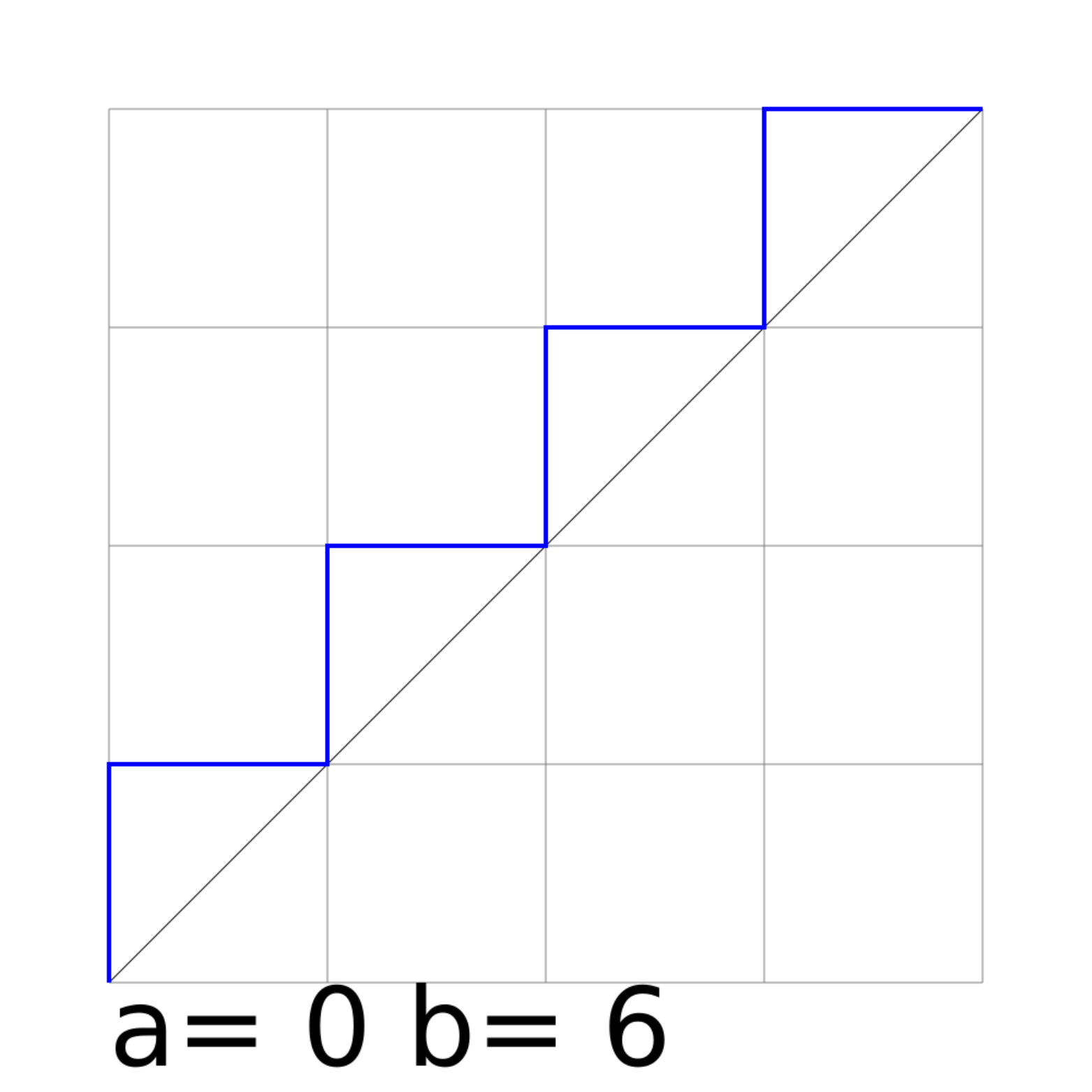}
  \end{subfigure}
  \vspace{0.5cm}

  \hspace{-2cm}
  \begin{subfigure}[t]{0.13\textwidth}
    \includegraphics[width=\textwidth]{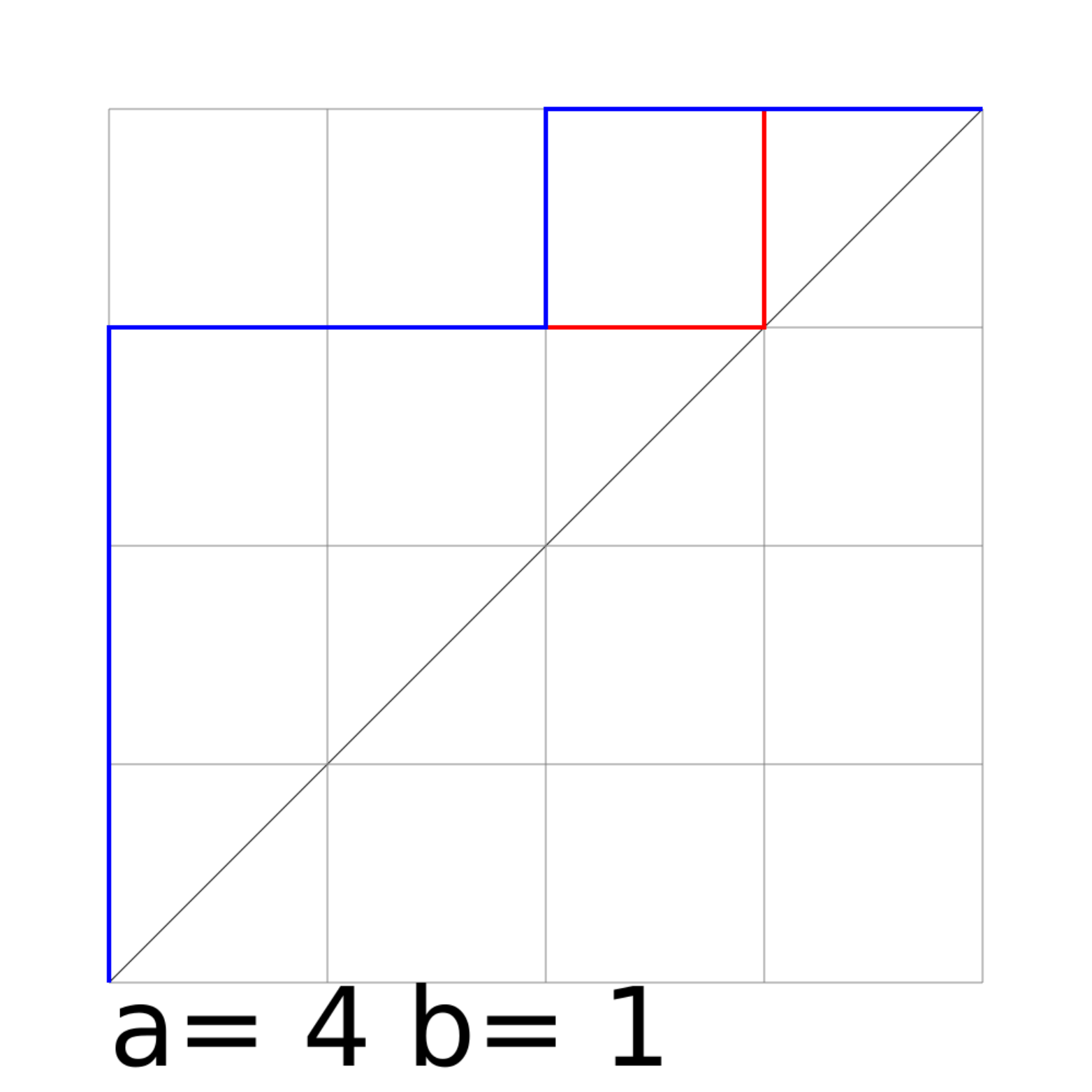}
  \end{subfigure}
  \begin{subfigure}[t]{0.13\textwidth}
    \includegraphics[width=\textwidth]{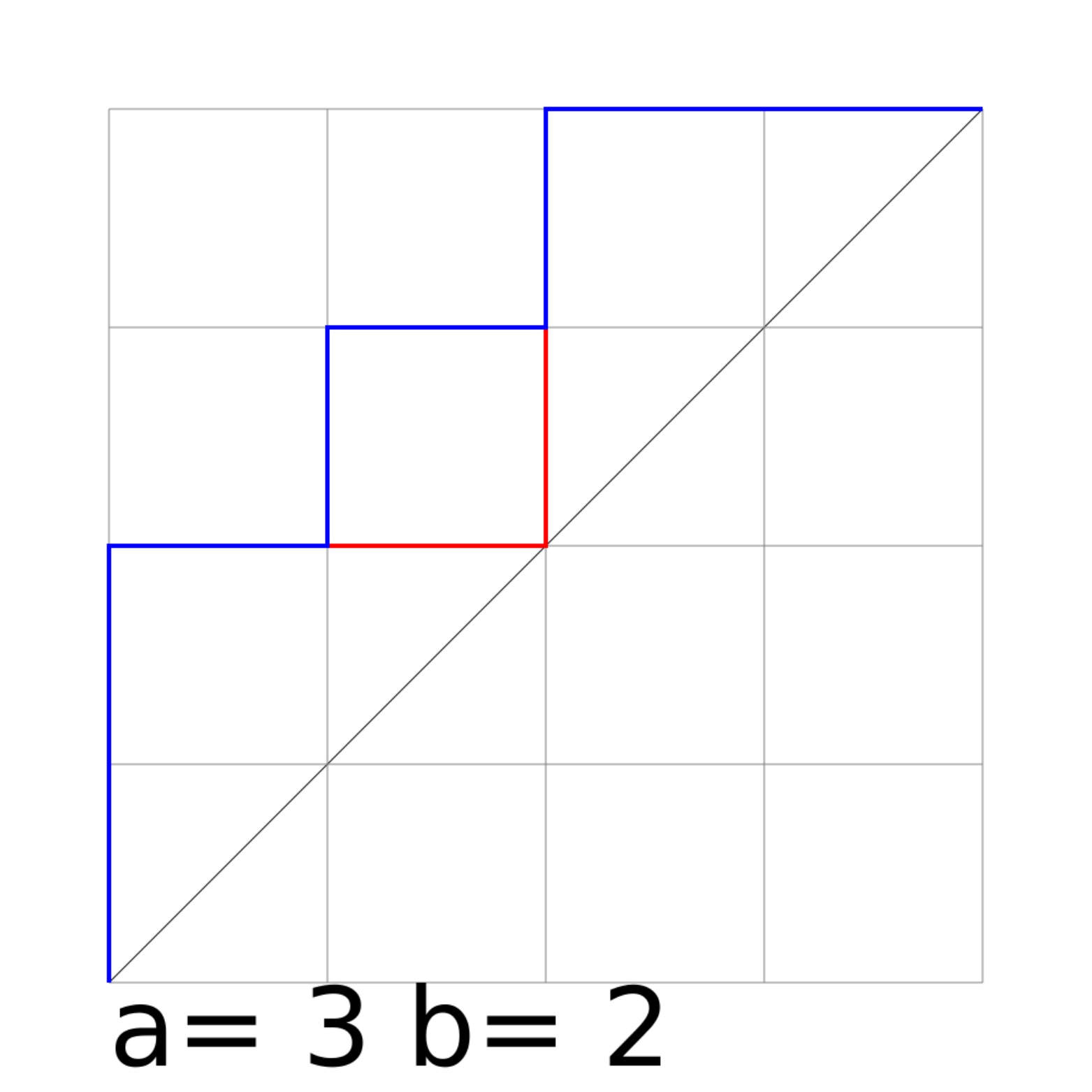}
  \end{subfigure}
  \begin{subfigure}[t]{0.13\textwidth}
    \includegraphics[width=\textwidth]{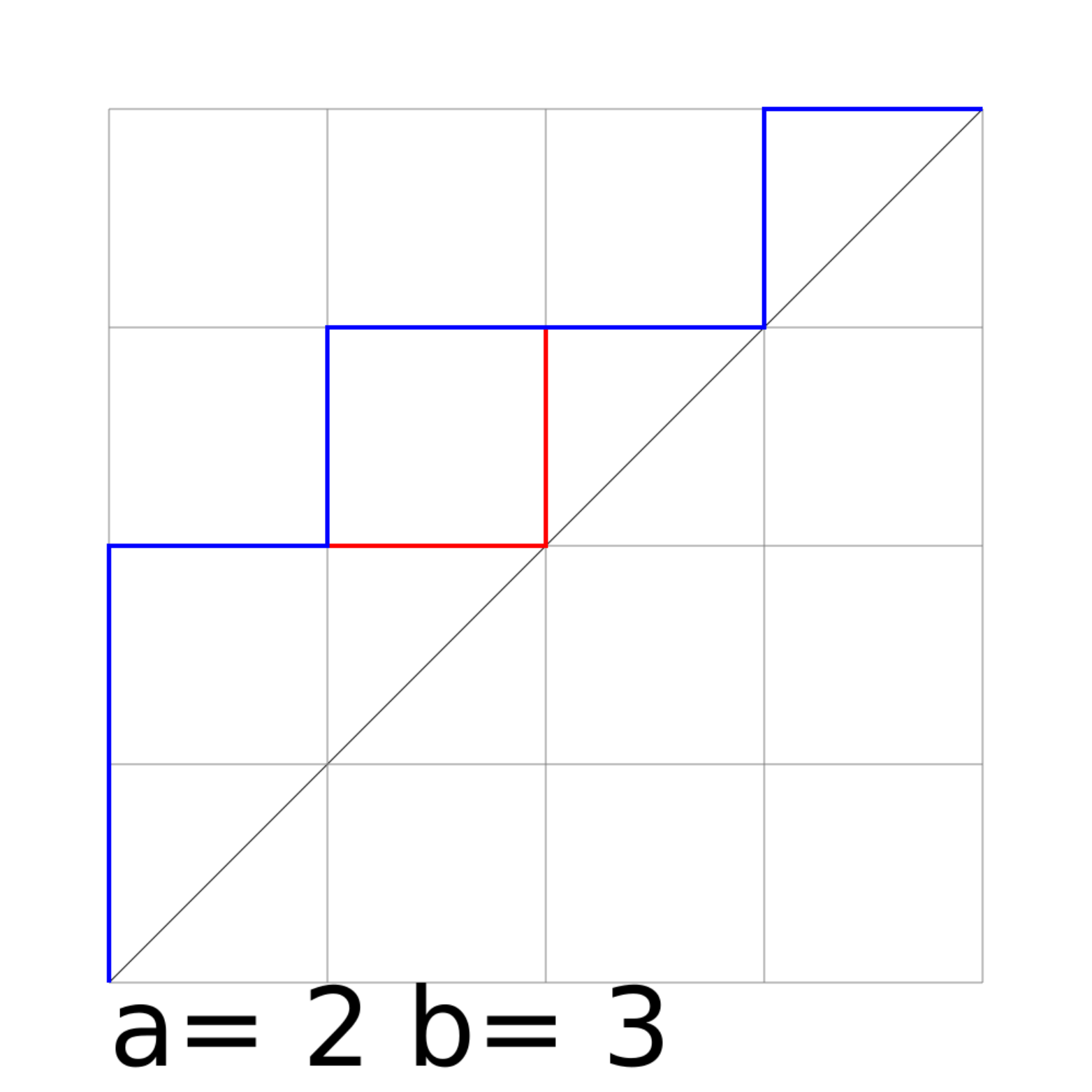}
  \end{subfigure}
  \begin{subfigure}[t]{0.13\textwidth}
    \includegraphics[width=\textwidth]{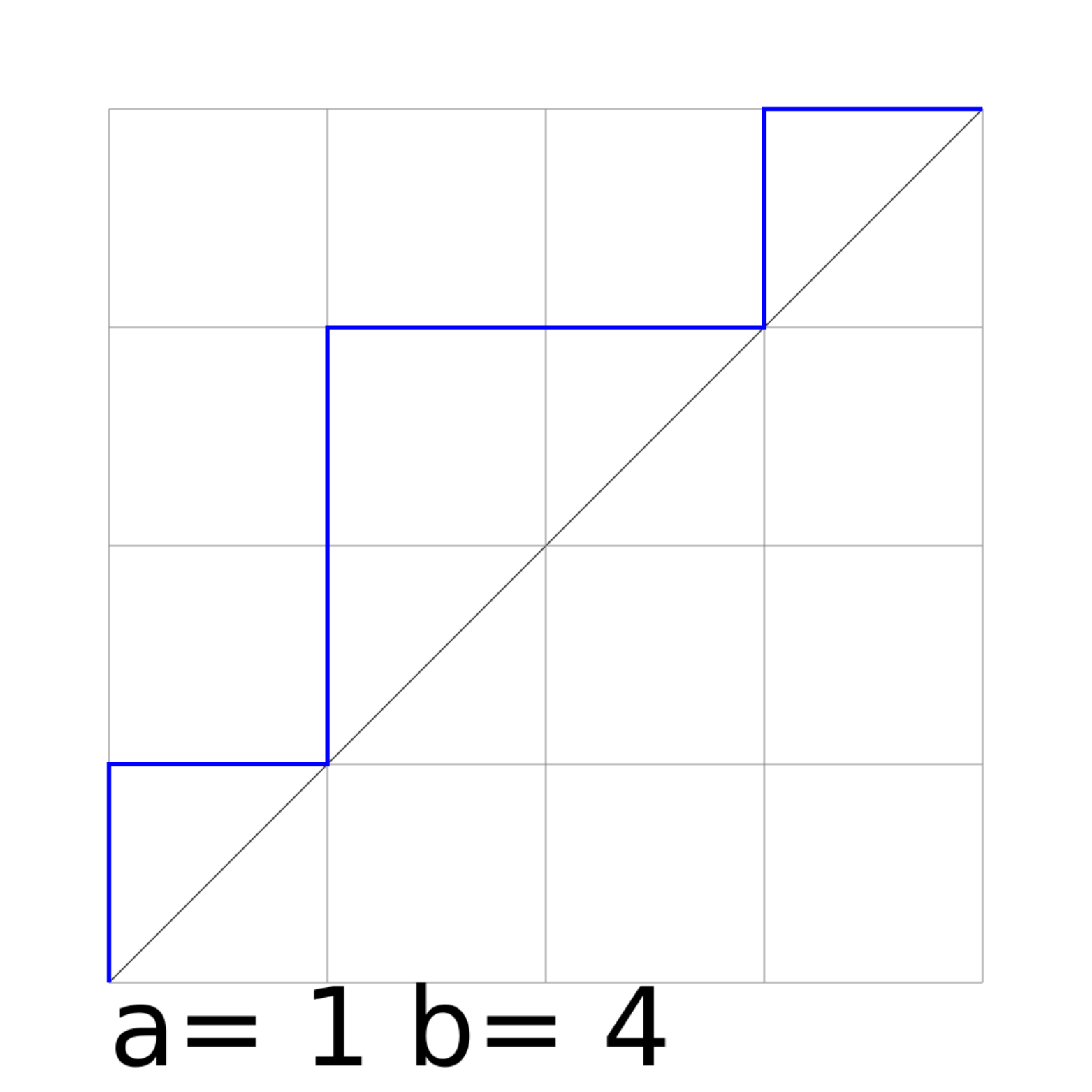}
  \end{subfigure}
  \caption{An illustration of \cref{prop:top_two_ab} for $n = 4$. 
  The top row shows paths in
    $\dycks(4)$ with $\ab = 6$ and the bottom row shows paths with $\ab = 5$.
    The paths in the bottom row are aligned with the path in the top row from
    which it is derived.}
  \label{fig:top_two_ab}
\end{figure}

\begin{figure}
  \centering
  \begin{tabular}{c | c | c}
    \includegraphics[width=0.2\textwidth]{n_4_6_7.pdf} & &\\
    \includegraphics[width=0.2\textwidth]{n_4_6_6.pdf} & 
    \includegraphics[width=0.2\textwidth]{n_4_5_4.pdf}&
    \includegraphics[width=0.2\textwidth]{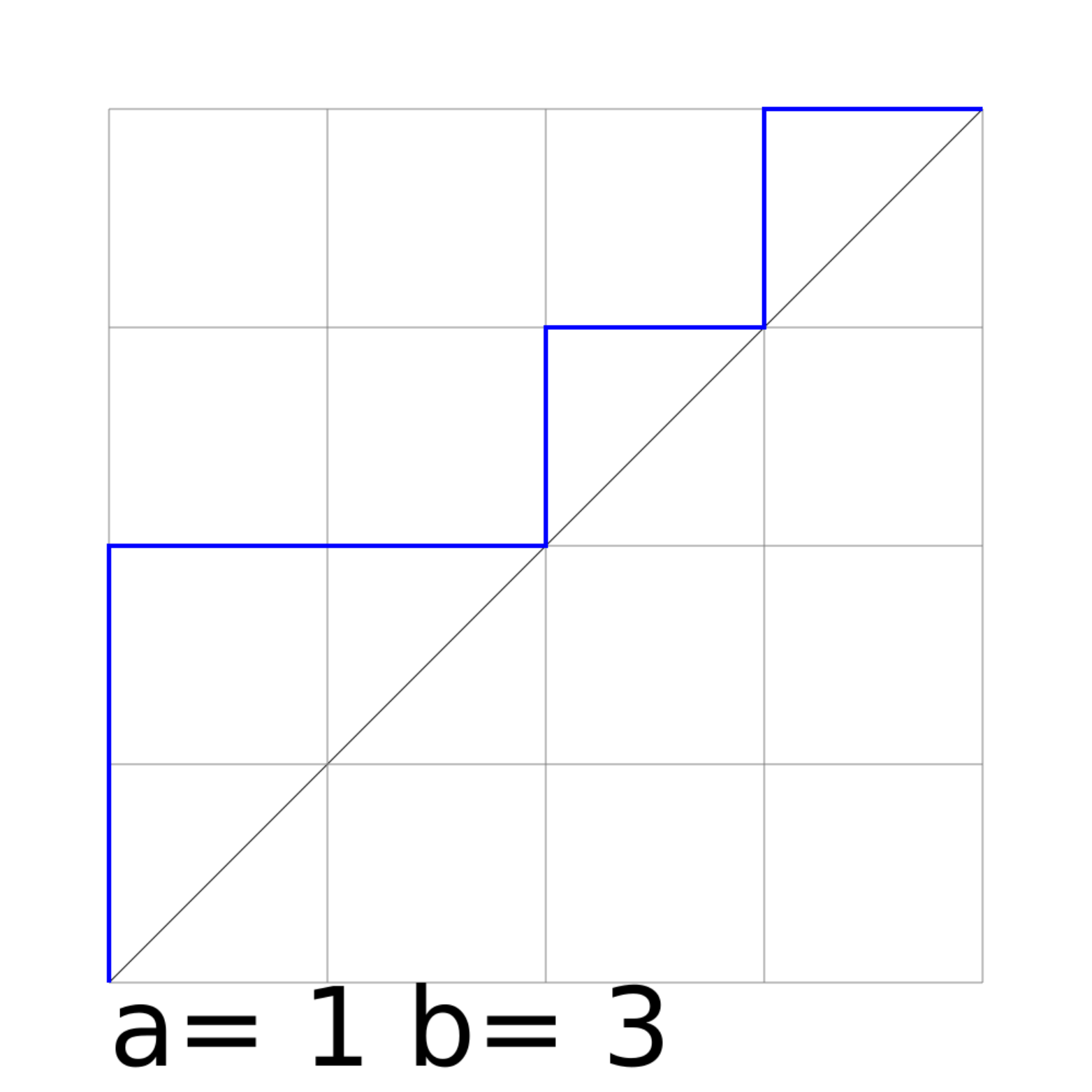}\\
    \includegraphics[width=0.2\textwidth]{n_4_6_5.pdf} &
    \includegraphics[width=0.2\textwidth]{n_4_5_3.pdf}&
    \includegraphics[width=0.2\textwidth]{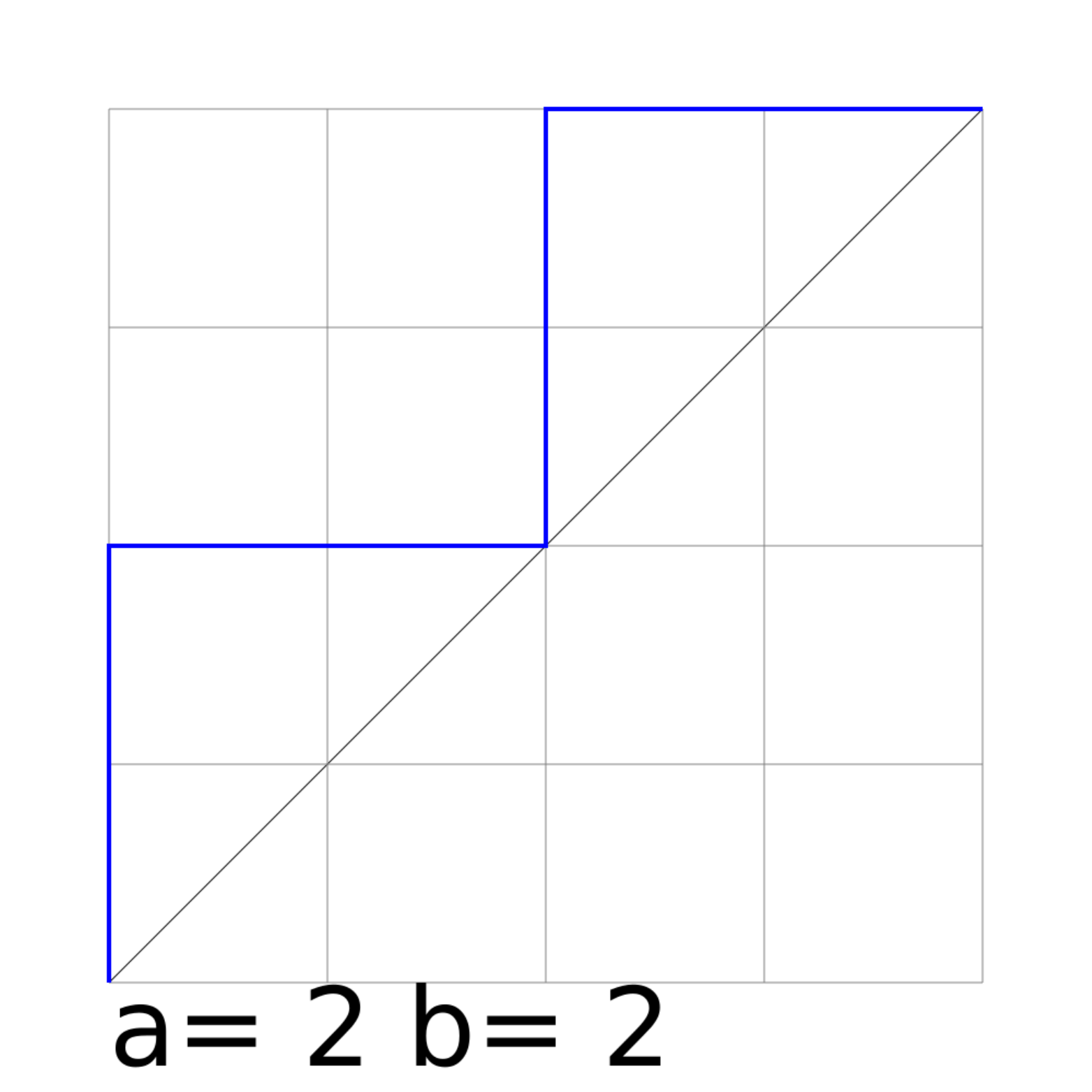}\\
    \includegraphics[width=0.2\textwidth]{n_4_6_4.pdf} &
    \includegraphics[width=0.2\textwidth]{n_4_5_2.pdf}&
    \includegraphics[width=0.2\textwidth]{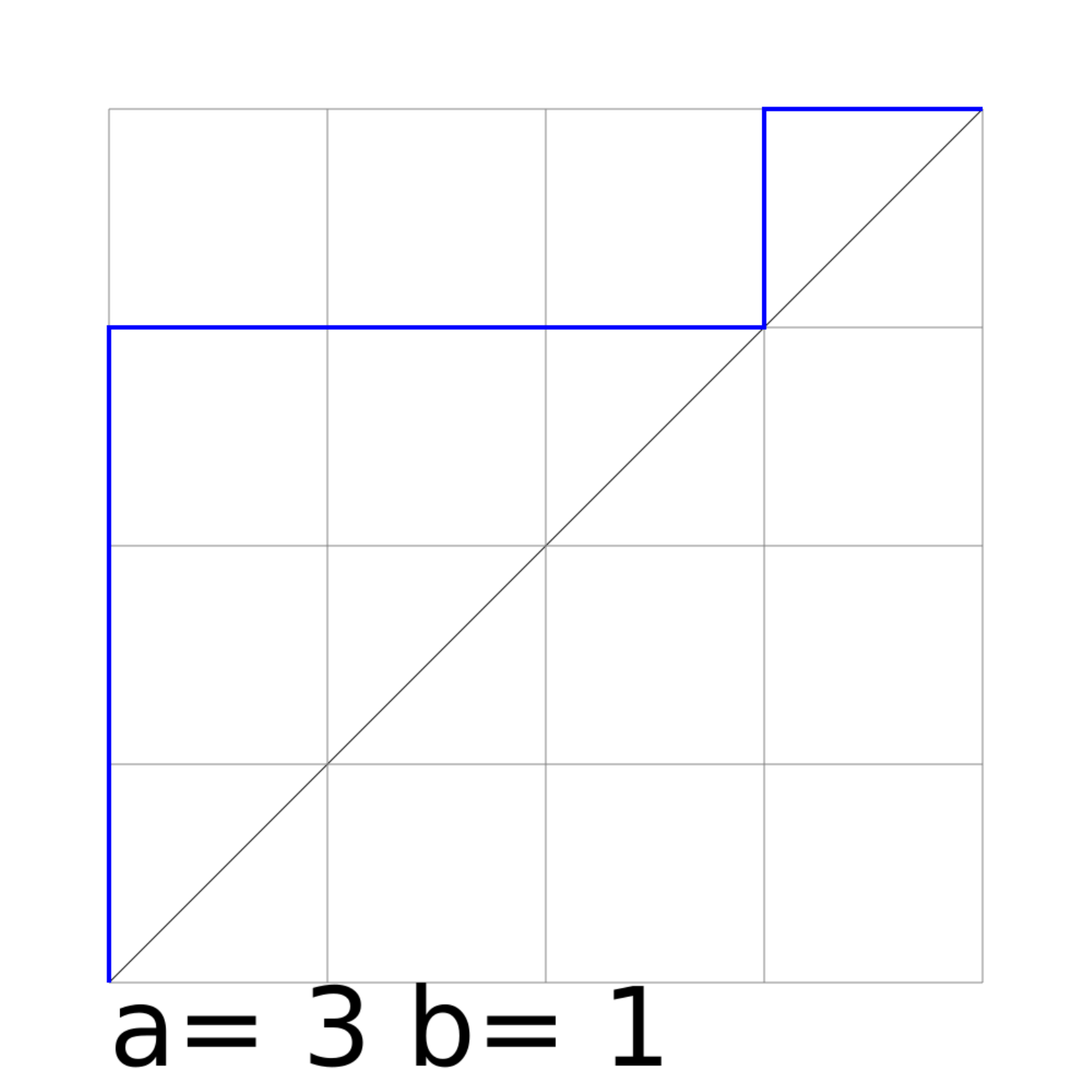} \\
    \includegraphics[width=0.2\textwidth]{n_4_6_3.pdf} &
    \includegraphics[width=0.2\textwidth]{n_4_5_1.pdf}& \\
    \includegraphics[width=0.2\textwidth]{n_4_6_2.pdf} & & \\
    \includegraphics[width=0.2\textwidth]{n_4_6_1.pdf} & &
  \end{tabular}
  \caption{All paths in $\dycks(4)$. Each column contains paths with constant area plus bounce, and each row contains paths with fixed area. 
}
  \label{fig:all_paths_4}
\end{figure}

\section{Number of distinct sums of area and bounce}
\label{sec:sums}

For any $\pi \in \dycks(n)$, let $\ab(\pi) = \area(\pi) + \bounce(\pi)$.
In this section, we will prove a formula for the number of distinct values that $\ab(\pi)$ can take. Surprisingly, this is related to one particular $q$-generalization of the Bell numbers in a tangential way. 

Recall that the Bell number $B_n$ counts set partitions of the set $[n]$. A standard recursion~\cite[Equation (1.6.13)]{wilf-1994} for the Bell numbers is
\[
B_n = \sum_{k = 0}^{n-1} \binom{n-1}{k} B_k
\]
for $n \geq 1$, with $B_0 = 1$.
One natural $q$-generalization of the Bell number due to Johnson~\cite[Equation (3.2)]{johnson-1996} is
\begin{equation}
\label{qbell}
  B_n(q) = \sum_{k=0}^{n-1} \qbinom{n-1}{k} B_k(q)
\end{equation}
for $n \geq 1$, with $B_0(q) = 1$, where $\qbinom nk$ is the \emph{$q$-binomial coefficient} or \emph{Gaussian polynomial}.
Johnson~\cite[Equation (3.1)]{johnson-1996} (see also \cite{deodhar-srinivasan-2003}) showed that $B_n(q)$ is the weight generating function of the number of inversions of a set partition, where he defined inversions as follows. Write a set partition $\pi = \pi_1 | \cdots | \pi_k$, where the parts are ordered in increasing order of the largest part. Then the inversions of $\pi$ are the pairs $(i, j)$ where $i > j$, $i \in \pi_a, j \in \pi_b$ and $a < b$. For example, $2|467|138|59$  has {10} inversions.
See \cite[Sequence $A125810$]{OEIS} for the triangle of coefficients listing the number of set partitions of $[n]$ with $k$ inversions.
We are interested in the number of nonzero coefficients in $B_n(q)$.
  Let $g: \mathbb{N} \to \mathbb{N}$ be defined by $g(0) = 0$ and
\[
    g(n) = \max \{(k-1)(n-k) + g(n-k) \mid 1 \le k \le n\}, \quad n \geq 1.
\]

\begin{thm}
\label{thm:distinct_ab}
The sequence of sizes $d(n) = |\{ \ab(\pi) | \pi \in \dycks(n) \}|$, i.e. the number of terms in $F_n(q,q)$, for $n \ge 0$, is the number of nonzero coefficients in $B_n(q)$.
\end{thm}

The first twenty terms of this sequence~\cite[Sequence $A125811$]{OEIS} are
\begin{equation}
  1, 1, 1, 2, 3, 5, 8, 11, 15, 20, 26, 32, 39, 47, 56, 66, 76, 87, 99, 112.
\end{equation}

\begin{lem}
  \label{lem:smallest_ab}
The maximal value of $\ab(\pi)$ for $\pi \in \dycks(n)$ is $\binom{n}{2}$.
  The minimal value of $\ab(\pi)$ over all $\pi \in \dycks(n)$ is $\binom{n}{2} - g(n)$.
\end{lem}

\begin{proof}
  We first show that the maximal value of $\ab(\pi)$ is $\binom{n}{2}$. For a bounce point at height $j$ which contributes $n-j$ to $\bounce(\pi)$, $\pi$ must have some at least one $\E$ step at row $j$. Therefore, it does not include the $n-j$ cells above it at that column. Thus, comparing with $p_{n, (n)}$, we see that $\bounce(\pi) =  \sum_i (n- b_i)$ and $\area(\pi) \leq \binom{n}{2} - \sum_i (n - b_i)$. Summing both of these, we get that $\ab(\pi) \leq \binom{n}{2}$.

  Let $\pi \in \dycks(n)$ have the minimal value of $\ab(\pi)$ with the smallest possible area.
  Then, $\pi$ must be equal to some $p_{n,\alpha}$; otherwise, we can remove a floating cell to decrease only the area.
  The path $\pi$ must maximize
  $\binom{n}{2} - \ab(\pi)$, which is the difference of the
  sum of the heights of the non-bounce points and its area.
  This difference, say $G_n(\alpha)$, is
  \begin{align*}
    G_n(\alpha) = &\sum_{i=1}^{\ell(\alpha)} \left( \sum_{j=1}^{\alpha_i-1} (n - b_i + j) - \binom{\alpha_i}{2} \right)\\
    = &\sum_{i=1}^{\ell(\alpha)} (\alpha_i-1)(n-b_i) \\
    = &\sum_{i=1}^{\ell(\alpha)} (\alpha_i-1)(n-\alpha_1 - \cdots - \alpha_i)\\
    = &(\alpha_1-1)(n-\alpha_1) + G_{n-\alpha_1}((\alpha_2, \alpha_3, \dots)).
  \end{align*}
  Thus, $g(n) = \underset{\beta \in \textbf{Comp}(n)}{\max}\ G_n(\beta)$,
  $\alpha = \underset{\beta \in \textbf{Comp}(n)}{\arg \max}\ G_n(\beta)$,
  and the path with the minimal $\ab$ is $p_{n,\alpha}$.
\end{proof}

\begin{lem}
  \label{lem:continuously_increase_ab}
  Let $\pi \in \dycks(n)$. Then for every $x$ such that
  $\ab(\pi) \le x \le \binom{n}{2}$ there exists a $\tau \in \dycks(n)$ with $\ab(\tau) = x$.
\end{lem}

\begin{proof}
  Create a sequence of Dyck paths ending with the unique Dyck path having area
  $\binom{n}{2}$ by adding one cell at a time to the rightmost column. Each cell
  addition causes an increase in just the area except when a cell is added to
  the column $b_i+1$ in which case the area increases by one and the bounce
  decreases by one. The latter follows from \cref{rem:simple_D} (1).
  \cref{fig:continuously_increasing_ab} shows an example of this process.
\end{proof}

The list of all paths of size $4$ arranged according to their values of $\ab$ is given in \cref{fig:all_paths_4}.

\begin{figure}
  \centering
  \captionsetup[subfigure]{labelformat=empty}
  \begin{subfigure}[t]{0.15\textwidth}
    \includegraphics[width=\textwidth]{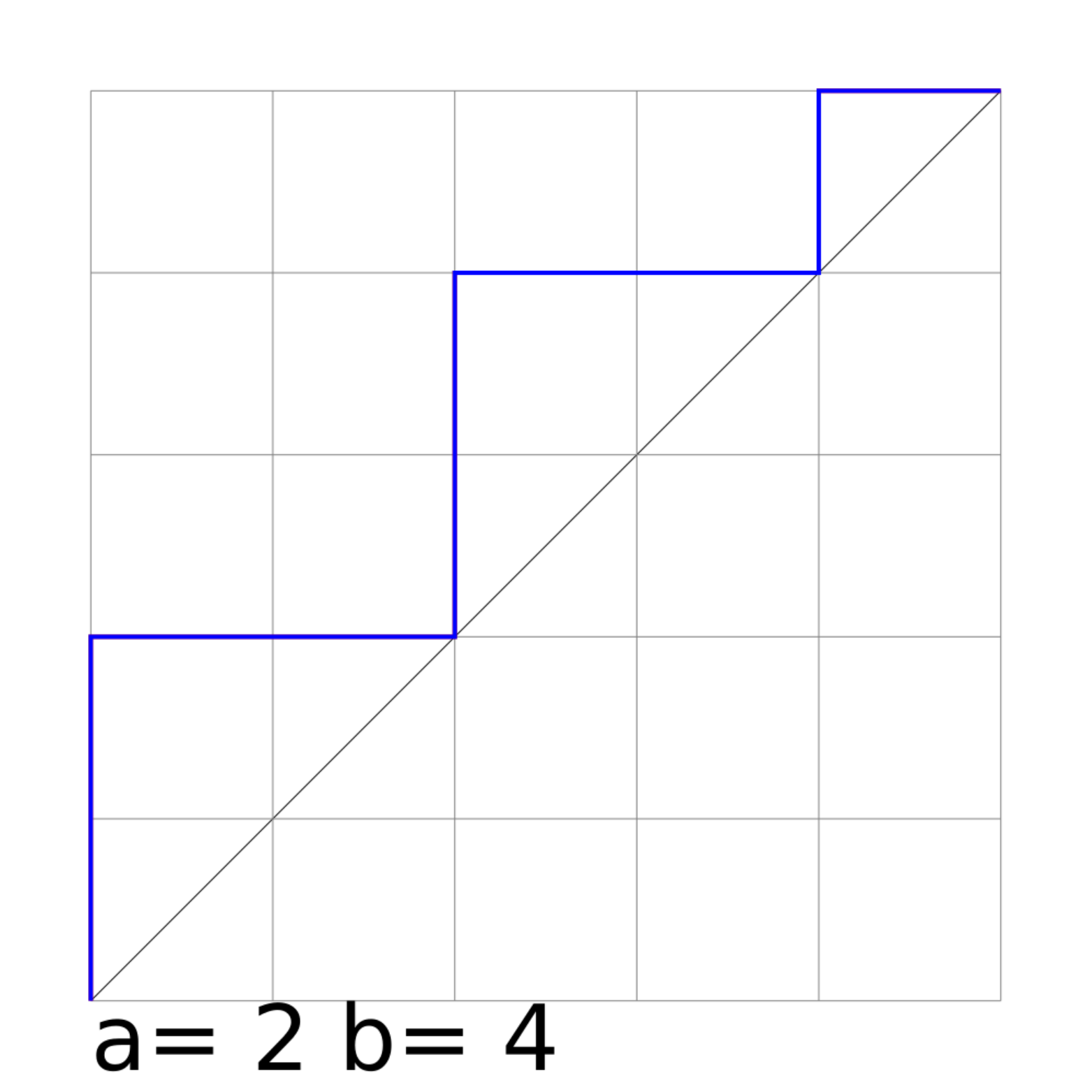}
  \end{subfigure}
  \begin{subfigure}[t]{0.15\textwidth}
    \includegraphics[width=\textwidth]{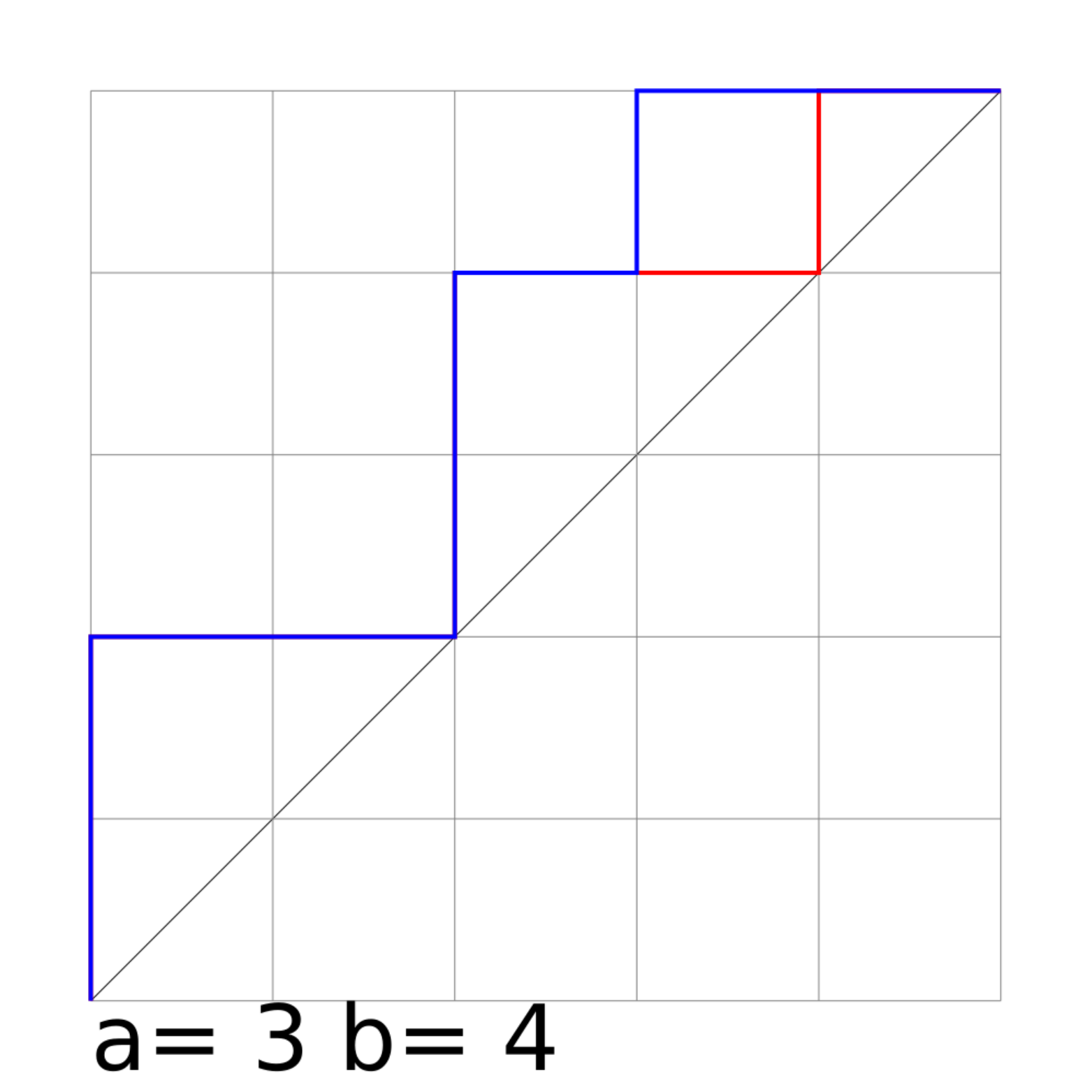}
  \end{subfigure}
  \begin{subfigure}[t]{0.15\textwidth}
    \includegraphics[width=\textwidth]{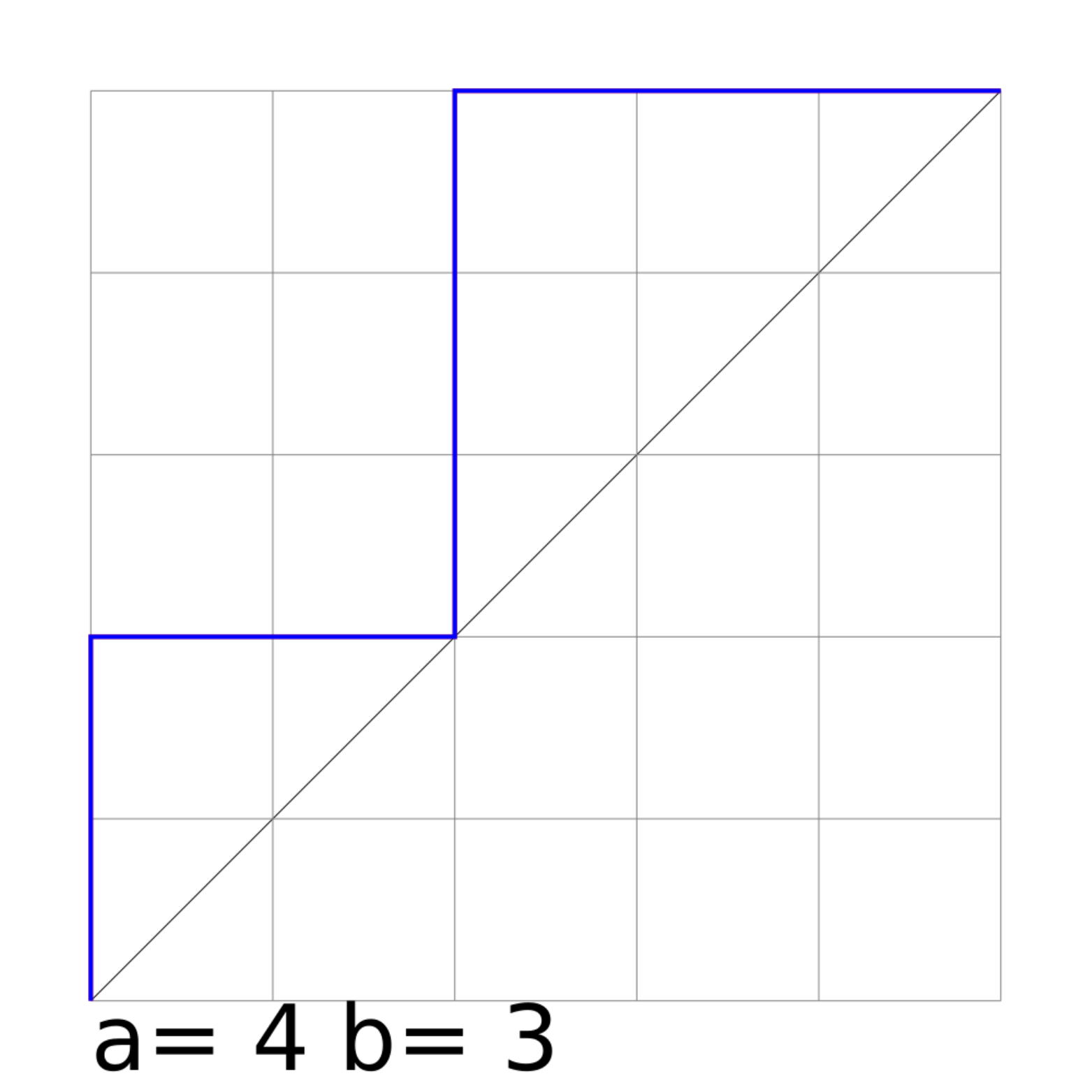}
  \end{subfigure}
  \begin{subfigure}[t]{0.15\textwidth}
    \includegraphics[width=\textwidth]{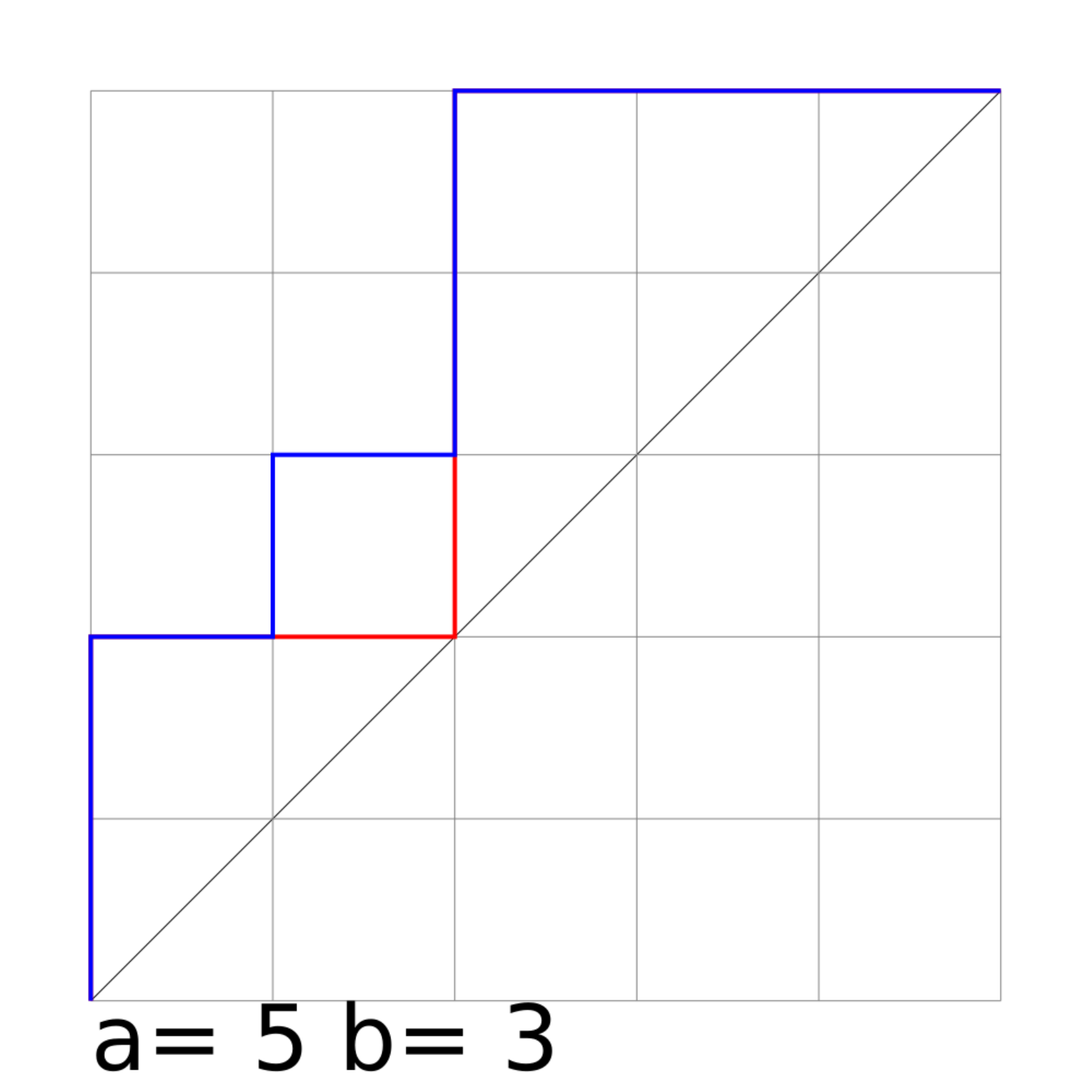}
  \end{subfigure}
  \begin{subfigure}[t]{0.15\textwidth}
    \includegraphics[width=\textwidth]{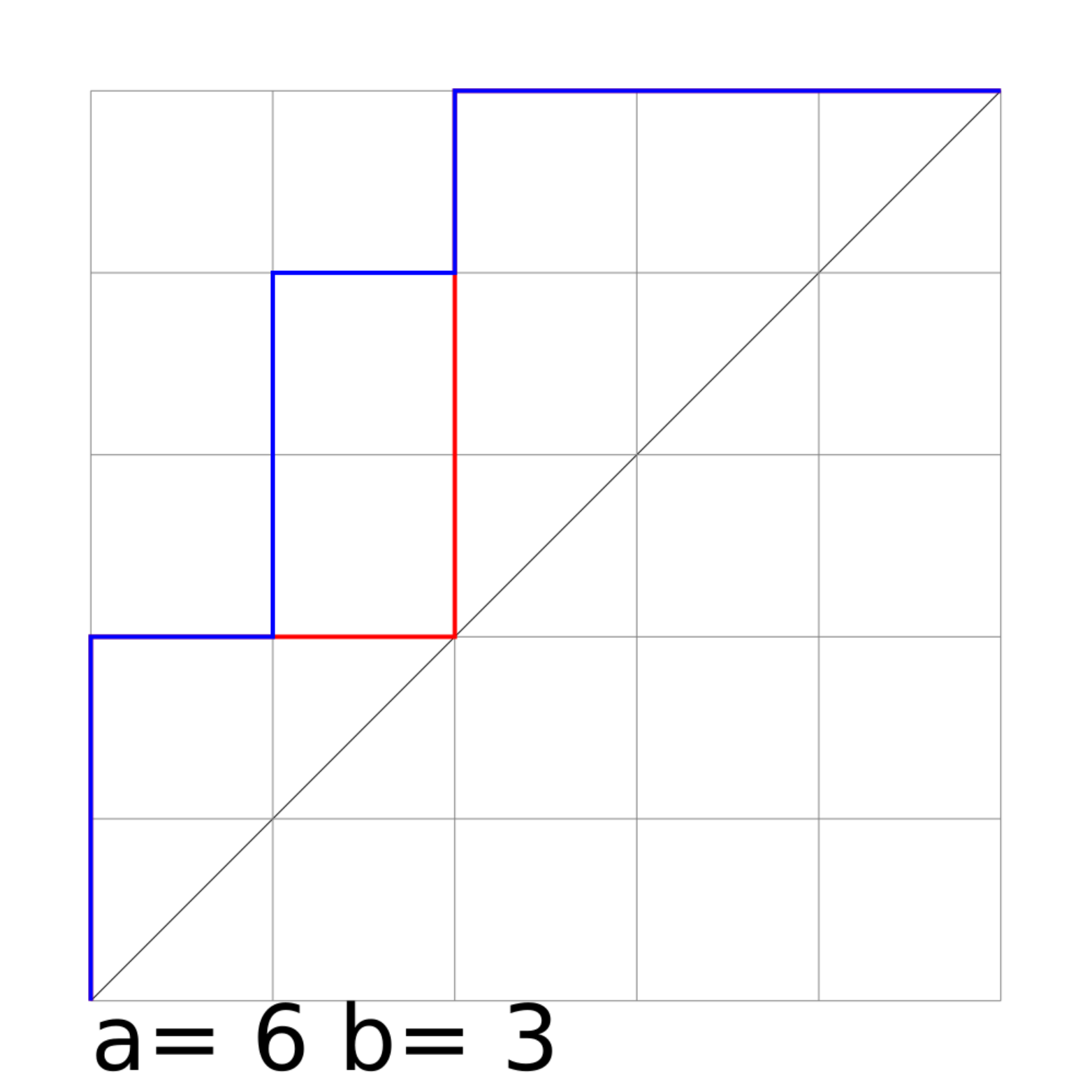}
  \end{subfigure}

  \begin{subfigure}[t]{0.15\textwidth}
    \includegraphics[width=\textwidth]{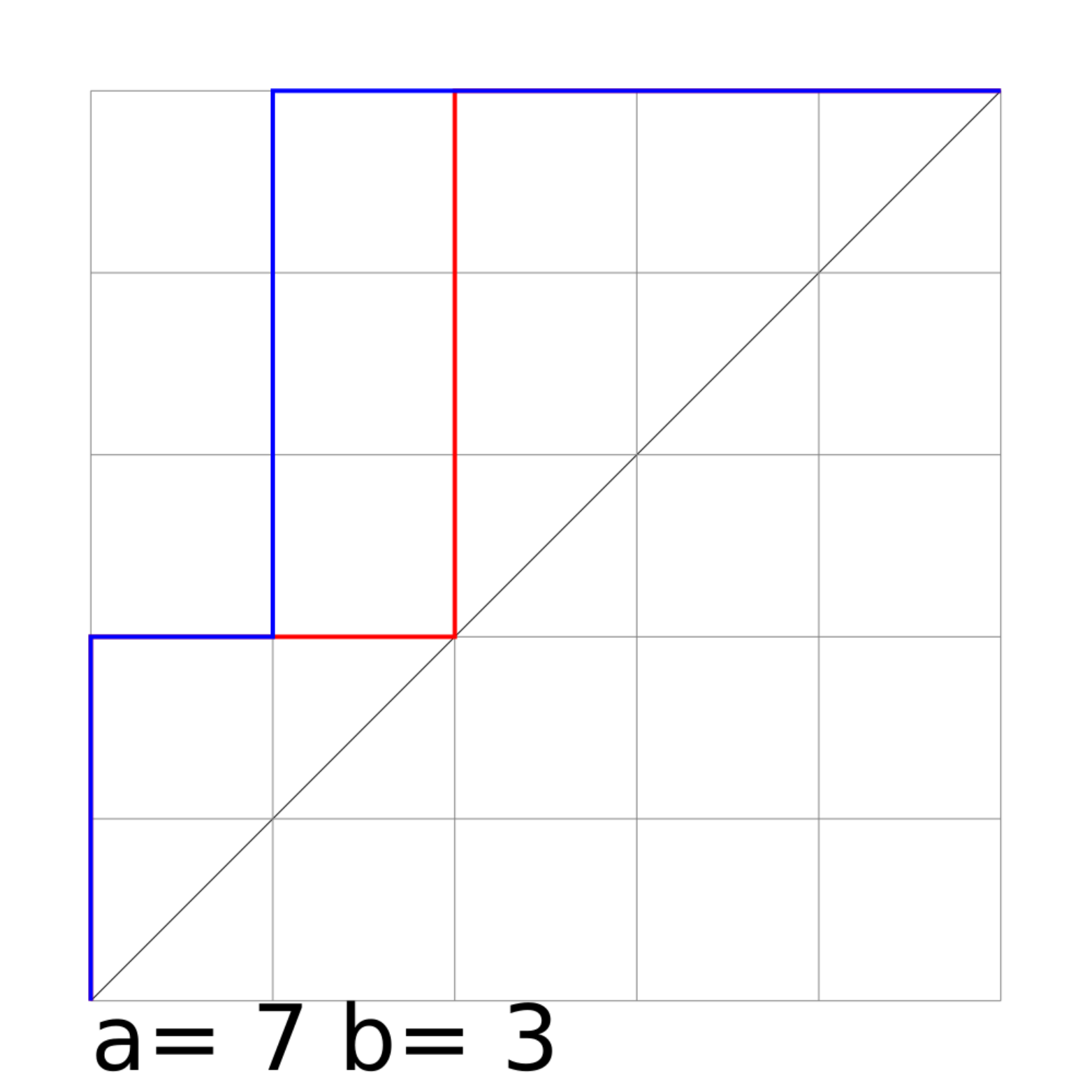}
  \end{subfigure}
  \begin{subfigure}[t]{0.15\textwidth}
    \includegraphics[width=\textwidth]{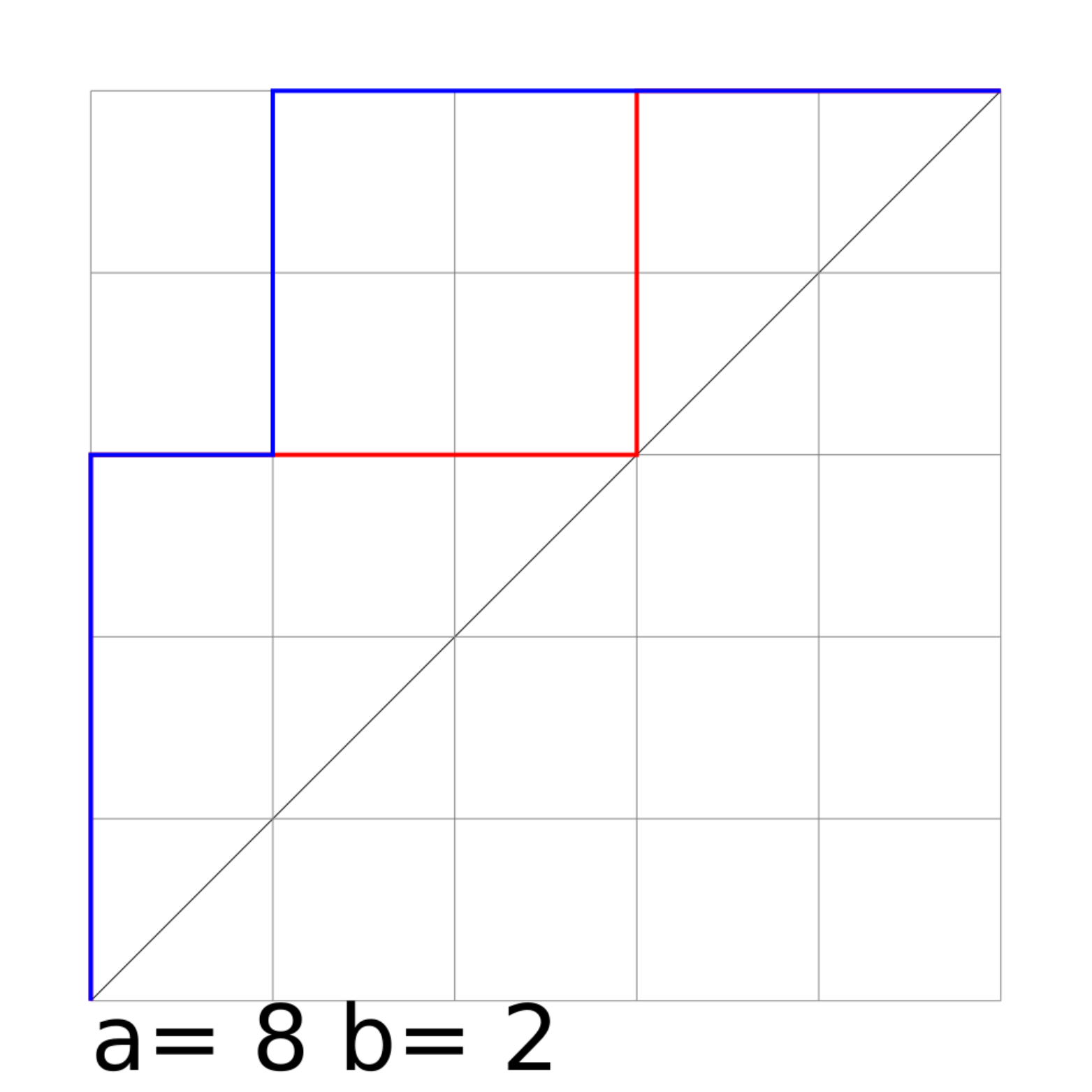}
  \end{subfigure}
  \begin{subfigure}[t]{0.15\textwidth}
    \includegraphics[width=\textwidth]{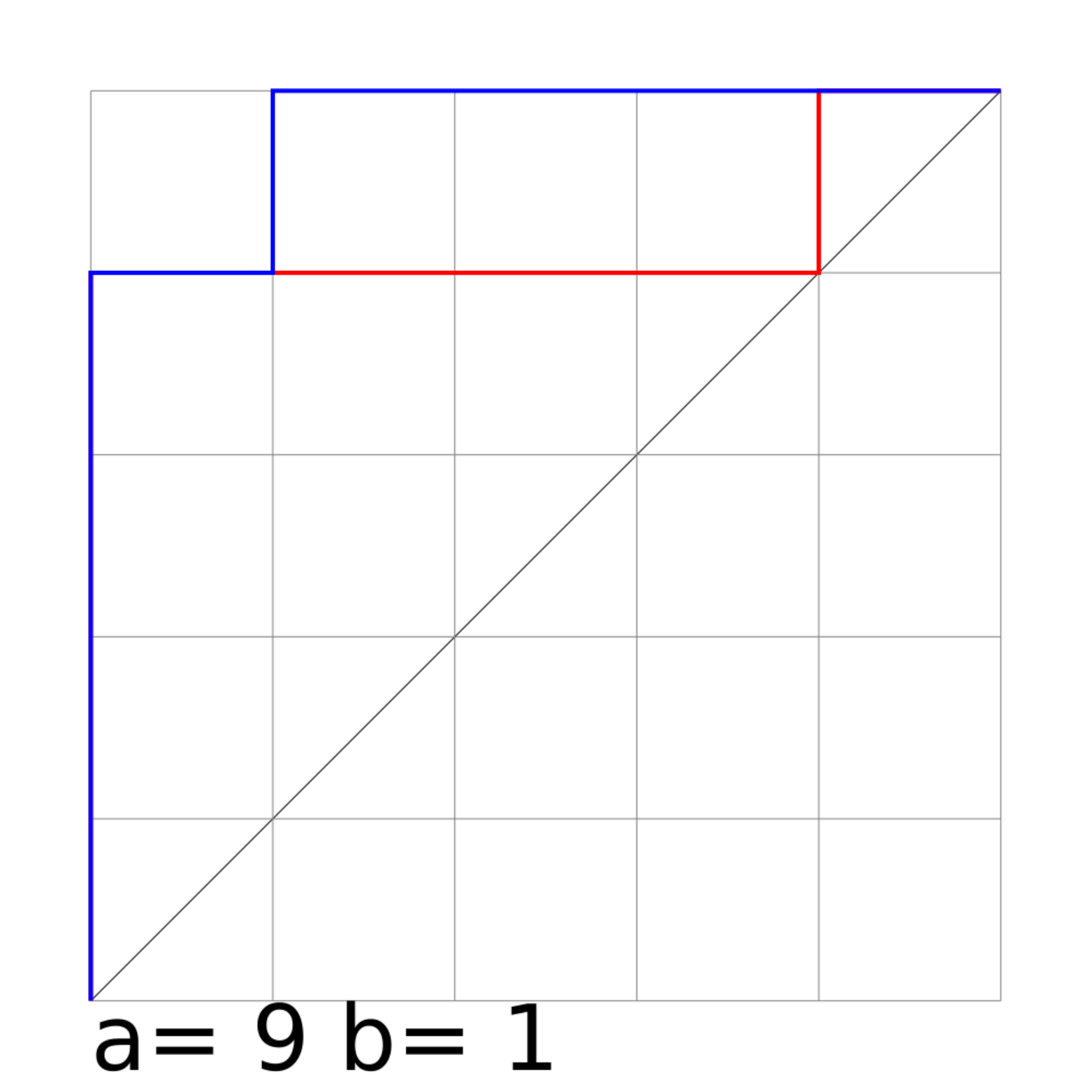}
  \end{subfigure}
  \begin{subfigure}[t]{0.15\textwidth}
    \includegraphics[width=\textwidth]{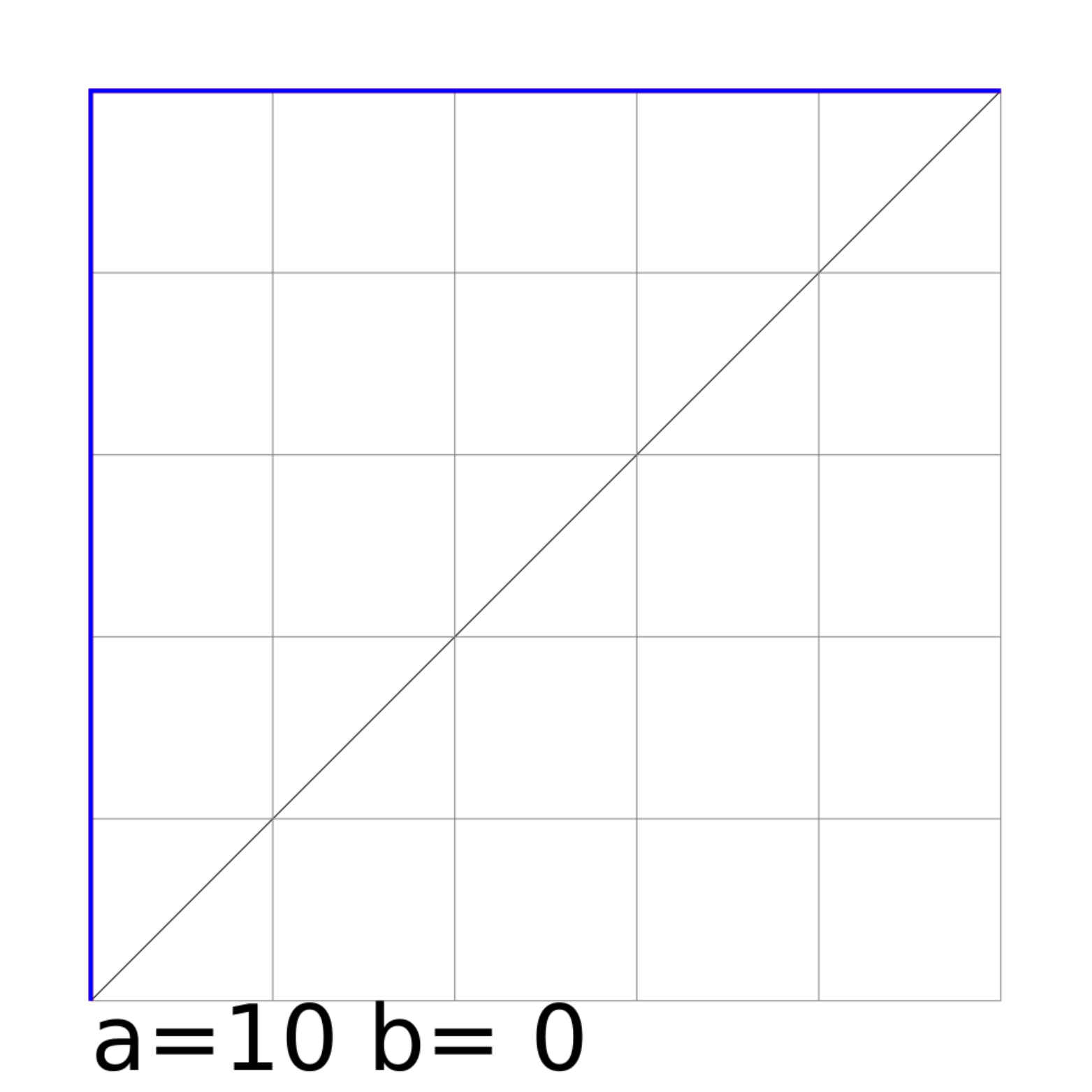}
  \end{subfigure}
  \caption{An illustration of \cref{lem:continuously_increase_ab}.}
  \label{fig:continuously_increasing_ab}
\end{figure}

\begin{lem}
  \label{lem:q_bell_numbers}
  The number of coefficients in $B_n(q)$ is given by $g(n)+1$.
\end{lem}
\begin{proof}
  Unrolling the recursion in $B_n(q)$ leads to each term being a product of
  $q$-binomials like
  \begin{equation}
    \qbinom{n-1}{n-k_1} \qbinom{n-k_1-1}{n-k_1-k_2} \dots \qbinom{n-k_1-\dots-k_{r-1}-1}{n-k_1-\dots-k_r}
  \end{equation}
  where $k_1 + \dots + k_r = n$. The degree of this term is the
  sum of the degrees of each $q$-binomial and the degree of a $q$-binomial coefficient $\qbinom{a+b}{b}$ is $ab$.
  To maximize the degree of this product, we have to choose $k_1, \dots, k_r$
  such that
  \[
    \sum_{i=1}^r (k_i-1)(n - k_1 - \cdots - k_i)
  \]
  is maximized.
  In other words, $g(n)$ maximizes the degree of this product.
  Since the $q$-binomial coefficient $\qbinom{a+b}{a}$ enumerates all up-right lattice paths (indexed by area) in a $a \times b$ grid from the bottom left to the top right point, every term with degree less than $ab$ has a non-zero coefficient. Thus, $g(n)+1$ is the number of coefficients, proving the result.
\end{proof}

\begin{proof}[Proof of \cref{thm:distinct_ab}]
  By \cref{lem:smallest_ab}, the smallest value of $\ab$ is
  $\binom{n}{2} - g(n)$. By \cref{lem:continuously_increase_ab}, there is a
  path expressing every $\ab$ between this minimum and
  $\binom{n}{2}$. Hence, there are $g(n)+1$ distinct values of $\ab$ in
  $\dycks(n)$. By \cref{lem:q_bell_numbers}, $g(n)+1$ is also
  the number of nonzero coefficients in $B_n(q)$.
\end{proof}

\section*{Acknowledgements}
We thank Digjoy Paul and Hiranya Dey for discussions and the anonymous referees for extremely useful feedback.
The first author (AA) acknowledges support from SERB Core grant CRG/ 2021/001592 as well as the DST FIST program - 2021 [TPN - 700661].

\bibliographystyle{alpha}
\bibliography{qtcatalan}

\end{document}